\documentclass[10pt,a4paper]{article}

\usepackage{lineno,hyperref}
\modulolinenumbers[5]

\usepackage[english]{babel}
\usepackage{graphicx}
\usepackage{marginnote}

\usepackage{tikz}
\usetikzlibrary{patterns}
\usepackage{bbm}
\usetikzlibrary{arrows}
\usepackage{mathtools}
 \usepackage{multirow}
\usepackage{amssymb}
\usepackage{amsfonts}
\usepackage{amsthm}
\usepackage{subfig}
\usepackage[scr]{rsfso} 
\usepackage{algorithm}     
\usepackage{algpseudocode} 

\newcommand{\email}[1]{\hspace*{\stretch{1}}\emph{\texttt{#1}}}

\makeatletter
\def\blfootnote{\xdef\@thefnmark{$\star$}\@footnotetext}
\makeatother
\newenvironment{Authors}%
  {\begin{center}\begin{bfseries}}%
  {\end{bfseries}\end{center}}
\newenvironment{Addresses}%
  {\begin{flushleft}\begin{itshape}}%
  {\end{itshape}\end{flushleft}}

 \usepackage{fancyhdr}  
  
\fancypagestyle{plain}{
	\fancyhead{}
	\fancyhead[C]{\hfill submitted to SIAM Journal of Numerical Analysis, December 2017}
}
  
\newtheorem{theorem}{Theorem}[section]
\newtheorem{theorem1}{Theorem}[section]
\newtheorem{theorem2}{Theorem}[section]
\newtheorem{proposition}[theorem]{Proposition}
\newtheorem{lemma}[theorem1]{Lemma}
\newtheorem{remark}[theorem2]{Remark}

  \newcommand{\vertiii}[1]{{\left\vert\kern-0.25ex\left\vert\kern-0.25ex\left\vert #1 
    \right\vert\kern-0.25ex\right\vert\kern-0.25ex\right\vert}}

\begin{document}

\thispagestyle{plain}

\title{Adaptive PBDW approach to state estimation: noisy observations; user-defined update spaces}
 \date{}
 
 \maketitle
\vspace{-50pt} 
 
\begin{Authors}
{Yvon Maday}$^{1,2}$,
Tommaso Taddei$^{1}$
\end{Authors}

\begin{Addresses}
$^1$   Sorbonne Universit{\'e}s,
Laboratoire Jacques-Louis Lions, France       \email{taddei@ljll.math.upmc.fr,
maday@ann.jussieu.fr} \\ 
$^{2}$ Brown University, Division of Applied Mathematics, USA
       \email{yvon{\_}jean{\_}maday@brown.edu} \\
\end{Addresses}

\begin{abstract}
 We provide a number 
of extensions and further interpretations of  the  Parameterized-Background Data-Weak (PBDW) formulation,  a real-time and in-situ Data Assimilation (DA) framework for physical systems modeled by parametrized Partial Differential Equations (PDEs),   proposed in [Y Maday, AT Patera, JD Penn, M Yano, Int J Numer Meth Eng, 102(5), 933-965].
  Given $M$ noisy measurements of the state,
  PBDW seeks an approximation of the form
  $u^{\star} = z^{\star} + \eta^{\star}$, where 
  the \emph{background} $z^{\star}$  
  belongs to a $N$-dimensional \emph{background space} informed by a parameterized mathematical model, 
and the \emph{update} $\eta^{\star}$  belongs to a $M$-dimensional \emph{update space}
informed by the experimental observations.
The contributions of the present work are threefold:
first, we extend the adaptive formulation proposed in 
[T Taddei, M2AN, 51(5), 1827-1858] to general linear observation functionals, to effectively deal with noisy observations;
second, we consider an user-defined choice of the update space,  to improve convergence with respect to the number of measurements; third, we propose an \emph{a priori} error analysis for general linear functionals in the presence of noise,
to identify the different sources of state estimation error and ultimately motivate the adaptive procedure. We present results for two  synthetic model problems  
in Acoustics, to illustrate the elements of the methodology and to prove its effectiveness.
We further present results for a synthetic problem in Fluid Mechanics to demonstrate the applicability of the approach to vector-valued fields.
\end{abstract}

\emph{Keywords:}
variational data assimilation; parametrized partial differential equations; model order reduction; design of experiment.

 \section{Introduction}
\label{sec:intro}
Data Assimilation (DA) refers to the process of integrating information coming from a (possibly parameterized) mathematical model with experimental observations, for prediction.
State estimation is a particular DA task in which the Quantity of Interest (QOI) is the state $u^{\rm true}$ of a physical system over a domain of interest $\Omega \subset \mathbb{R}^d$. 
In this work, we propose a number of extensions, and further interpretations, of  the  Parameterized-Background Data-Weak (PBDW) approach to state estimation,
first presented  in \cite{maday2015parameterized}.

As in \cite{taddei_APBDW}, we denote by $\{ y_m \}_{m=1}^M$ the set of experimental measurements, and we denote by $u^{{\rm bk}}(\mu) \in \mathcal{X}$ the solution to the  parameterized best-knowledge (bk) mathematical model for the parameter value $\mu \in \mathcal{P}^{\rm bk}$,
$G^{\rm bk,\mu}(u^{{\rm bk}}(\mu) ) = 0$.
 Here, the space $\mathcal{X}=\mathcal{X}(\Omega)$ is a suitable Hilbert space defined over $\Omega$, 
 $G^{\rm bk,\mu}(\cdot ) $ denotes 
the  parameterized bk mathematical model associated with the physical system, and
$\mathcal{P}^{\rm bk}\subset \mathbb{R}^P$ reflects the uncertainty in the value of the parameters of the model.
We here consider  measurements  of the form
$y_m = \ell_m^o(u^{{\rm true}}) + \epsilon_m$, where 
$ \ell_m^o$ is a  local average of the state about the location
$x_m^{\rm obs}$ in $\Omega$, and $ \epsilon_m $ reflects the observational noise.
On the other hand, the uncertainty in the parameters of the model leads to the definition of the  bk  manifold $\mathcal{M}^{{\rm bk}} := \{u^{{\rm bk}}(\mu)|_{\Omega}: \, \mu \in \mathcal{P}^{\rm bk} \} \subset \mathcal{X}$ which collects the solution to the parameterized bk model for all values of the parameter in $\mathcal{P}^{\rm bk}$, restricted to the domain of interest $\Omega$.
Since in practice the model is affected by non-parameterized (unanticipated) uncertainty, 
the true field does not in general belong to $\mathcal{M}^{{\rm bk}}$.

The key idea of  the PBDW formulation is to  seek an approximation $u^{\star}= z^{\star} + \eta^{\star}$ to the true field $u^{{\rm true}} $ employing projection-by-data. The first contribution to $u^{\star}$, $z^{\star} \in \mathcal{Z}_N$, is the \emph{deduced background estimate}.
The linear $N$-dimensional space $\mathcal{Z}_N \subset \mathcal{X}$ is  informed by the bk manifold $\mathcal{M}^{{\rm bk}}$:
it thus encodes --- in a mathematically-convenient way --- the available prior knowledge about the system coming from the model.
The second contribution to  $u^{\star}$, $\eta^{\star} \in \mathcal{U}_M$, is the \emph{update estimate}.
The linear $M$-dimensional space $\mathcal{U}_M$ is the span of the Riesz representers of the $M$ observation functionals $\{ \ell_m^o \}_{m=1}^M$: $\mathcal{U}_M$  improves the approximation properties of the search space associated with the state estimation procedure.
From a modelling perspective, the deduced background $z^{\star}$ addresses the parameterized uncertainty in the bk   model,
while the update  addresses the non-parametric or unanticipated uncertainty.
In \cite{maday2015parameterized}, for  perfect measurements ($\epsilon_m = 0$ for $m=1,\ldots,M$),  
the pair $(z^{\star}, \eta^{\star}) \in \mathcal{Z}_N \times \mathcal{U}_M$ is computed by searching for $\eta^{\star}$ of minimum norm subject to the observation constraints $\ell_m^o(z^{\star} + \eta^{\star}) = y_m$ for $m=1,\ldots,M$.
In \cite{taddei_APBDW},
for  pointwise noisy measurements,
the pair $(z^{\star}, \eta^{\star})$ is computed by solving a suitable Tikhonov regularization of the original constrained minimization statement proposed in \cite{maday2015parameterized}.

PBDW provides some new contributions.
First, the variational formulation facilitates the construction of \emph{a priori} error estimates, which might guide the optimal choice of the experimental observations.
Second, the background space $\mathcal{Z}_N$ accommodates anticipated uncertainty associated with the parameters of the model in a computationally-convenient way;
the construction of $\mathcal{Z}_N$ based on $\mathcal{M}^{{\rm bk}}$ relies on the application of  parametric Model Order Reduction (pMOR) techniques.  
Third, unlike standard least-squares methods, PBDW provides a mechanism --- the update $\eta^{\star}$ --- to correct the deficiencies of the bk model. 
Finally,  projection-by-data, as opposed to projection-by-model, simplifies the treatment of uncertainty in boundary conditions, particularly
when the mathematical model is defined over a domain $\Omega^{\rm bk}$ that strictly contains $\Omega$ (\cite{taddei2017_extracted}). 
We remark that several of these ingredients have appeared in different contexts: we refer to 
\cite{maday2015parameterized,maday2015pbdw,taddei_APBDW}
 for a thorough overview of the links between the PBDW formulation and other DA techniques, and to 
 \cite{cacuci2013computational} for a thorough introduction to Data Assimilation.

In this work, we propose 
three  contributions to the original PBDW formulation.

\begin{enumerate}
\item
We extend the adaptive formulation proposed in 
\cite{taddei_APBDW} for pointwise measurements to general linear observation functionals.
The approach reads as a Tikhonov regularization of the  constrained minimization statement  in
\cite{maday2015parameterized}, and can also be interpreted as a convex relaxation of the Partial Spline Model (\cite[Chapter 9]{wahba1990spline}).
\item
In \cite{maday2015parameterized}, the update space $\mathcal{U}_M$ is induced by the particular inner product chosen for $\mathcal{X}$
(\emph{variational update}); in this work, we propose to first choose the space $\mathcal{U}_M$, and then recover the variational formulation through the definition of a suitable inner product (\emph{user-defined update}).
We demonstrate that our choice reduces the offline costs, and improves the approximation properties of the update space for smooth fields. We emphasize that the idea of selecting the approximation space before introducing the inner products is widely used in the kernel methods' literature
(see, e.g., \cite{wendland2004scattered})  for pointwise measurements, and has already been exploited in \cite{taddei_APBDW}.
\item
We propose an \emph{a priori} error analysis for general linear functionals in the presence of noise.
First, we present a bound for the deterministic error between $u^{\rm true}$ and the PBDW solution $u^{\rm opt}$, fed with perfect measurements.
Then, we present a bound for the stochastic error between the possibly regularized PBDW solution $u^{\star}$ fed with imperfect (noisy) measurements,  and the PBDW solution $u^{\rm opt}$, fed with perfect measurements. 
The idea of identifying different sources of error and deriving distinct bounds for each source is the same exploited in \cite{taddeivalidation}, for the estimate of the state estimation error based on local measurements of the state.
\end{enumerate}

The outline of the paper is as follows.
In section \ref{sec:formulation}, we present the methodology: 
we review  the derivation of the adaptive PBDW formulation for noisy measurements as discussed in \cite{taddei2017model}; 
we introduce the user-defined update space, and we discuss its variational interpretation; 
we discuss the selection of the observation functionals; 
we present an adaptive procedure for the selection of the hyper-parameter associated with the Tikhonov regularizer;
and we discuss the extension to vector-valued problems.
In section \ref{sec:analysis}, we present an \emph{a priori} error analysis for noisy measurements for the state estimation error.
In section \ref{sec:acoustics}   we present results for two  synthetic model problems  in Acoustics, to  illustrate the elements of the methodology and to prove its effectiveness.
 Finally, in section \ref{sec:fluids}, we consider  a synthetic Fluid Mechanics model problem to demonstrate the applicability of PBDW to vector-valued problems.

\section{Formulation}
\label{sec:formulation}
\subsection{Preliminaries}
\label{sec:preliminaries}
Given the Lipschitz domain $\Omega \subset \mathbb{R}^d$,   we introduce the Hilbert space $\mathcal{X}$
defined over  $\Omega$; we 
 endow  $\mathcal{X}$ with the inner product $(\cdot, \cdot)$ and the induced norm $\|  \cdot \| = \sqrt{ (\cdot, \cdot  )}$.
 We further denote by $\mathcal{X}'$ the dual space of $\mathcal{X}$.
 For any closed linear  subspace $\mathcal{Q} \subset \mathcal{X}$, we denote by $\Pi_{\mathcal{Q}} : \mathcal{X} \to \mathcal{Q}$ the orthogonal projection operator onto $\mathcal{Q}$, and we denote by $\mathcal{Q}^{\perp}$ its orthogonal complement.
Given  the linear functional $\ell \in \mathcal{X}'$, we denote by $R_{\mathcal{X}} \ell  \in \mathcal{X}$ the corresponding Riesz element,
$(R_{\mathcal{X}} \ell, f) = \ell(f)$ for all $f \in \mathcal{X}$.
 
 Given a random variable $X$, we denote by $\mathbb{E}[X]$ and by $\mathbb{V}[X]$ the mean and the variance, where $\mathbb{E}$ denotes expectation.
We denote by  $X \sim \mathcal{N}(m, \sigma^2)$ a Gaussian random variable with mean   $m$ and variance  $\sigma^2$. Similarly, we  denote by  $X \sim {\rm Uniform}(\Omega)$ an uniform random variable over $\Omega$.
Furthermore, we refer to an arbitrary random variable $\varepsilon$ such that $\mathbb{E}[\varepsilon]=0$ and $ \mathbb{V}[\varepsilon] = \sigma^2$ using the notation $\varepsilon \sim (0,\sigma^2)$.

We state upfront that in this section,  we limit ourselves to real-valued problems; however, the discussion can be trivially extended to the complex-valued case. Furthermore, in sections \ref{sec:pb_statement} --- \ref{sec:choice_xi}, we consider scalar fields; then, in section \ref{sec:vector_valued_problems}, we discuss the extension to vector-valued problems.
 
\subsection{Derivation of the PBDW statement}
\label{sec:pb_statement}

As explained in the introduction, we aim to estimate the deterministic state $u^{\rm true} \in \mathcal{X}$ over the domain of interest $\Omega \subset \mathbb{R}^d$.
We shall afford ourselves two sources of information:
a bk mathematical model 
$$
G^{\rm bk,\mu}(u^{\rm bk}(\mu)) =0,
\quad
\mu \in \mathcal{P}^{\rm bk}
$$
defined over a domain $\Omega^{\rm bk}$ that contains $\Omega$;
and $M$ experimental observations 
$y_1, \ldots,$ $y_M$ such that
$$
y_m= \ell_m^o(u^{\rm true}) + \epsilon_m, \qquad
 m=1,\ldots,M,
$$
where $\ell_1^{o}, \ldots, \ell_M^{o} \in \mathcal{X}'$ are suitable observation functionals,
 and 
$\{ \epsilon_m \}_{m=1}^M$ are unknown disturbances caused by either systematic error in the data acquisition system or experimental random noise.
Here, 
$\mathcal{P}^{\rm bk} \subset \mathbb{R}^P$ is a confidence region for the true values of the parameters of the model.
We further introduce the bk manifold $\mathcal{M}^{\rm bk} = \{ u^{\rm bk} (\mu)|_{\Omega}: \mu \in \mathcal{P}^{\rm bk} \}$ associated with the solution to the parameterized model for all values of the parameters in $\mathcal{P}^{\rm bk}$, restricted to the domain of interest $\Omega$.

In order to combine the parameterized model with the experimental observations, 
Wahba proposed the following generalization  of the 3D-VAR 
(\cite{bennett2002inverse,lorenc1986analysis})
statement for parameterized background, known as Partial Spline Model  (\cite[Chapter 9]{wahba1990spline}):
find the state estimate
$u_{\xi}^{\star} = u^{\rm bk}(\mu_{\xi}^{\star})  + \eta_{\xi}^{\star}$
such that
\begin{equation}
\label{eq:partial_spline_model_var_update}
(\mu_{\xi}^{\star}, \eta_{\xi}^{\star} )
 :=
 {\rm arg} \min_{(\mu,\eta) \in \mathcal{P}^{\rm bk} \times \mathcal{U}}
 \,
 \xi \|  \eta  \|^2
 \,
 +
  V_M \left(
 \mathcal{L}_M(u^{\rm bk}(\mu) + \eta) 
 - \mathbf{y}
 \right);
 \end{equation}
 where 
$\xi >0$ is a  tuning hyper-parameter,
 $\mathbf{y}=[y_1,\ldots,y_M]$ is the vector of experimental observations, $\mathcal{L}_M(u)=[\ell_1^o(u), \ldots,\ell_M^o(u)]$, and
 $V_M: \mathbb{R}^M \to \mathbb{R}_+$ is a suitable strictly convex loss function with minimum in $\mathbf{0}$ that will be specified later.
 Finally, $\mathcal{U}$ is 
the \emph{search space}:  $\mathcal{U}$ is  a closed linear subspace 
contained in $\mathcal{X}$
that will be specified in the next section.
We observe that    \eqref{eq:partial_spline_model_var_update} is non-convex in $\mu$;
furthermore, evaluations of the map $\mu \mapsto u^{\rm bk}(\mu)$ involve the solution to the bk model. Therefore, \eqref{eq:partial_spline_model_var_update} is
not suitable for real-time computations.

If we introduce the rank-$N$ approximation (\cite{cohen2015approximation}) of the bk field $u^{\rm bk}(\mu)$,
$$
u_N^{\rm bk}|_{\Omega}(x,\mu)
=
\sum_{n=1}^N 
\;
z_n(\mu) \, \zeta_n(x),
\qquad
x \in \Omega,
\;\;
\mu \in \mathcal{P}^{\rm bk},
$$
we can approximate statement \eqref{eq:partial_spline_model_var_update} as
\begin{equation}
 \label{eq:partial_spline_model_var_update_rankN}
(\mu_{\xi}^{\star}, \eta_{\xi}^{\star} )
 :=
 {\rm arg} \min_{(\mu,\eta) \in \mathcal{P}^{\rm bk} \times \mathcal{U}}
 \,
 \xi \|  \eta  \|^2
 \,
 +
 \,
V_M
\,
\left(
\mathcal{L}_M
\left( \sum_{n=1}^N  z_n(\mu) \zeta_n   +  \eta \right)
-
\mathbf{y}
\right).
 \end{equation}
Assuming that we are not interested in the estimate of the parameter $\mu_{\xi}^{\star}$, \eqref{eq:partial_spline_model_var_update_rankN} is equivalent to:
\begin{equation}
\label{eq:PBDW_general_non_convex}
(\boldsymbol{z}_{\xi}^{\star}, \eta_{\xi}^{\star})
=
{\rm arg} \min_{  ( \boldsymbol{z}, \eta  ) \in \Phi_N \times \mathcal{U}    }
\,
 \xi \|  \eta  \|^2
 \,
 +
 \,
V_M
\,
\left(
\mathcal{L}_M
\left( \sum_{n=1}^N  z_n \zeta_n   +  \eta \right)
-
\mathbf{y}
\right)
\end{equation}
where $\Phi_N = \{ [z_1(\mu),\ldots, z_N(\mu)]: \mu \in \mathcal{P}^{\rm bk} \} \subset \mathbb{R}^N$  is the image of the parameter-dependent coefficients, and
the corresponding state estimate is given by 
$u_{\xi}^{\star} = \sum_{n=1}^N$
$ (\mathbf{z}_{\xi}^{\star})_n \zeta_n + \eta_{\xi}^{\star}$.

The set $\Phi_N$ in \eqref{eq:PBDW_general_non_convex} is non-convex; therefore, we convexify 
\eqref{eq:PBDW_general_non_convex} by minimizing the objective over the convex approximation of $\Phi_N$, 
$\widetilde{\Phi}_N$:
\begin{equation}
\label{eq:PBDW_general}
(\boldsymbol{z}_{\xi}^{\star}, \eta_{\xi}^{\star})
=
{\rm arg} \min_{  ( \boldsymbol{z}, \eta  ) \in \widetilde{\Phi}_N \times \mathcal{U}    }
\,
 \xi \|  \eta  \|^2
 \,
 +
 \,
V_M
\,
\left(
\mathcal{L}_M
\left( \sum_{n=1}^N  z_n \zeta_n   +  \eta \right)
-
\mathbf{y}
\right)
\end{equation}
We refer to \eqref{eq:PBDW_general} as to the PBDW formulation: the formulation generalizes the statements introduced in 
\cite{maday2015parameterized,maday2015pbdw,taddei_APBDW}.

In this work, we study the following instance of the general PBDW formulation:
$$
(\boldsymbol{z}_{\xi}^{\star}, \eta_{\xi}^{\star})
=
{\rm arg} \min_{  ( \boldsymbol{z}, \eta  ) \in \mathbb{R}^N  \times \mathcal{U}    }
\,
 \xi \|  \eta  \|^2
 \,
 +
 \,
 \frac{1}{M}
\| 
\,
\mathcal{L}_M
\left( \sum_{n=1}^N  z_n \zeta_n   +  \eta \right)
-
\mathbf{y}
\|_2^2,
$$
which can also be reformulated as follows: 
\begin{equation}
\label{eq:PBDW_two_field_formulation}
(z_{\xi}^{\star}, \eta_{\xi}^{\star}) :=
{\rm arg}
\inf_{  (z,\eta)  \in \mathcal{Z}_N \times \mathcal{U}  }
\,
J_{\xi}(z, \eta):=
\xi \|  \eta   \|^2
\,
+
\,
 \frac{1}{M}
\| 
\,
\mathcal{L}_M
\left(z   +  \eta \right)
-
\mathbf{y}
\|_2^2,
\end{equation}
where $\mathcal{Z}_N = \mbox{span}\{ \zeta_n\}_{n=1}^N \subset \mathcal{X}$ is called the \emph{background space}.
In the limit $\xi \to 0^+$, the formulation reduces to 
(\cite[Proposition 2.5.1]{taddei2017model}): 
\begin{equation}
\label{eq:PBDW_two_field_formulation_noise_free}
(z_{\xi}^{\star}, \eta_{\xi}^{\star}) :=
{\rm arg}
\inf_{  (z,\eta)  \in \mathcal{Z}_N \times \mathcal{U}  }
\,
 \|  \eta   \|,
 \quad
 {\rm subject \, to}
 \;
\mathcal{L}_M
\left(z   +  \eta \right)
= 
\mathbf{y};
\end{equation}
which corresponds to the original formulation proposed in 
\cite{maday2015parameterized}.

Some comments are in order.
The choice of the loss function $V_M$ should be based on the expected experimental disturbances: the particular form considered here is appropriate for homoscedastic noise
($\epsilon_1,\ldots,\epsilon_M$ are independent realizations of the random variable $\epsilon \sim (0,\sigma^2)$).
Furthermore, as observed in 
\cite{argaud2016stabilization} in a related framework, the relaxation $\widetilde{\Phi}_N = \mathbb{R}^N$ might lead to stability issues for $N \simeq M$. We refer to a future work for the analysis of more general losses and for tighter relaxations $\widetilde{\Phi}_N$ of $\Phi_N$.
The hyper-parameter $\xi$ is selected adaptively, based on hold-out validation as in \cite{taddei_APBDW};
for completeness, in section \ref{sec:choice_xi},
we summarize the validation procedure.

\subsection{Choice of $\mathcal{U}$: variational and user-defined update spaces}
\label{sec:update_space}

We consider two choices for the search space $\mathcal{U}$: 
(i) $\mathcal{U}= \mathcal{X}$; and
(ii) $\mathcal{U}  = {\rm span} \{  \psi_m \}_{m=1}^M$ where 
$\psi_1,\ldots,\psi_M \in \mathcal{X}$ are linearly independent user-defined functions satisfying the unisolvency condition:
\begin{equation}
\label{eq:unisolvency_update}
\psi \in \mathcal{U}_M,
\;  \ell_m^o(\psi) = 0, \quad m=1,\ldots,M
\quad
\Leftrightarrow
\quad
\psi \equiv 0.
\end{equation}
The first choice is considered in the original PBDW papers, while the second choice is proposed here for the first time.

As shown in \cite[Proposition 2.2.3]{taddei2017model} for \eqref{eq:PBDW_two_field_formulation}, for $\mathcal{U}= \mathcal{X}$, the update $\eta_{\xi}^{\star}$ belongs to 
the update space $\mathcal{U}_M = {\rm span} \{ R_{\mathcal{X}} \ell_m^{o}    \}_{m=1}^M$;
we refer to this choice of the  update space $\mathcal{U}_M$ (or equivalently of $\mathcal{U}$) as to \emph{variational update}.
If $\ell_1^o,\ldots,\ell_M^o$ are linearly independent in $\mathcal{X}'$, it is easy to verify that  
the space $\mathcal{U}_M$ satisfies \eqref{eq:unisolvency_update}.
Computation of the  update space $\mathcal{U}_M$ requires the solution to $M$
 variational problems, and depends on the choice of the inner product $(\cdot, \cdot)$ of $\mathcal{X}$,
 and on the observation functionals $\ell_1^o,\ldots,\ell_M^o$.

We refer to the second choice 
 $\mathcal{U}_M = \mathcal{U} = {\rm span} \{ \psi_m    \}_{m=1}^M$  as \emph{user-defined update}.
The functions $\psi_1,\ldots,\psi_M$ are chosen based on approximation considerations, to guarantee fast convergence with respect to the number of measurements.
We observe that $\mathcal{U}_M$ mildly depends on the observation functionals --- it should satisfy \eqref{eq:unisolvency_update} --- but it does not depend on the choice of the inner product $(\cdot, \cdot)$ of $\mathcal{X}$. 
In particular, computation of the user-defined update does not in principle require the  solution to  any variational problem: 
we might thus rely on user-defined updates to simplify the implementation of the method.
It is possible to recover a variational interpretation for this update space for a suitable inner product of $\mathcal{X}$: we discuss this point in  section \ref{sec:var_empiricalUM}.

\subsubsection*{Practical choice of the update space}

As explained in the introduction, we are interested in experimental observations 
that can be modelled by linear functionals 
of the form
\begin{equation}
\label{eq:observation_functional}
\ell_m^o(u)
=
\ell 
\left(
u, x_m^{\rm obs}, r_{\rm w}
\right)
= 
C(x_m^{\rm obs})
\int_{\Omega} \,
\omega
\left(
 \frac{1}{r_{\rm w}}  \| x - x_m^{\rm obs}  \|_2   
\right) 
  \, u(x) \, dx,
\end{equation}
where $r_{\rm w}$ reflects the filter width of the transducer,  
$x_m^{\rm obs}$ denotes the  transducer location, and $\omega$ describes the local averaging process performed by the experimental device.
For this class of observation functionals, we consider updates of the form
$$
\mathcal{U}_M = {\rm span}
\{
\phi(\cdot, x_m^{\rm obs} )
\}_{m=1}^M.
$$
where $\phi(\cdot, x_m^{\rm obs}) = R_{\mathcal{X}} \ell(\cdot,   x_m^{\rm obs}, r_{\rm w} )$ (variational update), or $\phi$ might be an user-defined function.
We refer to $\phi$ as to \emph{update generator}.

If $\mathcal{X}$ is the free space $H^s(\Omega)$ for some $s \geq 1$, we might consider the generator
\begin{equation}
\label{eq:user_defined_generator_phi}
\phi(\cdot, x) =
\Phi \left(  \|  \cdot - x \|_2 \right),
\end{equation}
where  $\Phi: \mathbb{R}_+ \to \mathbb{R}_+$ is a positive definite kernel $\Phi$ (see, e.g., \cite{wendland2004scattered}).
Computation of \eqref{eq:user_defined_generator_phi}
does not require the solution to any variational procedure, and for small values of $r_{\rm w}$ it leads to superior approximation properties  compared to the standard $H^1$ variational update, as we will demonstrate in the numerical results.
Another potential choice is to consider:
\begin{equation}
\label{eq:user_defined_generator_phi_2}
\phi(\cdot, x) = R_{\mathcal{X}} \ell(\cdot,   x, R_{\rm w} ),
\end{equation}
for some $R_{\rm w} > r_{\rm w}$.
The generator \eqref{eq:user_defined_generator_phi_2}  requires the solution to $M$ variational problems during the offline stage (as opposed to \eqref{eq:user_defined_generator_phi});
on the other hand, it permits the imposition of homogeneous Dirichlet boundary conditions\footnote{
Homogeneous Dirichlet boundary conditions could  be imposed using the generator \eqref{eq:user_defined_generator_phi} by adding  measurements on the Dirichlet boundary:
this approach increases the dimensionality $M$ of the update space, and thus it increases the online cost associated with the computation of the state estimate.
}.

\subsection{Variational interpretation of the user-defined space, and well-posedness analysis}
\label{sec:var_empiricalUM}

Given the update space $\mathcal{U}_M \subset \mathcal{X}$, we introduce the interpolation operator
$\mathcal{I}_M: \mathcal{X} \to \mathcal{U}_M$ s.t. $\ell_m^o(  \mathcal{I}_M(u)   ) = \ell_m^o(u)$ for all $u \in \mathcal{X}$, $m=1,\ldots,M$.
Then, we define the symmetric bilinear form

\begin{equation}
\label{eq:GEIM_inner_product_alternative}
((u,v))
=
(  u - \mathcal{I}_M(u),   v- \mathcal{I}_M(v) ) 
\,
+
\,
(   \mathcal{I}_M(u),  \mathcal{I}_M(v) ).
\end{equation}
The bilinear form \eqref{eq:GEIM_inner_product_alternative} and the induced (semi-)norm 
$ \vertiii{\cdot} = \sqrt{(\cdot, \cdot)}$. satisfy the following properties.

\begin{lemma}
\label{th:properties_inner_product}
Let $\mathcal{U}_M = {\rm span}\{ \psi_m \}_{m=1}^M$ be an $M$-dimensional space satisfying \eqref{eq:unisolvency_update},
where $\psi_1, \ldots, \psi_M$ are orthonormal. Then, the following hold.
\begin{enumerate}
\item
$
\vertiii{u} 
=
\|  u   \|
\qquad
\forall \, u \in \mathcal{U}_M.
$
\item
Given the orthonormal basis $\psi_1,\ldots,\psi_M$ of $\mathcal{U}_M$, we define 
$\mathbb{L}_{\eta}\in \mathbb{R}^{M,M}$ such that
$(\mathbb{L}_{\eta})_{m,m'}= \ell_m(\psi_{m'})$. Then, we can rewrite \eqref{eq:GEIM_inner_product_alternative} as follows:
\begin{equation}
\label{eq:new_product_algebraic}
((u,v))
=
(  u - \mathcal{I}_M(u),   v- \mathcal{I}_M(v) ) 
\,
+
\,
\sum_{m,m'=1}^M \, \ell_m^o(u) \, \ell_{m'}^o(v) \, \mathbb{W}_{m,m'},
\end{equation}
where $\mathbb{W} := \mathbb{L}_{\eta}^{-T} \mathbb{L}_{\eta}^{-1}$.
\item
The bilinear form $((\cdot, \cdot ))$  induces a norm over $\mathcal{X}$. More precisely, the following estimate holds: 
\begin{equation}
\label{eq:estimate_new_product}
\frac{1}{2} \|  u  \|^2
\leq \vertiii{u}^2
\leq
\left(
2 + 3 \|  \mathcal{I}_M \|_{\mathcal{L}(\mathcal{X})}^2
\right) \|  u  \|^2,
\quad
\|  \mathcal{I}_M \|_{\mathcal{L}(\mathcal{X})}
=
\sup_{v \in \mathcal{X}} \, \frac{\|   \mathcal{I}_M(v)   \|}{ \| v \|}.
\end{equation}
\item
Let $\phi_m = \sum_{p=1}^M (\mathbb{L}_{\eta})_{m,p} \psi_p$. Then, $(( \phi_m, v  )) = \ell_m^o(v)$ for all $v \in \mathcal{X}$.
\item
If $\mathcal{U}_M = {\rm span} \{  R_{\mathcal{X}} \ell_m^o \}_{m=1}^M$,
then   $((\cdot, \cdot)) = (\cdot, \cdot)$.
\item
$\Pi_{\mathcal{U}_M}^{ \vertiii{\cdot} } u = \mathcal{I}_M(u)$ for all $u \in \mathcal{X}$.
\end{enumerate}
\end{lemma}

\begin{proof}
If $u \in \mathcal{U}_M$, $\mathcal{I}_M(u)=u$.
Therefore, 
$$
\vertiii{u}^2
=
(u - \mathcal{I}_M(u), u - \mathcal{I}_M(u)) + 
(\mathcal{I}_M(u), \mathcal{I}_M(u))
=
(u,u)
=
\|  u \|^2,
$$
which proves the first statement.

Observing that $\mathcal{I}_M(u) = \sum_{m=1}^M  ( \mathbb{L}_{\eta}^{-1} \mathcal{L}_M(u) )_m \psi_m$, and recalling that 
$\psi_1,\ldots,\psi_M$ are orthonormal, we find 
$$
\begin{array}{ll}
(\mathcal{I}_M(u), \mathcal{I}_M(v)) 
&
\displaystyle{
=
\sum_{m,m'=1}^M \,
 ( \mathbb{L}_{\eta}^{-1} \mathcal{L}_M(u) )_m
  ( \mathbb{L}_{\eta}^{-1} \mathcal{L}_M(v) )_{m'}
  \,
  (\psi_m, \psi_{m'})
 }
\\[4mm]
&
\displaystyle{
 =
  \sum_{m=1}^M \,
 ( \mathbb{L}_{\eta}^{-1} \mathcal{L}_M(u) )_m
  ( \mathbb{L}_{\eta}^{-1} \mathcal{L}_M(v) )_{m}
=
\sum_{p,p'=1}^M
\,
\ell_p^o(u) \,  \ell_{p'}^o(v) \, \mathbb{W}_{p,p'},
}
\\
\end{array}
$$ 
which proves the second statement.

We now prove \eqref{eq:estimate_new_product} (cf. statement 3).
First, we find that
$$
\|  u \|^2
\leq 2 \left(
\|  u - \mathcal{I}_M(u) \|^2
+
\| \mathcal{I}_M(u) \|^2
\right)
=
2 \vertiii{u}^2,
$$
where the first inequality follows from 
the triangular inequality and
$(a+b)^2 \leq 2 a^2 + 2b^2$, while the second identity follows 
 directly 
from the definition of $((\cdot, \cdot))$.
Second, exploiting  the fact that the interpolation operator is continuous in $\mathcal{X}$, we find
$$
\vertiii{u}^2
=
\|  u - \mathcal{I}_M(u) \|^2
+
\| \mathcal{I}_M(u) \|^2
\leq
2 \| u \|^2
+
3 \|  \mathcal{I}_M(u)   \|^2
\leq 
\left(
2 + 3 \|  \mathcal{I}_M \|_{\mathcal{L}(\mathcal{X})}^2
\right)   \| u \|^2.
$$
Combining the latter two estimates, we obtain \eqref{eq:estimate_new_product}.

We prove the fourth statement.
Since $\phi_m \in \mathcal{U}_M$, $\phi_m - \mathcal{I}_M(\phi_m) \equiv 0$. Therefore,
$$
((  \phi_m, v)) = \sum_{p,p'=1}^M \, \ell_p^o(\phi_m) \, \ell_{p'}^o(v) \mathbb{W}_{p,p'}.
$$
By definition of the interpolation operator, 
we have 
$\ell_{p'}^o(v) = \ell_{p'}^o(  \mathcal{I}_M(v)  )$ for $p'=1,\ldots,M$, and 
$\mathcal{I}_M(v) = \sum_{k'=1}^M \left(
\mathbb{L}_{\eta}^{-1} \mathcal{L}_M(v)
\right)_{k'} \psi_{k'}$. Then, we find
$$
\begin{array}{ll}
(( \phi_m, v   )) = &
\displaystyle{
\sum_{k,k'=1}^M \sum_{p,p'=1}^M
\,
\left( \mathbb{L}_{\eta} \right)_{m,k}
\ell_p^o(\psi_k) \ell_{p'}^o(\psi_{k'})
\mathbb{W}_{p,p'} \left( \mathbb{L}_{\eta}^{-1} \mathcal{L}_M(v) \right)_{k'}
}
\\[4mm]
&
\displaystyle{
=
\sum_{k,q=1}^M (\mathbb{L}_{\eta})_{m,k}
\left(
\mathbb{L}_{\eta}^{-1} 
\right)_{k,q}
\ell_q^o(v)
=
\sum_{q=1}^M
\,
\left(
\mathbb{L}_{\eta}
\mathbb{L}_{\eta}^{-1} 
\right)_{m,q}
\ell_q^o(v)
=
\ell_m^o(v),}
\end{array}
$$
for all  $v \in \mathcal{X}$.
Thesis follows.

In order to prove the fifth statement,
we observe that, 
for $\mathcal{U}_M = {\rm span} \{ R_{\mathcal{X}} \ell_m^o \}_{m=1}^M$,
$\mathcal{I}_M(u) = \Pi_{\mathcal{U}_M} u$.
Then, exploiting the properties 
of $((\cdot, \cdot))$ shown so far
and the projection theorem, we find
$$
\vertiii{u}^2  =
\| u -  \mathcal{I}_M(u)\|^2 + \|  \mathcal{I}_M(u)\|^2
=
\| \Pi_{\mathcal{U}_M^{\perp}}(u)\|^2
+  \| \Pi_{\mathcal{U}_M} (u)\|^2
=
\| u \|^2.
$$

The proof the sixth statement follows directly from the fourth property:
$$
(( \Pi_{\mathcal{U}_M}^{ \vertiii{\cdot} } u, v)) =  (( u, v)) \; \forall \, v \in \mathcal{U}_M \,
\Leftrightarrow \,
\ell_m^o(  \Pi_{\mathcal{U}_M}^{ \vertiii{\cdot} } u    ) = \ell_m^o(v), \;m=1,\ldots,M.
$$
The latter completes the proof.
\end{proof}

\begin{remark}
\label{norm_calIM}
The norm $\|  \mathcal{I}_M \|_{\mathcal{L}(\mathcal{X})}$ is the Lebesgue constant in the $\mathcal{X}$ norm, and it is the inverse of the inf-sup constant (\cite[Theorem 2.4]{maday2015generalized}):
$$
\gamma_M:=
\inf_{w \in \mathcal{U}_M} \sup_{v \in \mathcal{W}_M} \, 
\frac{(w,v)}{\| w  \| \| v \|  },
\quad
\mathcal{W}_M
=
{\rm span} \{ 
R_{\mathcal{X}} \ell_m^o
 \}_{m=1}^M.
$$
The Lebesgue constant depends on the triple $\mathcal{O}= (\{ \ell_m^o \}_{m=1}^M,  \mathcal{U}_M, (\mathcal{X}, \| \cdot \|)   )$.
For certain choices of the triple $\mathcal{O}$, it   is possible to determine the asymptotic behavior of $\|  \mathcal{I}_M \|_{\mathcal{L}(\mathcal{X})}$:
to provide a concrete example, 
we  refer  to \cite[Corollary 1.17]{bernardi2004discretisations} for an important result concerning the behavior of $\|  \mathcal{I}_M \|_{\mathcal{L}(\mathcal{X})}$ for
one-dimensional polynomial interpolation.
\end{remark}

Lemma \ref{th:properties_inner_product} can be exploited to 
recover an infinite-dimensional variational interpretation of the PBDW statement with user-defined update.
The proof is a straightforward consequence of \cite[Proposition 2.2.3]{taddei2017model},  and is here omitted.

\begin{proposition}
\label{th:variation_empirical_update}
Let $\mathcal{U}_M$ be an $M$-dimensional space satisfying \eqref{eq:unisolvency_update}; then, the PBDW solution to \eqref{eq:PBDW_two_field_formulation} 
for $\mathcal{U}= \mathcal{U}_M$
solves the following problem:
\begin{equation}
\label{eq:PBDW_GEIM_inner}
(z_{\xi}^{\star}, \eta_{\xi}^{\star}) :=
{\rm arg}
\inf_{  (z,\eta)  \in \mathcal{Z}_N \times \mathcal{X}  }
\,
\xi \vertiii{\eta}^2
+
\frac{1}{M}
\|  
\mathcal{L}_M  \left(z   +  \eta \right) -  \mathbf{y}
\|_2^2,
\end{equation}
where $\vertiii{\cdot}$ is the norm induced by the inner product $((\cdot, \cdot))$ of $\mathcal{X}$ defined in 
\eqref{eq:GEIM_inner_product_alternative}.

If the inf-sup constant (\cite{maday2015parameterized})
\begin{equation}
\label{eq:inf_sup_constant}
\beta_{N,M}
= \inf_{z \in \mathcal{Z}_N} \sup_{q \in \mathcal{U}_M}
\,
\frac{((z,q))}{  \vertiii{z} \vertiii{q}   }
\end{equation}
is strictly positive, then the solution 
$(z_{\xi}^{\star}, \eta_{\xi}^{\star}) $
to  \eqref{eq:PBDW_GEIM_inner} is unique, and solves the following saddle-point problem:
\begin{equation}
\label{eq:APBDW_saddle}
\left\{
\begin{array}{lll}
2\xi ((\eta_{\xi}^{\star}, q)) + 
\frac{2}{M}
\, \sum_{m=1}^M
\,
\left(   \ell_m^o(  z_{\xi}^{\star}  +  \eta_{\xi}^{\star}   )  - y_m   \right) \, \ell_m^o(q) 
&
=0
&
\forall \, q \in \mathcal{U}_M;
\\[3mm]
((\eta_{\xi}^{\star}, p))
&
=0
&
\forall \, p  \in \mathcal{Z}_N.
\\
\end{array}
\right.
\end{equation}  
Furthermore,  the estimate $u_{\xi}^{\star}$ belongs to the space:
$$
u_{\xi}^{\star} \in \mathcal{Z}_N \oplus \mathcal{Z}_N^{\perp, \vertiii{\cdot}} \cap \mathcal{U}_M
=
\{
z + \eta: 
z \in \mathcal{Z}_N,
\; \;
\eta \in \mathcal{U}_M,
\; \;
((\eta, q)) = 0 \; \; \; \forall q \in \mathcal{Z}_N
\};
$$
where 
$\mathcal{Z}_N^{\perp, \vertiii{\cdot}}$ denotes the orthogonal complement  of
$\mathcal{Z}_N$ with respect to the $\vertiii{\cdot}$ norm.
\end{proposition}

\begin{remark}
In the variational approach, we first choose an inner product $(\cdot, \cdot)$ for $\mathcal{X}$, and then we appeal to an high-fidelity solver to generate the update space;
on the other hand, in the user-defined  approach, we first choose the space $\mathcal{U}_M$, and then we use it to define the inner product $((\cdot, \cdot))$ \eqref{eq:GEIM_inner_product_alternative}.
We remark that the second approach is used in the kernel methods' literature (see, e.g., \cite{wendland2004scattered}) for pointwise measurements. Given a positive definite kernel is possible to characterize the Sobolev regularity of the resulting ambient (\emph{native}) space: the characterization of the properties of the ambient space can be then exploited 
to prove \emph{a priori} error bounds for the state estimation error, and estimate the asymptotic convergence rate as $M \to \infty$.
On the other hand, the  construction presented in this section does not allow us to characterize the smoothness of the space induced  by the norm $\vertiii{\cdot}$
in the limit $M \to \infty$, unless  $\| \mathcal{I}_M \|_{\mathcal{L}(\mathcal{X})}$ is bounded for $M \to \infty$.
\end{remark}

\subsection{Algebraic formulation}
\label{sec:algebraic}

Given the fields
$z \in \mathcal{Z}_N$ and $\eta \in \mathcal{U}_M$, we introduce the vectors $\mathbf{z} \in \mathbb{R}^N$ and $\boldsymbol{\eta} \in \mathbb{R}^M$ such that
$z = \sum_{n=1}^N z_n \zeta_n$, and $\eta = \sum_{m=1}^M \eta_m \psi_m$,
where $\zeta_1,\ldots,\zeta_N$ and
$\psi_1,\ldots,\psi_M$ are orthonormal bases
(with respect to $\|  \cdot \|$)
 of $\mathcal{Z}_N$ and $\mathcal{U}_M$, respectively.
 We remark upfront that the derivation is independent of the particular construction of the update space.
Then, we can rewrite \eqref{eq:PBDW_two_field_formulation} as a discrete optimization problem for the coefficients:
\begin{equation}
\label{eq:PBDW_algebraic}
(\mathbf{z}_{\xi}^{\star},  \boldsymbol{\eta}_{\xi}^{\star})
=
{\rm arg} \min_{  (\mathbf{z},  \boldsymbol{\eta}) \in \mathbb{R}^N \times \mathbb{R}^M  }
\,
\xi \|  \boldsymbol{\eta}  \|_2^2
\,
+
\,
\frac{1}{M}
\|   \mathbb{L}_{\eta} \boldsymbol{\eta} + 
 \mathbb{L}_{z} \mathbf{z} - \mathbf{y}  \|_2^2;
\end{equation}
where $\mathbb{L}_z \in \mathbb{R}^{M,N}$, and 
$\mathbb{L}_{\eta} \in \mathbb{R}^{M,M}$ are given by
$(\mathbb{L}_z)_{m,n} = \ell_m^o(\zeta_n)$,
$(\mathbb{L}_{\eta})_{m,m'} = \ell_m^o(\psi_{m'})$.
We observe that   \eqref{eq:unisolvency_update} implies that $\mathbb{L}_{\eta}$ is invertible.
By differentiating the objective function with respect to $\mathbf{z}$ and $\boldsymbol{\eta}$, we obtain
\begin{equation}
\label{eq:PBDW_algebraic_temp}
\left\{
\begin{array}{l}
(\xi M  \mathbb{I} +   \mathbb{L}_{\eta}^T \mathbb{L}_{\eta}     ) \boldsymbol{\eta}_{\xi}^{\star}
+  \mathbb{L}_{\eta}^T \mathbb{L}_{z}  \boldsymbol{z}_{\xi}^{\star}
 =   \mathbb{L}_{\eta}^T    \boldsymbol{y};   \\[4mm]
\mathbb{L}_{z}^T \mathbb{L}_{\eta}  \boldsymbol{\eta}_{\xi}^{\star}
+  \mathbb{L}_{z}^T \mathbb{L}_{z}  \boldsymbol{z}_{\xi}^{\star}
 = \mathbb{L}_{z}^T    \boldsymbol{y}.   \\
\end{array}
\right.
\end{equation}
If we multiply \eqref{eq:PBDW_algebraic_temp}$_1$ by $\mathbb{L}_{z}^T  \mathbb{L}_{\eta}^{-T}$, and we exploit \eqref{eq:PBDW_algebraic_temp}$_2$ we obtain
$$
\xi M    \mathbb{L}_{\eta}^T  \mathbb{L}_{\eta}^{-T}  \boldsymbol{\eta}_{\xi}^{\star} + 
\mathbb{L}_{z} ^T \mathbb{L}_{\eta}  \boldsymbol{\eta}_{\xi}^{\star} +
\mathbb{L}_{z}^T \mathbb{L}_{z}  \boldsymbol{z}_{\xi}^{\star}
 =
\mathbb{L}_{z}^T    \boldsymbol{y} 
   =
\mathbb{L}_{z}^T \mathbb{L}_{\eta}  \boldsymbol{\eta}_{\xi}^{\star} +  \mathbb{L}_{z}^T \mathbb{L}_{z}  \boldsymbol{z}_{\xi}^{\star}  
 \Rightarrow
\mathbb{L}_{z}^T \mathbb{L}_{\eta}^{-T}  \boldsymbol{\eta}_{\xi}^{\star}  = \mathbf{0}
$$
Finally, if we introduce $ 
\tilde{ \boldsymbol{\eta}}_{\xi}^{\star} =  
\mathbb{L}_{\eta}^{-T}      \boldsymbol{\eta}_{\xi}^{\star}$ and we multiply \eqref{eq:PBDW_algebraic_temp}$_1$ by $\mathbb{L}_{\eta}^{-T}$, we obtain
\begin{equation}
\label{eq:PBDW_algebraic_linear}
\left\{
\begin{array}{l}
\left( \xi M  \mathbb{I} + \mathbb{L}_{\eta}  \mathbb{L}_{\eta}^{T} \right) \tilde{ \boldsymbol{\eta}}_{\xi}^{\star} 
+ \mathbb{L}_{z}  \boldsymbol{z}_{\xi}^{\star} = \boldsymbol{y},  \\[3mm]
\mathbb{L}_{z}^T  \tilde{ \boldsymbol{\eta}}_{\xi}^{\star}  = \boldsymbol{0};
\\
\end{array}
\right.
\qquad
 \boldsymbol{\eta}_{\xi}^{\star} =  \mathbb{L}_{\eta}^{T}   \tilde{ \boldsymbol{\eta}}_{\xi}^{\star}. 
\end{equation}
System \eqref{eq:PBDW_algebraic_linear} can be used to efficiently compute the PBDW solution.

\begin{remark}
(\textbf{spectrum of the PBDW system})
Following the argument in 
\cite[Section 3.4]{benzi2005numerical}, we observe that
the saddle-point system \eqref{eq:PBDW_algebraic_linear}  is congruent to the block-diagonal matrix 
$\left[
\begin{array}{cc}
\mathbb{L}_{\eta} \mathbb{L}_{\eta}^T  + \xi M \mathbb{I} & 0  \\
0   & - \mathbb{L}_z^T (\mathbb{L}_{\eta} \mathbb{L}_{\eta}^T   + \xi M \mathbb{I})^{-1} \mathbb{L}_z \\
\end{array}
\right]$. 
If the update satisfies \eqref{eq:unisolvency_update}, $\mathbb{L}_{\eta}$ is invertible, and thus 
$\mathbb{L}_{\eta} \mathbb{L}_{\eta}^T  + \xi M \mathbb{I}$ is positive definite for all $\xi \geq 0$.

On the other hand, the second block 
$- \mathbb{L}_z^T (\mathbb{L}_{\eta} \mathbb{L}_{\eta}^T   + \xi M \mathbb{I})^{-1} \mathbb{L}_z$ is negative definite if and only if 
the rank of $\mathbb{L}_z$ is equal to $N$: this condition is equivalent to the positivity of $\beta_{N,M}$ in Proposition \ref{th:variation_empirical_update}.
\end{remark}

In view of the analysis of the method, and of the definition of the Greedy procedure for the selection of the transducers' location,  next Proposition  provides a computable expression for the inf-sup constant $\beta_{N,M}$ defined in \eqref{eq:inf_sup_constant}, and for the norm $\|  \mathcal{I}_M \|_{\mathcal{L}(\mathcal{X})}$ in \eqref{eq:estimate_new_product}.

\begin{proposition}
\label{th:inf_sup}
Suppose that $\zeta_1,\ldots,\zeta_N$ and $\psi_1,\ldots,\psi_M$ are orthonormal in $\| \cdot \|$.
The inf-sup constant $\beta_{N,M}$ is the square root of the minimum eigenvalue of the following eigenproblem:
\begin{equation}
\label{eq:inf_sup_algebraic}
\mathbb{L}_z^T \mathbb{W} \mathbb{L}_z \, \mathbf{z}_n
=
\nu_n \left(
\mathbb{I} + 
2 \mathbb{L}_z^T  \mathbb{W} \mathbb{L}_z
- 2  {\rm sym} \left( \mathbb{C}^T \mathbb{L}_{\eta}^{-1} \mathbb{L}_z \right)
\right)
\mathbf{z}_n,
\quad
n=1,\ldots,N;
\end{equation}
while the norm $\|  \mathcal{I}_M \|_{\mathcal{L}(\mathcal{X})}$ is the inverse of the square root of the minimum eigenvalue of the following eigenproblem:
\begin{equation}
\label{eq:gammaM_algebraic}
\mathbb{L}_{\eta}^T \mathbb{K}^{-1}  \mathbb{L}_{\eta} \boldsymbol{\eta}_m =
\lambda_m  \boldsymbol{\eta}_m, \quad
m=1,\ldots,M.
\end{equation}
Here,
$\mathbb{W}  = \mathbb{L}_{\eta}^{-T} \mathbb{L}_{\eta}^{-1}$,
${\rm sym}(A)=\frac{1}{2}(A+A^T)$, while the matrices 
$\mathbb{C} \in \mathbb{R}^{M,N}$ and  
$ \mathbb{K} \in \mathbb{R}^{M,M}$ 
are given by
$\mathbb{C}_{m,n} = (\psi_m, \zeta_n)$, and 
$\mathbb{K}_{m,m'} = (R_{\mathcal{X}} \ell_m^o, R_{\mathcal{X}} \ell_{m'}^o)$.
\end{proposition}

\begin{proof}
Given $z = \sum_{n=1}^N z_n \zeta_n$, $\eta = \sum_{m=1}^M \eta_m \psi_m$,
we can write the interpolation operator as
$\mathcal{I}_M(z) = \sum_{k=1}^M \, \left(   \mathbb{L}_{\eta}^{-1} \mathbb{L}_z \mathbf{z}  \right)_k$
$ \psi_k$. Then, we obtain
$$
((z, \eta   )) = \mathbf{z}^T \mathbb{L}_z^T \mathbb{W} \mathbb{L}_{\eta} \boldsymbol{\eta},
\quad
\vertiii{\eta}^2 = \boldsymbol{\eta}^T \boldsymbol{\eta};
$$
and
$$
 \vertiii{z}^2 = 
  \|    z  \|^2 +  \|   \mathcal{I}_M(z)  \|^2 - 2 (z, \mathcal{I}_M(z))
+ 
\mathbf{z}^T \mathbb{L}_z^T \mathbb{W} \mathbb{L}_z  \boldsymbol{z}
=
\mathbf{z}^T \, \mathbb{B} \,
\mathbf{z}.
$$
where 
$\mathbb{B} = 
\mathbb{I} + 
2 \mathbb{L}_z^T  \mathbb{W} \mathbb{L}_z
-2  {\rm sym} \left( \mathbb{C}^T \mathbb{L}_{\eta}^{-1} \mathbb{L}_z \right)$.
Recalling the definition of $\mathbb{W}$,
we find
$$
\beta_{N,M}^2
=
\inf_{\mathbf{z} \in \mathbb{R}^N} \sup_{ \boldsymbol{\eta} \in \mathbb{R}^M  }
\,
\frac{
\left(   
\mathbf{z}^T \mathbb{L}_z^T \mathbb{W} \mathbb{L}_{\eta} \boldsymbol{\eta}
\right)^2
}
{ ( \boldsymbol{\eta}^T  \boldsymbol{\eta} )
\,
\mathbf{z}^T \mathbb{B} \mathbf{z}
}
=
\inf_{\mathbf{z} \in \mathbb{R}^N} 
\,
\frac{   
\mathbf{z}^T \mathbb{L}_z^T \mathbb{W} \mathbb{L}_{z} \mathbf{z}
}
{
\mathbf{z}^T \mathbb{B} \mathbf{z}
}.
$$
Introducing the Lagrangian multiplier $\nu$, we can write the optimality conditions as
$$
\left\{
\begin{array}{l}
\mathbb{L}_z^T \mathbb{W} \mathbb{L}_{z} \mathbf{z} 
=
\nu \mathbb{B} \mathbf{z};  \\[3mm]
\mathbf{z}^T \, \mathbb{B} \mathbf{z}  = 1.
\\
\end{array}
\right.
$$
Thesis follows.

The proof of \eqref{eq:gammaM_algebraic} exploits Remark \ref{norm_calIM}, and the same  argument used to prove \eqref{eq:inf_sup_algebraic}. 
We omit the details.
\end{proof}

\subsection{\texttt{SGreedy+approximation} algorithm for the selection of the observation centers}
\label{sec:sgreedy}

Algorithm \ref{SGreedy_plus} summarizes the Greedy procedure for the selection of the transducers' locations
 $x_1^{\rm obs}, \ldots, x_M^{\rm obs}$.
 During the first (stability) stage, we maximize the constant $\beta_{N,M}$ in a Greedy manner.
During the second (approximation) stage, we minimize  the fill distance 
$h_{M} := \sup_{x \in \Omega} \, \min_{m=1,\ldots, M} \| x - x_m^{\rm obs} \|_2$ in a Greedy manner.
We switch from the first to the second stage when the inf-sup constant $\beta_{N,M}$ is larger than an user-defined constant $tol$:
if $tol=0$, all the centers are selected through the approximation loop,
if $tol=1$, all centers are selected through the stability loop.
Representative values for $tol$ used in the numerical simulations are 
$tol \in  [0.2, 0.6]$.

The Greedy procedure was first presented in
\cite[Algorithm 3.2.1]{taddei2017model} for the variational update space.
We observe that the stability loop --- for variational update --- corresponds\footnote{
In the SGreedy procedure listed in \cite[Algorithm 2]{maday2015parameterized}, steps 6 and 7 of Algorithm \ref{SGreedy_plus} are replaced by 
$\ell_{m+1}^o= {\rm arg} \min_{\ell   \in  \mathscr{L}} | \ell( z_{{\rm min},m} - \mathcal{I}_M( z_{{\rm min},m}  )  )|  $, and
$\mathcal{U}_{m+1} = \mathcal{U}_m \cup {\rm span}\{ R_{\mathcal{X}} \ell_{m+1}^o  \}$,
where $\mathscr{L} = \{ \ell(\cdot, x, r_{\rm w}): \, x \in \overline{\Omega}  \}$ is the dictionary of available functionals. 
The approach presented here is easier to implement, and less computationally expensive.
Furthermore, the two procedures are asymptotically equivalent for  $r_{\rm w} \to 0^+$, and return similar results if $r_{\rm w}$ is small enough compared to the characteristic length-scale of the elements in $\mathcal{Z}_N$.}
 to the SGreedy algorithm proposed in \cite{maday2015parameterized}; on the other hand,  the  strategy for the approximation step is strongly related to the so-called \emph{farthest-first traversal} approach to the minimax facility location problem (see, e.g, \cite{owen1998strategic}), first proposed by Rosenkrantz et al. in \cite{rosenkrantz1977analysis}.
For the variational update, since $\beta_{N,M}$ is a non-decreasing function of $M$ for a fixed value of $N$ (see \cite{maday2015parameterized}), the stability constant remains above the threshold during the approximation stage.
On the other hand, for the user-defined update, there is in general no guarantee that the inf-sup constant will be above the threshold at the end of the Greedy procedure
since the inner product $((\cdot, \cdot))$ and the induced norm vary with $m$.
In the numerical experiments
we investigate the behavior  of the inf-sup constant for large values of $M$.

\begin{algorithm}[H]                      
\caption{Greedy stability-approximation balancing (\texttt{SGreedy+approx}) algorithm}     
\label{SGreedy_plus}                           
 \small
\begin{flushleft}
\begin{tabular}{|lll|}
\hline
\textbf{Input} &
$\mathcal{Z}_N  = {\rm span} \{ \zeta_n \}_{n=1}^N$
&
background space
\\ 
&
$M$
&
number of sensors
\\ 
&
$tol>0$
&
threshold for the stability constant
\\
&
$\phi: \Omega \times  \Omega \to \mathbb{R}$
&
 update generator
\\[2mm]
\hline
\textbf{Output}
&
$\mathcal{U}_M$
&
update space 
\\[2mm]
  \hline
\end{tabular}
\end{flushleft}  
 \normalsize 

\textbf{Stability}
 
\begin{algorithmic}[1]
\State
Compute 
$
x_{1}^{\rm obs} := 
\mbox{arg} \max_{x \in \bar{\Omega} } |  \zeta_1  (x)   |$, 
$\mathcal{U}_{1} =   \mbox{span} \{ \phi (\cdot, x_{1}^{\rm obs})   \} $,
$m=1$
\vspace{3pt}

\While {$m \leq M$}          
\State
Compute $\beta_{N,m}= \min_{w \in \mathcal{Z}_{N}} 
\max_{v \in \mathcal{U}_{m}} 
\frac{    ((w,v))    }
{\vertiii{w}  \vertiii{v}     }   $.
\vspace{3pt}

\If{ $\beta_{N,m}\leq tol$}
\State
Compute
$
z_{{\rm min},m}:= \mbox{arg} \min_{z \in \mathcal{Z}_N} 
\max_{v \in \mathcal{U}_{m}} 
\frac{    ((z,v))    } {\vertiii{z}  \vertiii{v}     } .
$
\State
Compute 
$
x_{m+1}^{\rm obs} := 
\mbox{arg} \max_{x \in  \bar{\Omega}  } 
|z_{{\rm min},m}(x)  - 
\mathcal{I}_m ( z_{{\rm min},m}) (x)   |.
$
\vspace{2pt}

\State
Set 
$\mathcal{U}_{m+1} = \mathcal{U}_{m} \cup \mbox{span} \{ \phi (\cdot, x_{m+1}^{\rm obs})   \} $,
$m=m+1$.
\vspace{2pt}

\Else
\State
\texttt{Break}
\EndIf
\EndWhile
\end{algorithmic}
\vspace{4pt}

\textbf{Approximation}

\begin{algorithmic}[1]
\While {$m \leq M$}          
\State
Compute 
$ x_{m+1}^{\rm obs} :=   \mbox{arg} \max_{x \in \bar{\Omega} } \,   \min_{m'=1,\ldots,m }  \,  \|  x  - x_{m'}^{\rm obs} \|_2  $.
\vspace{2pt}

\State
Set 
$\mathcal{U}_{m+1} = \mathcal{U}_{m} \cup \mbox{span} \{ \phi (\cdot, x_{m+1}^{\rm obs})   \} $,
$m=m+1$.
\vspace{2pt}

\EndWhile
\end{algorithmic}
\end{algorithm}

\subsection{Choice of $\xi$}
\label{sec:choice_xi}

The choice of $\xi$ is performed using holdout validation.
We consider two mutually exclusive datasets $\mathcal{D}_M^{\rm train} = \{ (\ell_m^o, y_m)  \}_{m=1}^M$ and
$\mathcal{D}_I^{\rm val} = \{ (\ell_i^o, y_i)  \}_{i=1}^I$
;
given the finite-dimensional search space $\Xi \subset \mathbb{R}_+$, we then select $\xi^{\star}$ such that
$$
\xi^{\star}
=
{\rm arg} \min_{\xi \in \Xi }
\,
\widehat{{\rm MSE}}(I)
:=
\frac{1}{I}
\sum_{i=1}^I \, \left( y_i - \ell_i^o(u_{\xi}^{\star}   ) \right)^2,
$$
where $u_{\bar{\xi}}^{\star} $ is the PBDW solution based on the training dataset $\mathcal{D}_M^{\rm train}$, for $\xi = \bar{\xi}$.

In this work, we consider  validation measurements of the form 
$\{ y_i =  \ell(u, x_i^{\rm obs}, r_{\rm w} )$  $+ \epsilon_i  \}_{i=1}^I$, 
where $\ell $ is introduced in \eqref{eq:observation_functional}, and
$x_1^{\rm obs}, \ldots, x_I^{\rm obs}$ are independent realizations of an uniformly-distributed random variable over $\Omega$, $X \sim {\rm Uniform}(\Omega)$.
 If $r_{\rm w}$ is small, for this choice of the observation centers, and assuming that 
$\epsilon_1, \ldots,\epsilon_I$ are independent realizations of $\epsilon  \sim (0,\sigma^2)$,  it is possible to show that  (see \cite{taddeivalidation})
$$
\mathbb{E} \left[
\widehat{{\rm MSE}}(I)
\right]
= 
\frac{1}{|\Omega|}
\,
\| u^{\rm true} - u_{\xi}^{\star} \|_{L^2(\Omega)}^2 + \sigma^2
+
C(r_{\rm w}, u^{\rm true} - u_{\xi}^{\star}),
$$
where $C(r_{\rm w}, u^{\rm true} - u_{\xi}^{\star}) \to 0$ as $r_{\rm w} \to 0^+$ if 
$\nabla (u^{\rm true} - u_{\xi}^{\star}) \in L^q(\Omega)$, $q>d$.
Therefore, for sufficiently large values of $I$, the validation procedure approximately minimizes the $L^2$ error.

The choice of the number of validation measurements is a trade-off between experimental cost (given by the number of transducers dedicated to validation) and reliability of the validation procedure.
$\kappa$-fold cross-validation 
(see, e.g., \cite[Chapter 7]{hastie2009elements} and  \cite{kohavi1995study})
reduces the experimental cost by partitioning the training  dataset $\mathcal{D}_M$ into $\kappa$ equal-sized subsamples (folds) $\{\mathcal{D}_M^{(k)}\}_{k=1}^{\kappa}$. Of the $\kappa$ folds,  a single fold is retained for testing and the remaining $\kappa-1$ folds are used for training. The procedure is then repeated $\kappa$ times with each of the $\kappa$ folds used once as the validation dataset.
After having selected the hyper-parameter $\xi$, the entire dataset is finally used for training.
We emphasize that cross-validation
relies on the assumption that the pairs $(x_m^{\rm obs}, y_m)$ 
are independently sampled from the joint distribution of inputs and outputs
(\emph{random design}).
On the other hand, in our framework the training centers $\{ x_m^{\rm obs} \}_{m=1}^M$ are chosen deterministically 
(\emph{fixed design}).
For this reason, in this work, we simply consider holdout validation with $I=M/2$; we refer to a future work for more advanced cross-validation strategies.

\subsection{Extension to vector-valued problems}
\label{sec:vector_valued_problems}

We can trivially extend the Greedy procedure for the selection of the observation centers to vector-valued fields
$u^{\rm true} = [u_1^{\rm true}, \ldots, u_J^{\rm true}] \in \mathcal{X}=\mathcal{X}(\Omega; \mathbb{R}^J)$.
Assuming that each transducer is able to measure all $J$ components of the true field in all $d$ directions, at each iteration of the Greedy algorithm, we select $x_m^{\rm obs}$ such that
(compare with Algorithm \ref{SGreedy_plus}, step 6)
$$
x_{m+1}^{\rm obs} = {\rm arg} \max_{x \in \overline{\Omega} } 
\,
\| z_{\rm min}(x) - \mathcal{I}_m (z_{\rm min}) (x)  \|_2.
$$
Then, we add $J$ modes $\phi_1(\cdot, x_{m+1}^{\rm obs}), \ldots, \phi_J(\cdot, x_{m+1}^{\rm obs}) \in \mathcal{X}$ in the update; note that  the subspace $\mathcal{U}_m$ is $m \times J$ dimensional).

As regards the choice of the $J$ generators, we might consider the generator
\begin{subequations}
\begin{equation}
\label{eq:variational_gen_vec}
\phi_i(\cdot, x) = R_{\mathcal{X}} \ell_i(\cdot, x, R_{\rm w}),
\quad
{\rm where} \,
\ell_i(u,x, R_{\rm w}) = \ell(u_i, x,R_{\rm w}),
\quad
i=1,\ldots,J;
\end{equation}
and $R_{\rm w} \geq r_{\rm w}$. Alternatively, we might consider the update 
\begin{equation}
\label{eq:user_def_vec}
\phi_i(\cdot, x) =
\Phi(\| \cdot - x \|_2) \mathbf{e}_i,
\quad
i=1,\ldots,J;
\end{equation}
where $\Phi$ is a properly-defined kernel, and 
$\mathbf{e}_1,\ldots,\mathbf{e}_J$ are the canonical vectors in $\mathbb{R}^J$.
\end{subequations}

Note that, in the case of incompressible flows considered in section  \ref{sec:fluids},
for a proper choice of the ambient space $\mathcal{X}$, the  update space \eqref{eq:variational_gen_vec} is divergence-free, 
while the user-defined update space \eqref{eq:user_def_vec} is not divergence-free.
For this reason, we might resort to matrix-valued divergence-free kernels (see, e.g., \cite{wendland2009divergence}).
We refer to a future work for the use of matrix-valued divergence-free kernels in the PBDW framework.

\section{Error analysis}
\label{sec:analysis}
We present an error analysis for general linear functionals in the presence of noise, based on the identification of two sources:
the first source is related to the finite number of measurements, while the second source is related to the presence of noise.
For simplicity of notation, we consider scalar problems; however, the analysis can be trivially extended to vector-valued fields.
Given the measurements $\{ y_m = \ell_m^o(u^{\rm true}) + \epsilon_m  \}_{m=1}^M$, 
we define the vector 
$\mathbf{u}^{\rm opt} = [\tilde{\boldsymbol{\eta}}^{\rm opt}, \mathbf{z}^{\rm opt}] \in \mathbb{R}^{M+N}$ corresponding to the solution 
$u_{\xi=0}^{\rm opt}$ to  \eqref{eq:PBDW_two_field_formulation_noise_free}
fed with perfect observations $\{ y_m^{\rm true} = \ell_m^o(u^{\rm true}) \}_{m=1}^M$ (see  \eqref{eq:PBDW_algebraic_linear}).
On the other hand, we define  $\mathbf{u}_{\xi}^{\star} = [\tilde{\boldsymbol{\eta}}_{\xi}^{\star}, \mathbf{z}_{\xi}^{\star}]$ corresponding to the solution
${u}_{\xi}^{\star}$  to the stabilized formulation \eqref{eq:PBDW_two_field_formulation}, fed with  imperfect observations.

We first present the result for perfect measurements.

\begin{proposition}
\label{th:noise_free}
Let $y_m = \ell_m^o(u^{\rm true})$, $m=1,\ldots,M$, 
let $\beta_{N,M}>0$, and let $\mathcal{U}_M$ satisfy 
\eqref{eq:unisolvency_update}.
We denote by $u_{\xi=0}^{\rm opt}$ the solution to the PBDW formulation \eqref{eq:PBDW_two_field_formulation_noise_free}.
Then, the following estimates hold:
\begin{subequations}
\begin{equation}
\label{eq:estimate_noise_free}
\vertiii{ u^{\rm true} - u_{\xi=0}^{\rm opt} } 
\leq
\frac{1}{\beta_{N,M}}
\,
\inf_{z \in \mathcal{Z}_N} \, 
\inf_{  q \in \mathcal{U}_M \cap \mathcal{Z}_N^{\perp, \vertiii{\cdot} } } \,
\vertiii{  u^{\rm true}  - z - q   };
\end{equation}
\begin{equation}
\label{eq:estimate_noise_free_H1}
\|  u^{\rm true} - u_{\xi=0}^{\rm opt} \|
\leq
\frac{ \sqrt{4 + 6 \|  \mathcal{I}_M \|_{\mathcal{L}(\mathcal{X})}^2   }         }{\beta_{N,M}}
\,
\inf_{z \in \mathcal{Z}_N} \, 
\inf_{  q \in \mathcal{U}_M \cap \mathcal{Z}_N^{\perp, \vertiii{\cdot} } } \,
\|  u^{\rm true}  - z - q \|.
\end{equation}
\end{subequations}
\end{proposition}

\begin{proof}
Thanks to the  variational interpretation of the user-defined update
(see Proposition \ref{th:variation_empirical_update}), 
proof of \eqref{eq:estimate_noise_free} follows  
directly from \cite[Remark 2.12]{binev2017data}.
On the other hand, exploiting 
\eqref{eq:estimate_noise_free} and
\eqref{eq:estimate_new_product}, we find
$$
\begin{array}{ll}
\|  u^{\rm true} - u_{\xi=0}^{\rm opt} \|
&
\displaystyle{
\leq \sqrt{2} \vertiii{ u^{\rm true} - u_{\xi=0}^{\rm opt} }
\leq \frac{\sqrt{2}}{\beta_{N,M}}
\,
\inf_{z \in \mathcal{Z}_N} \, 
\inf_{  q \in \mathcal{U}_M \cap \mathcal{Z}_N^{\perp, \vertiii{\cdot} } } \,
\vertiii{  u^{\rm true}  - z - q   }}
\\[3mm]
&
\displaystyle{
\leq
\frac{ \sqrt{4 + 6 \|  \mathcal{I}_M \|_{\mathcal{L}(\mathcal{X})}^2   }         }{\beta_{N,M}}
\,
\inf_{z \in \mathcal{Z}_N} \, 
\inf_{  q \in \mathcal{U}_M \cap \mathcal{Z}_N^{\perp, \vertiii{\cdot} } } \,
\|  u^{\rm true}  - z - q   \|,}
\\
\end{array}
$$
which proves  \eqref{eq:estimate_noise_free_H1}.
\end{proof}

If we introduce the matrices and vectors
$$
\mathbb{A}(\xi) = 
\left[
\begin{array}{cc}
\left( \xi M  \mathbb{I} + \mathbb{L}_{\eta}  \mathbb{L}_{\eta}^{T} \right)  & 
 \mathbb{L}_{z}  \ \\[3mm]
\mathbb{L}_{z}^T   & 0   \\
\end{array}
\right];
\quad
\mathbf{e}  =
\left[
\begin{array}{c}
\boldsymbol{\epsilon} \\
\mathbf{0}\\
\end{array}
\right];
\quad
\boldsymbol{\epsilon} = [\epsilon_1,\ldots,\epsilon_M]^T;
$$
we obtain the following identity that links $\mathbf{u}_{\xi}^{\star}$ to $ \mathbf{u}^{\rm opt}$
\begin{equation}
\label{eq:key_identity_uxi}
\mathbb{A}(\xi) \mathbf{u}_{\xi}^{\star}
=
\mathbb{A}(0) \mathbf{u}^{\rm opt}
+
 \mathbf{e}
 =
 \mathbb{A}(\xi) \mathbf{u}^{\rm opt}
-
 \xi M 
 \left[  
\begin{array}{c}
\tilde{\boldsymbol{\eta}}^{\rm opt} \\
\mathbf{0}
\end{array} 
  \right]
 +
 \mathbf{e}.
\end{equation}
Identity \eqref{eq:key_identity_uxi} can be used to estimate the bias of our estimate, and the expected error.
For simplicity, we present below estimates for the errors in the coefficients, 
$\mathbf{u}_{\xi}^{\star} - \mathbf{u}^{\rm opt}$;
the estimates can also be extended to compute the state estimation error  in an integral norm of interest. We omit the details.

\begin{proposition}
\label{th:stochastic_error}
Suppose that $\epsilon_1,\ldots,\epsilon_M$ are independent realizations of
$\epsilon \sim (0,\sigma^2)$. Then, the following hold:
\begin{equation}
\label{eq:bias_xi}
\|   \mathbb{E}[  \mathbf{u}_{\xi}^{\star}   ]    -    \mathbf{u}^{\rm opt}   \|_2
\leq
\frac{\xi M}{s_{\rm min} ( \mathbb{A}(\xi)  )  }
\,
\|  \tilde{\boldsymbol{\eta}}^{\rm opt}  \|_2;
\end{equation}
and
\begin{equation}
\label{eq:variance_xi}
\mathbb{E} 
\left[
\|   \mathbf{u}_{\xi}^{\star} -  \mathbf{u}^{\rm opt}      \|_2^2
\right]
\leq
\left(
\frac{\xi M}{s_{\rm min} ( \mathbb{A}(\xi)  )  }
\,
\|  \tilde{\boldsymbol{\eta}}^{\rm opt}  \|_2
\right)^2
+
\sigma^2
{\rm trace} \left(
\mathbb{A}(\xi)^{-1}
\Sigma
\mathbb{A}(\xi)^{-T}
\right);
\end{equation}
where
$s_{\rm min}( \mathbb{A}(\xi) )$ denotes the minimum singular value of 
$\mathbb{A}(\xi)$, 
 $\Sigma = 
\left[
\begin{array}{cc}
 \mathbb{I} & 0   \\
0  & 0   \\
\end{array}
\right].$
\end{proposition}
\begin{proof}
Exploiting \eqref{eq:key_identity_uxi}, we find
$$
\mathbb{A}(\xi) \mathbf{u}_{\xi}^{\star}
 =
 \mathbb{A}(\xi) \mathbf{u}^{\rm opt}
-
 \xi M 
 \left[  
\begin{array}{c}
\tilde{\boldsymbol{\eta}}^{\rm opt} \\
\mathbf{0}
\end{array} 
  \right]
 +
 \mathbf{e}
 \,
 \Rightarrow
 \,
 \|   \mathbb{E}[  \mathbf{u}_{\xi}^{\star}   ]    -    \mathbf{u}^{\rm opt}   \|_2
\leq
\frac{\xi M}{s_{\rm min} ( \mathbb{A}(\xi)  )  }
\,
\|  \tilde{\boldsymbol{\eta}}^{\rm opt}  \|_2;
 $$
which is \eqref{eq:bias_xi}.

In order to show \eqref{eq:variance_xi}, we first observe that
(cf. \cite[Thm. C, Chapter 14.4]{rice2006mathematical}):
$$
\mathbb{E} \left[
\| 
\mathbb{E}[\mathbf{u}_{\xi}^{\star}  ]  -
\mathbf{u}_{\xi}^{\star}
\|_2^2
\right]
=
\sigma^2
{\rm trace} \left(
 \mathbb{A}(\xi)^{-1} \Sigma
  \mathbb{A}(\xi)^{-T} 
\right).
$$
Then, 
we find
$$
\begin{array}{ll}
\mathbb{E} 
\|   \mathbf{u}_{\xi}^{\star}  - \mathbf{u}^{\rm opt}  \|_2^2 
=
&
 \|  \mathbb{E}[\mathbf{u}_{\xi}^{\star}  ]  - \mathbf{u}^{\rm opt}  \|_2^2
 +
\mathbb{E} \left[
\|  \mathbb{E}[\mathbf{u}_{\xi}^{\star}  ]  - \mathbf{u}_{\xi}^{\star}  \|_2^2
\right]
\\[3mm]
&
\leq 
\left(
\frac{\xi M}{s_{\rm min} ( \mathbb{A}(\xi)  )  }
\,
\|  \tilde{\boldsymbol{\eta}}^{\rm opt}    \|_2
\right)^2
+
\sigma^2
{\rm trace} \left(
\mathbb{A}(\xi)^{-1}
\Sigma
\mathbb{A}(\xi)^{-1}
\right),
\\
\end{array}
$$
which is  \eqref{eq:variance_xi}.
\end{proof}

We observe that the bound for the mean squared error
\eqref{eq:variance_xi} is the sum of two contributions:
the former 
--- 
$\frac{\xi M}{s_{\rm min} ( \mathbb{A}(\xi)  )  } \,\|  \tilde{\boldsymbol{\eta}}^{\rm opt}  \|_2$
---
involves the accuracy of the background space,
measured by $\| \tilde{\boldsymbol{\eta}}^{\rm opt} \|_2$, 
and is monotonically increasing with $\xi$;
the latter 
---
$\sigma^2
{\rm trace} \left( \mathbb{A}(\xi)^{-1} \Sigma \mathbb{A}(\xi)^{-T} \right)$
---
involves the accuracy of the measurements, and is monotonically decreasing with $\xi$ according to our numerical experience.
Therefore, the optimal value of $\xi$ depends on the ratio between ``measurement inaccuracy"
and ``model  inaccuracy"
$\frac{\sigma}{\|  \tilde{\boldsymbol{\eta}}^{\rm opt}  \|_2}$:
 if $u^{\rm true} \in \mathcal{Z}_N$, estimate \eqref{eq:variance_xi} suggests to pick   $\xi \to \infty$; 
 if $\sigma=0$, estimate \eqref{eq:variance_xi} suggests to pick  
 $\xi \to 0^+$. 
 Since in practice estimates of the ratio $\frac{\sigma}{\|  \tilde{\boldsymbol{\eta}}^{\rm opt}  \|_2}$
 are rarely available, our analysis suggests that the value of $\xi$ should be chosen adaptively using (cross-)validation techniques such as the one discussed in section \ref{sec:choice_xi}.
These observations are in agreement with \cite[Remark 3.3]{taddei_APBDW}.

\begin{remark}
(\textbf{extension  to  complex-valued problems})
Proposition \ref{th:stochastic_error} can  be trivially extended to the complex-valued case. Assuming that measurements are of the form
$$
y_m = \ell_m^o(u^{\rm true}) + \sigma_{\rm re} \epsilon_m^{\rm re}  + {\rm i} \sigma_{\rm im} \epsilon_m^{\rm im},
\qquad
\epsilon_m^{\rm re}, \epsilon_m^{\rm im} 
\overset{\rm i.i.d.}{\sim} (0,1), 
\quad
\sigma_{\rm re}, \sigma_{\rm im}>0,
$$
we find the same estimate for the bias \eqref{eq:bias_xi}, and we find
\begin{equation}
\label{eq:variance_xi_complex}
\begin{array}{ll}
\mathbb{E}[ \|  \mathbf{u}_{\xi}^{\star} - \mathbf{u}^{\rm opt}    \|_2^2    ]
\leq
&
\displaystyle{
\left(
\frac{\xi M}{s_{\rm min}( \tilde{\mathbb{A}}(\xi) )  }
 \|  \widetilde{\boldsymbol{\eta}}^{\rm opt}    \|_2
\right)^2
+
\sigma_{\rm re}^{2}
{\rm trace} 
\left( 
\tilde{\mathbb{A}}(\xi)^{-1} \Sigma^{\rm re}  \tilde{\mathbb{A}}(\xi)^{-T}
\right)
}
\\[3mm]
&
\displaystyle{
+
\sigma_{\rm im}^{2}
{\rm trace} 
\left( 
\tilde{\mathbb{A}}(\xi)^{-1} \Sigma^{\rm im}  \tilde{\mathbb{A}}(\xi)^{-T}
\right);
}
\end{array}
\end{equation}
where 
$\tilde{\mathbb{A}}(\xi) = \left[ 
\begin{array}{cc}
{\rm Re}[{\mathbb{A}}(\xi)] & -{\rm Im}[{\mathbb{A}}(\xi)] \\
{\rm Im}[{\mathbb{A}}(\xi)] & {\rm Re}[{\mathbb{A}}(\xi)] \\
\end{array}
 \right]$,
 $\Sigma^{\rm re} = \left[ 
\begin{array}{cc}
\Sigma & 0 \\
0 & 0\\
\end{array}
 \right]$,
 $\Sigma^{\rm im} = \left[ 
\begin{array}{cc}
0 & 0 \\
0 & \Sigma \\
\end{array}
 \right]$.
\end{remark}

  \section{Application to Acoustics}
 \label{sec:acoustics}
 
 We illustrate the behavior of the PBDW formulation presented in this paper through the vehicle of two acoustic Helmholtz problems.
Since the objective of this work is to propose user-defined  update spaces that improve convergence with respect to the number of observations ($M$-convergence), we do not study the effect of the primary approximation provided by the background $\mathcal{Z}_N$: we refer to the PBDW literature for further results concerning the effect of the primary approximation.

\subsection{A two-dimensional model problem}
The model problem is the same considered in 
\cite[Section 3]{maday2015parameterized}, and \cite[Section 5]{taddei_APBDW}. 

\subsubsection{Problem definition}
Given the domain 
$\Omega = (0,1)^2$, we define the acoustic model problem: 
\begin{equation}
\label{eq:PBDW_numerics1_model}
\left\{
\begin{array}{ll}
-(1 + {\rm i} \epsilon\mu) \, \Delta u_g(\mu) \, - \, \mu^2 u_g(\mu) = \, \mu \left(  2 x_1^2 + e^{x_2} \right) \, + \, \mu g & \mbox{in} \; \Omega, \\[3mm]
\partial_n \, u_g(\mu) = 0 & \mbox{on} \, \partial \Omega, \\
\end{array}
\right.
\end{equation}
where 
${\rm i}$ is the imaginary unit, 
$\mu>0$ is the wave number, $\epsilon=10^{-2}$ is a fixed dissipation,  and $g \in L^2(\Omega)$ is a bias term that will be specified later.
Here, the parameter $\mu>0$ constitutes the anticipated, parametric uncertainty in the system, which might model our uncertainty in the speed of sound, while the function $g$ constitutes the unanticipated (non-parametric) uncertainty in the system.

To assess the performance of the PBDW formulation for various configurations,
we define the true field $u^{\rm true}$ as
the solution to \eqref{eq:PBDW_numerics1_model} for some $\mu^{\rm true} \in \mathcal{P}^{\rm bk} $ and for the following   choice  of the \emph{bias} $g$
\begin{subequations}
\label{eq:bias_numerics1_PBDW}
\begin{equation}
\label{eq:bias_a}
g:=  0.5 ( e^{-x_1} \, + \, 1.3 \cos ( 1.3 \pi x_2  )  ).
\end{equation}
On the other hand, we define the bk manifold as
\begin{equation}
\mathcal{M}^{\rm bk} := \; \{  u_{g=0}(\mu): \; \mu \in \mathcal{P}^{\rm bk}  \}.
\end{equation}
\end{subequations}

To measure  performance, we introduce  the  relative  $L^2$  and $H^1$ errors  averaged over $|\mathcal{P}_{\rm test}^{\rm bk}| = n_{\rm test}$  fields associated with different choices of the parameter $\mu$:
\begin{equation}
\label{eq:E_avg_APBDW}
E_{\rm avg}^{\rm rel}( n_{\rm test} )   := \frac{1}{  n_{\rm test} } \, \sum_{ \mu \in \mathcal{P}_{\rm test}^{\rm bk}  } \, \frac{\| u^{\rm true}(\mu) - u_{\xi}^{\star}(\mu)  \|_{\star} }{  \| u^{\rm true}(\mu)\|_{\star}  };
\qquad
\| \cdot \|_{\star} = 
\| \cdot \|_{L^2(\Omega)} \, {\rm or} \, \| \cdot \|_{H^1(\Omega)}.
\end{equation}
In all our tests, we consider $n_{\rm test}=10$ equispaced parameters in $\mathcal{P}^{\rm bk}$.

We model the (synthetic) observations by a Gaussian convolution 
with standard deviation $r_{\rm w}$:
\begin{subequations}
\label{eq:exp_observations}
\begin{equation}
\ell_m^o(v) 
= {\rm Gauss} (v, x_m^{\rm obs}; r_{\rm w})
=
C(x_m^{\rm obs}) \, \int_{\Omega} \, e^{-\frac{1}{2 r_{\rm w}^2 } \| x - x_m^{\rm obs} \|_2^2  } \, v(x ) \, dx,
\quad
m=1,\ldots,M;
\end{equation}
where $C(x_m^{\rm obs}) $ is a normalization constant such that $\ell_m^o(1)  = 1$, and
 $x_m^{\rm obs}$ denotes the transducer location.
 In order to simulate noisy measurements, we corrupt the synthetic measurements with homoscedastic random noise:
 \begin{equation}
 y_{\ell} =  \ell_m^o(u^{\rm true}) + 
\sigma^{\rm re} \epsilon_{\ell}^{\rm re} +  {\rm i} \sigma^{\rm im}  \epsilon_{\ell}^{\rm im},
\qquad
\epsilon_{\ell}^{\rm re} , \epsilon_{\ell}^{\rm im}   \,
\overset{ \rm iid}{\sim} 
\,
\mathcal{N}(0,1);
 \end{equation}
 where 
 \begin{equation}
 \sigma^{\rm re}  ={\rm  SNR} \times  {\rm std} \left(  \{  {\rm Re}[    \ell_m^o(u^{\rm true})   ]   \}_{m=1}^M   \right),
\quad
 \sigma^{\rm im}  = 
{\rm  SNR} \times 
 {\rm std} \left(  \{  {\rm Im}[    \ell_m^o(u^{\rm true})   ]   \}_{m=1}^M   \right),
  \end{equation}
  and ${\rm  SNR}$ denotes the signal-to-noise ratio in the measurements.
\end{subequations}

\subsubsection{PBDW spaces}

We introduce the ambient space $\mathcal{X}= H^1(\Omega)$ endowed with the inner product:
\begin{equation}
\label{eq:inner_product}
(u,v) = \int_{\Omega} \, u \bar{v} + \nabla u \cdot \nabla \bar{v} \, dx.
\end{equation}
Note that $\bar{(\cdot)}$ denotes the complex conjugate of $(\cdot)$. The background space 
$\mathcal{Z}_N$ is built using the Weak-Greedy algorithm (see, e.g., \cite{rozza2008reduced}):
we refer to \cite{maday2015parameterized} for further details.

As regards the update space, we consider the variational update associated with the inner product \eqref{eq:inner_product}, and the user-defined updates 
$\mathcal{U}_M = {\rm span} \{  \phi (   \lambda \|  \cdot - x_m^{\rm obs} \|_2   )   \}_{m=1}^M$ for 
$$
\left\{
\begin{array}{ll}
\displaystyle{  \phi(r) = \frac{1}{(1+r^2)^2} } & {\rm inverse \, multiquadrics}; \\[3mm]
\displaystyle{  \phi(r) =  (1-r)_+^4 (4r + 1)  } & {\rm csRBF}. \\
\end{array}
\right.
$$
The observation centers $\{  x_m^{\rm obs} \}_{m=1}^M$ are chosen according to  Algorithm \ref{SGreedy_plus}.
As regards the choice of the kernel scale, we here set  $\lambda=1$ for inverse multiquadrics, and $\lambda=2$ for csRBF. We remark that the optimal value for the kernel scale parameter $\lambda$ strongly depends on the number of measurements, and also on the characteristic length-scale of the field $u^{\rm true} - z_{\xi}^{\star}$: therefore, adaptation of the parameter $\lambda$ might improve performance, particularly for large values of $M$.

\subsubsection{Numerical results}

In Figures \ref{fig:PBDW_Sgreedy}(a)-(b)-(c),
we show the behavior of the inf-sup constant $\beta_{N,M}$ with $M$ 
 for the three different choices of the update space (and thus for the three choices of the norm $\vertiii{\cdot}$),
$r_{\rm w}=0.01$ and $N=6$, and for 
 centers $\{ x_m^{\rm obs} \}_{m=1}^{M}$ selected by the SGreedy algorithm with $tol=0.6$;
 to measure performance, we compare the results with the ones obtained using  randomly-generated centers. For the second choice, we average over $35$ different random choices of the $M$ centers. 
 We observe that the SGreedy selection of the sensor locations leads to a more stable formulation compared to randomly-generated points.
 In Figures \ref{fig:PBDW_Sgreedy}(d)-(e)-(f),
 we  show the centers $\{ x_m^{\rm obs} \}_{m=1}^{M}$ selected by the SGreedy algorithm for the three different choices of the update space: 
for all the three choices of the update space the Greedy procedure selects most points along the principal diagonal $(0,0) \to (1,1)$.  

\begin{figure}[h!]
\centering
\subfloat[ $H^1$] {\includegraphics[width=0.32\textwidth]
 {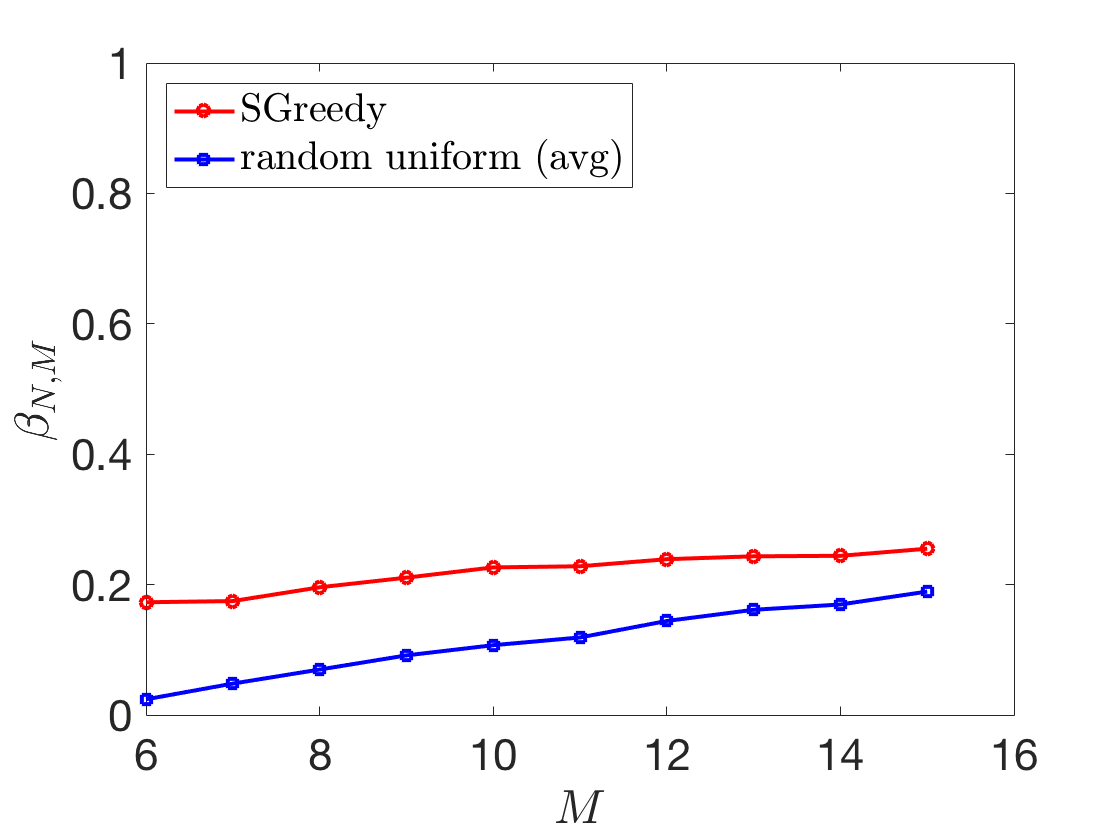}}
 ~
\subfloat[ csRBF ] {\includegraphics[width=0.32\textwidth]
 {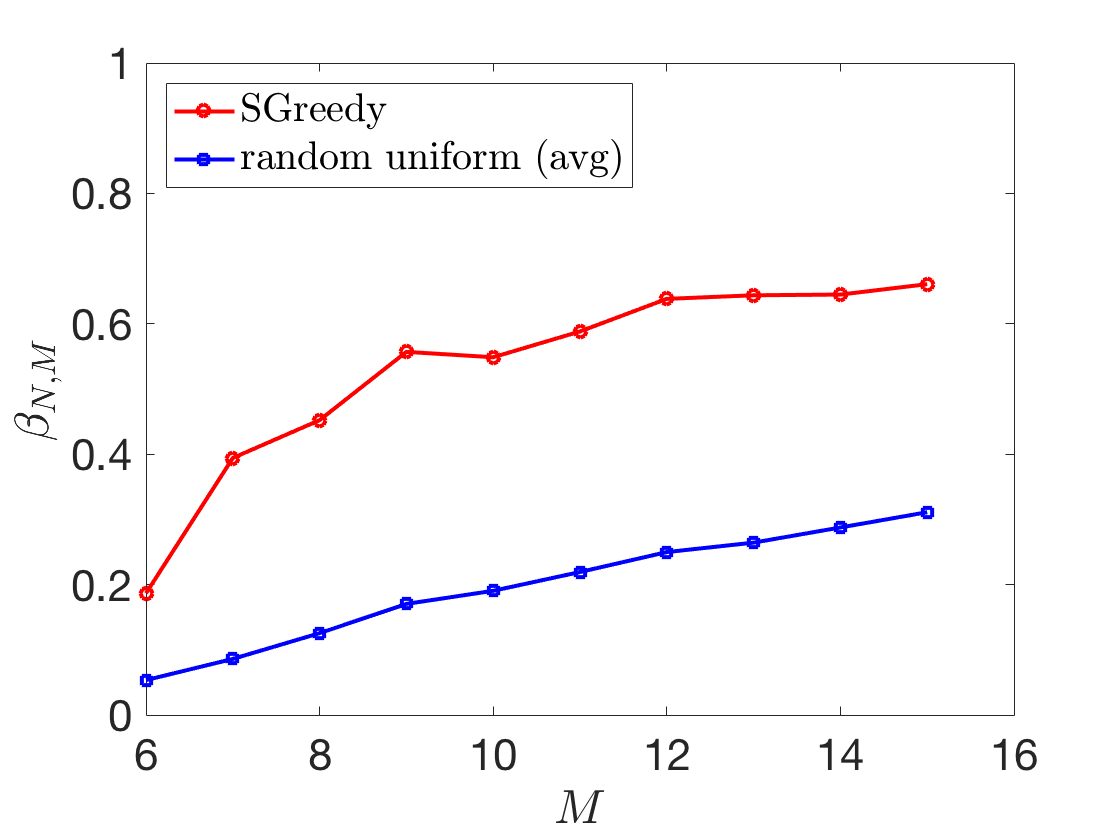}}
 ~
\subfloat[ inverse multiquadrics] {\includegraphics[width=0.32\textwidth]
 {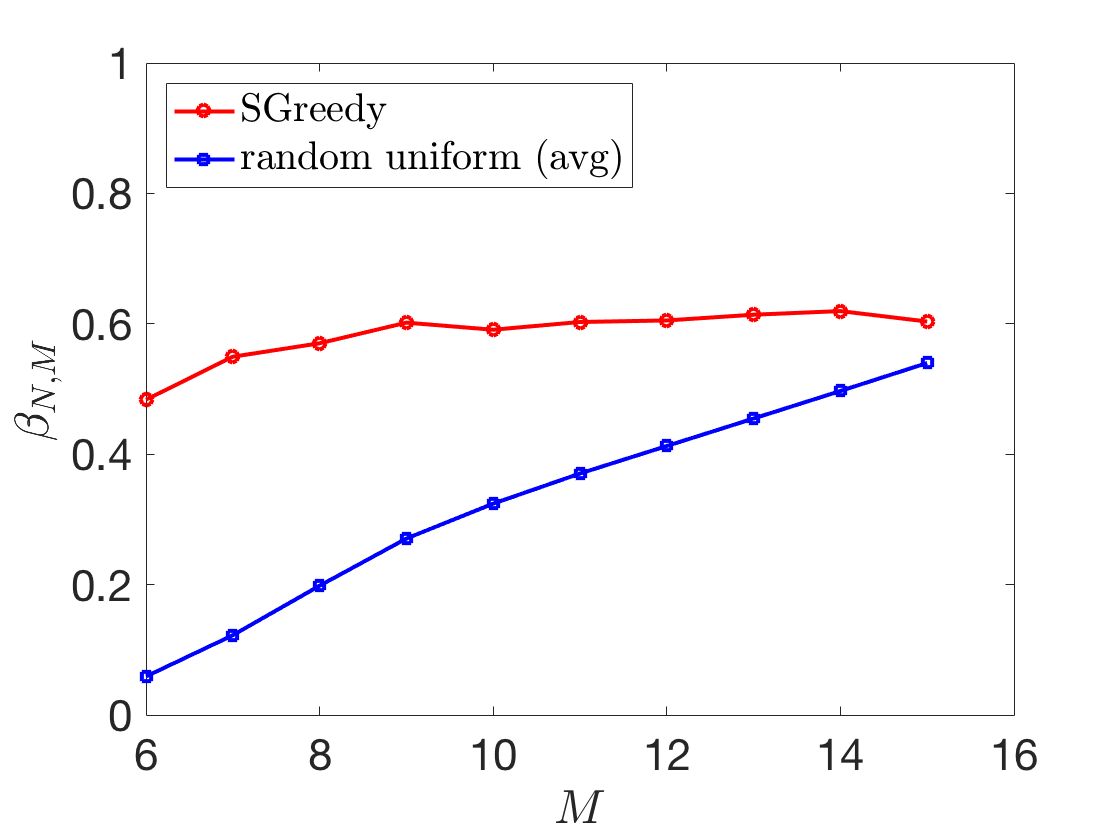}}
  
\subfloat[ $H^1$] {\includegraphics[width=0.32\textwidth]
 {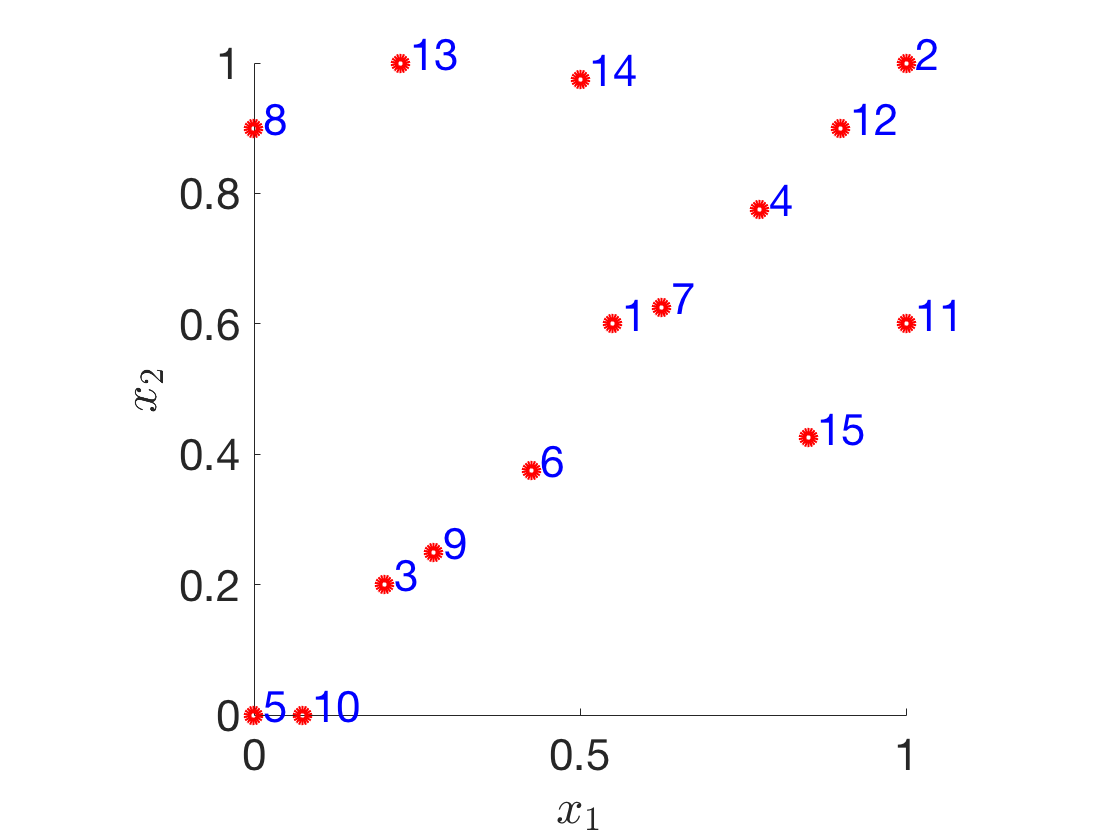}}
 ~
\subfloat[ csRBF ] {\includegraphics[width=0.32\textwidth]
 {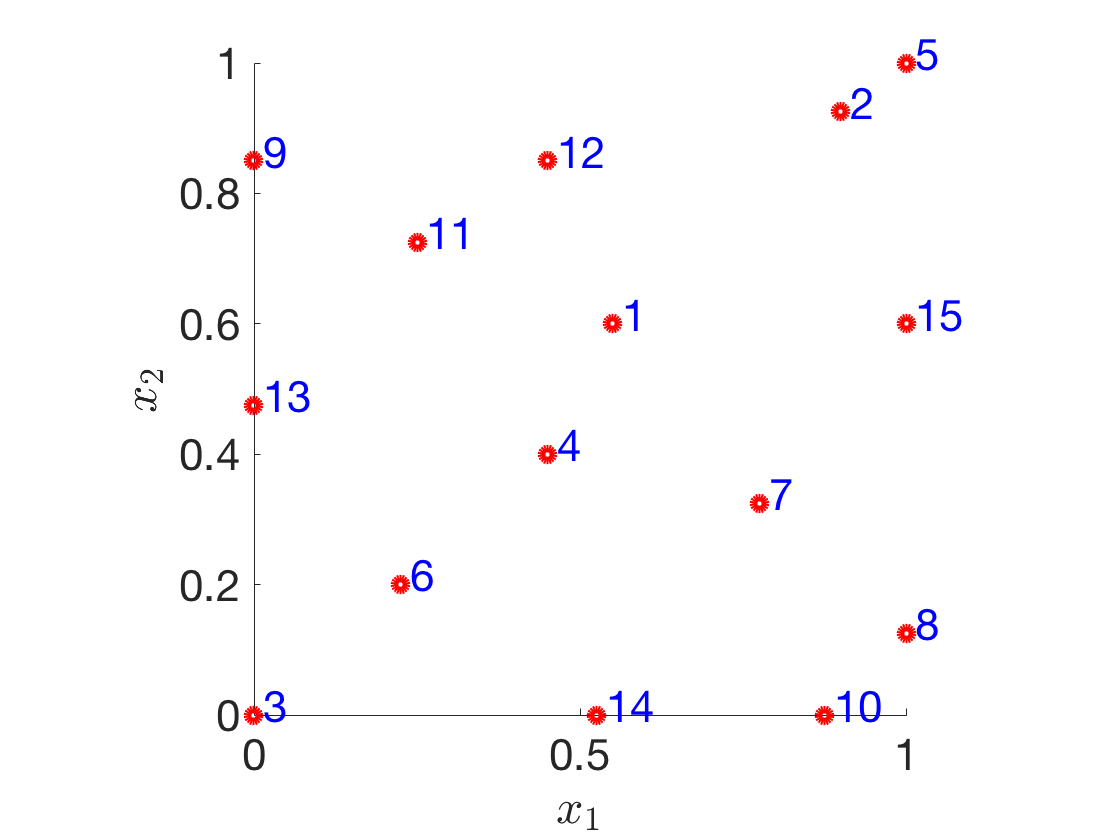}}
 ~
\subfloat[ inverse multiquadrics] {\includegraphics[width=0.32\textwidth]
 {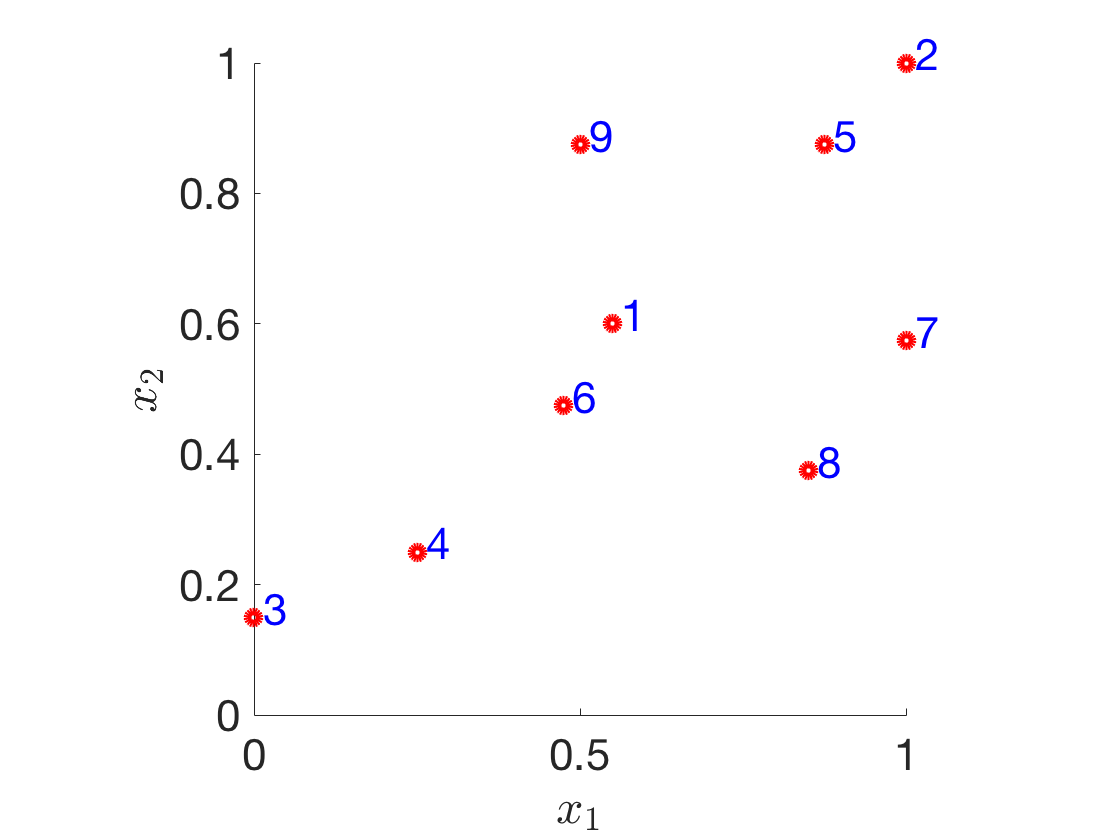}}

\caption{  Application to a  two-dimensional acoustic problem:
 application of the SGreedy algorithm ($N=6$, $r_{\rm w}=0.01$, $tol = 0.6$). Results for random centers are averaged over $35$ random trials.
  }
 \label{fig:PBDW_Sgreedy}
\end{figure}

In Figure \ref{fig:PBDW_Sgreedy_long}, we show the behavior of $\beta_{N=6,M}$ fo three
choices of the update space considered, for $M=6,\ldots,150$.
We observe that, for the user-defined update, $\beta_{N,M}$ is not monotonic increasing with $M$.
Nevertheless, we do not observe pathologic behaviors of the inf-sup constant as $M$ increases.

\begin{figure}[h!]
\centering
\subfloat[ $H^1$] {\includegraphics[width=0.32\textwidth]
 {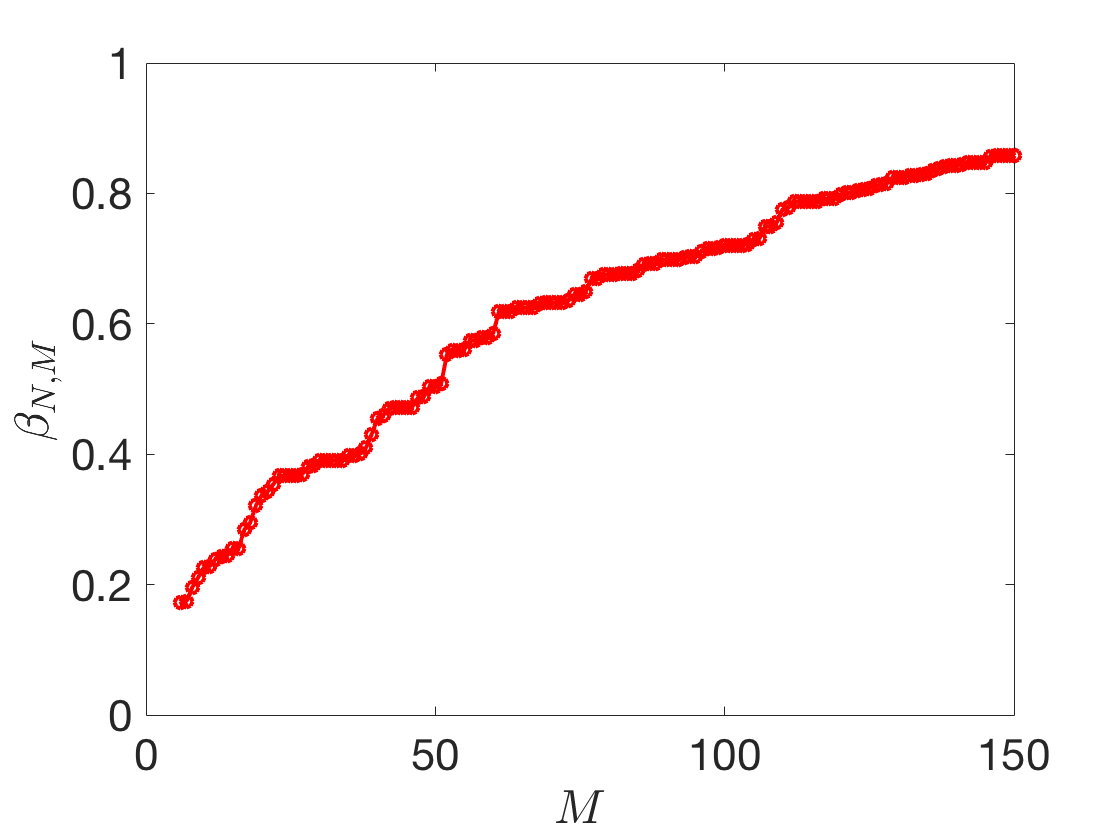}}
 ~
\subfloat[ csRBF ] {\includegraphics[width=0.32\textwidth]
 {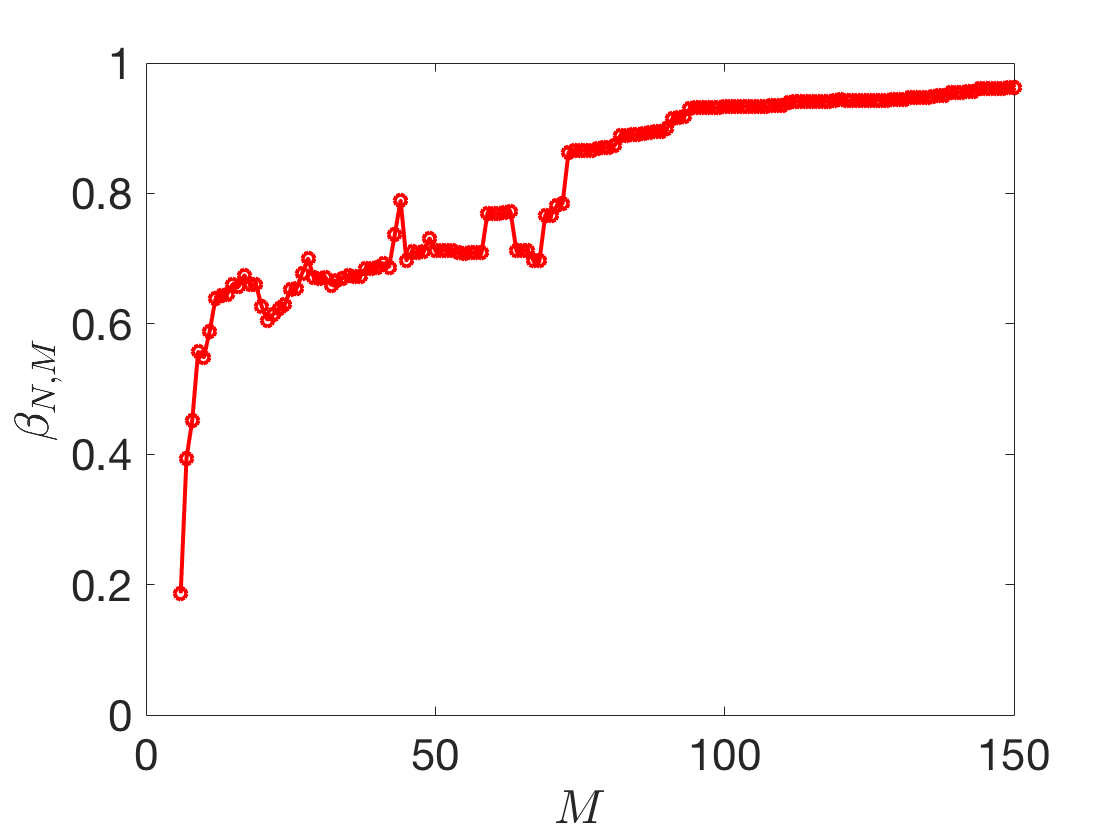}}
 ~
\subfloat[ inverse multiquadrics] {\includegraphics[width=0.32\textwidth]
 {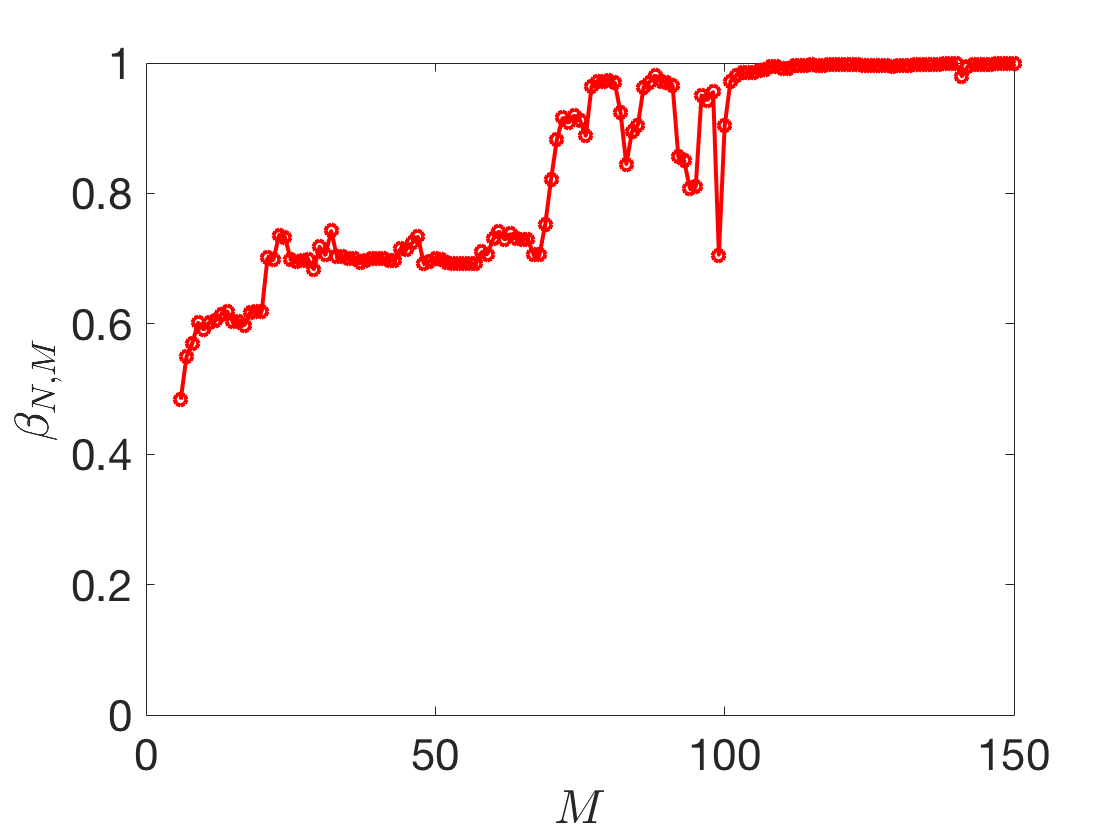}}
  
\caption{  Application to a  two-dimensional acoustic problem:
behavior of  $\beta_{N,M}$ for $M=6,\ldots,150$, for three choices of the update
($N=6$, $r_{\rm w}=0.01$, $tol = 0.6$). }
 \label{fig:PBDW_Sgreedy_long}
\end{figure}

In Figure \ref{fig:PBDW_calIM}, we show the behavior of the Lebesgue constant $\| \mathcal{I}_M \|_{\mathcal{L}(\mathcal{X})}$ defined in \eqref{eq:estimate_new_product}, for the three choices of the update. As expected for the variational update space, the Lebesgue constant is equal to one.
We also observe that $\|  \mathcal{I}_M \|_{\mathcal{L}(\mathcal{X})}$ is significantly larger for inverse multiquadrics than for csRBF: recalling estimate \eqref{eq:estimate_noise_free_H1}, the use of inverse multiquadrics (and more in general of smooth kernels) is appropriate only for smooth fields.

\begin{figure}[h!]
\centering
\subfloat[ $H^1$] {\includegraphics[width=0.32\textwidth]
 {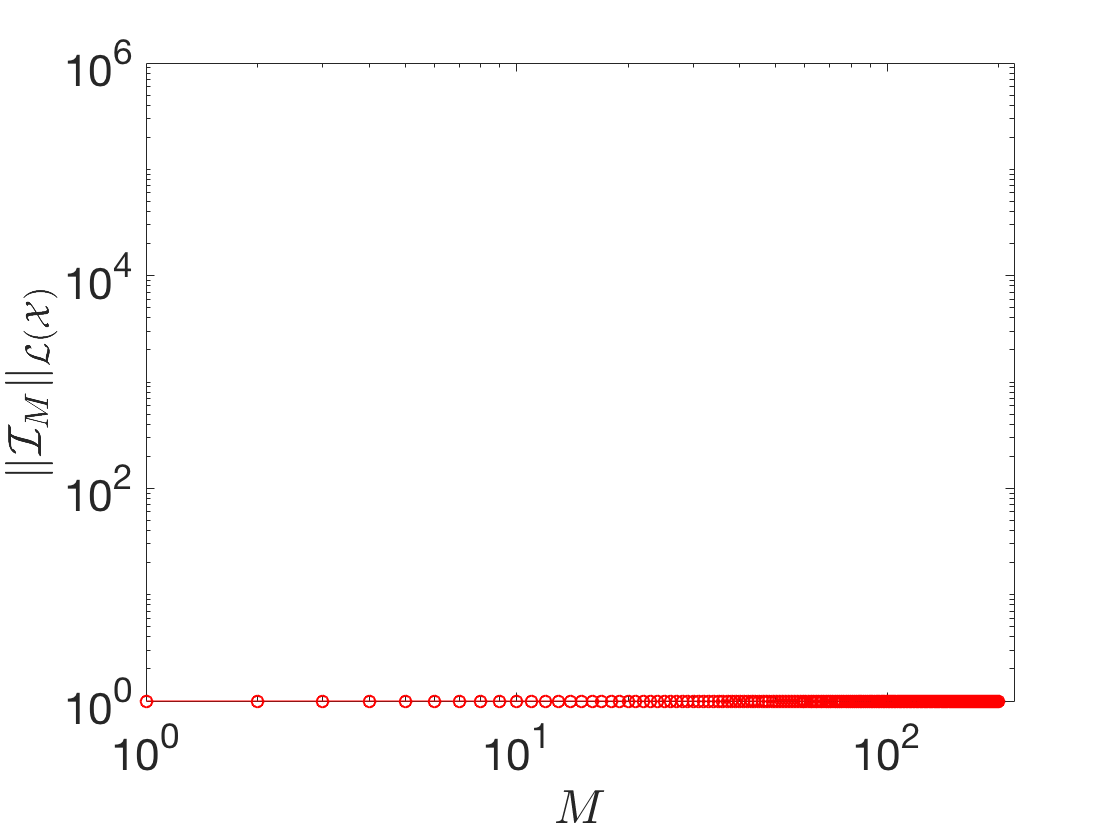}}
 ~
\subfloat[ csRBF ] {\includegraphics[width=0.32\textwidth]
 {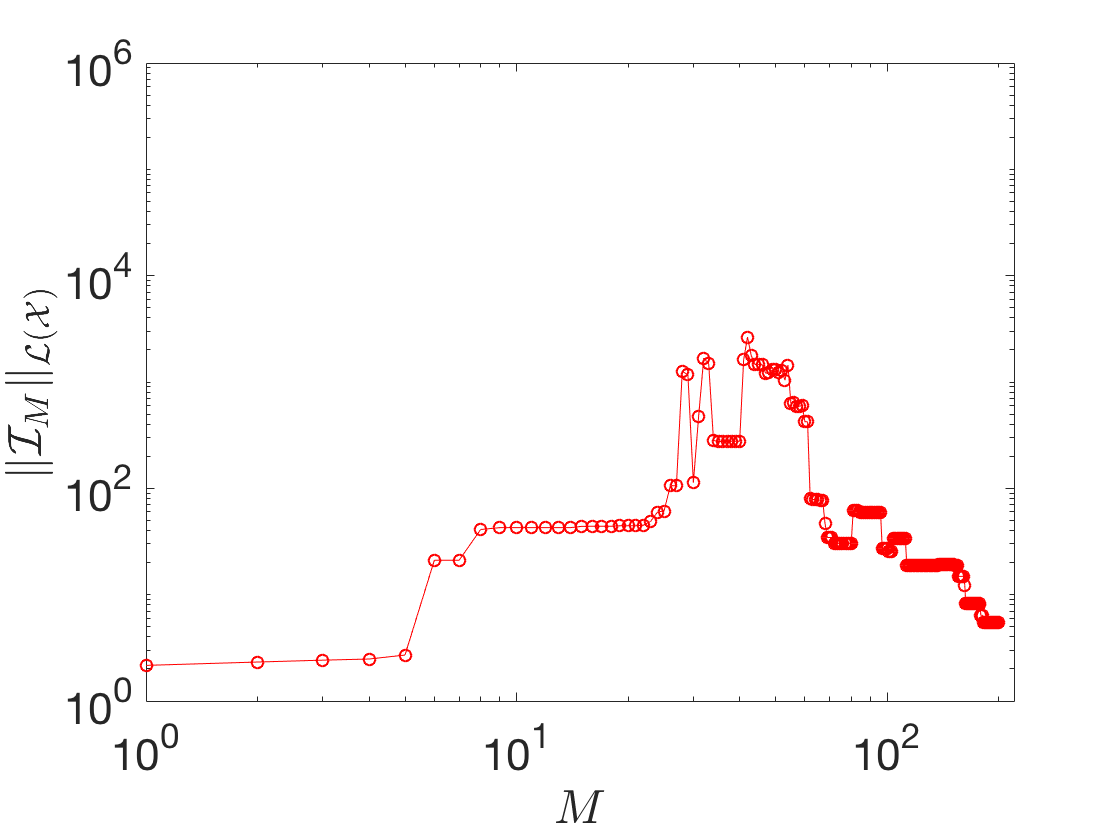}}
 ~
\subfloat[ inverse multiquadrics] {\includegraphics[width=0.32\textwidth]
 {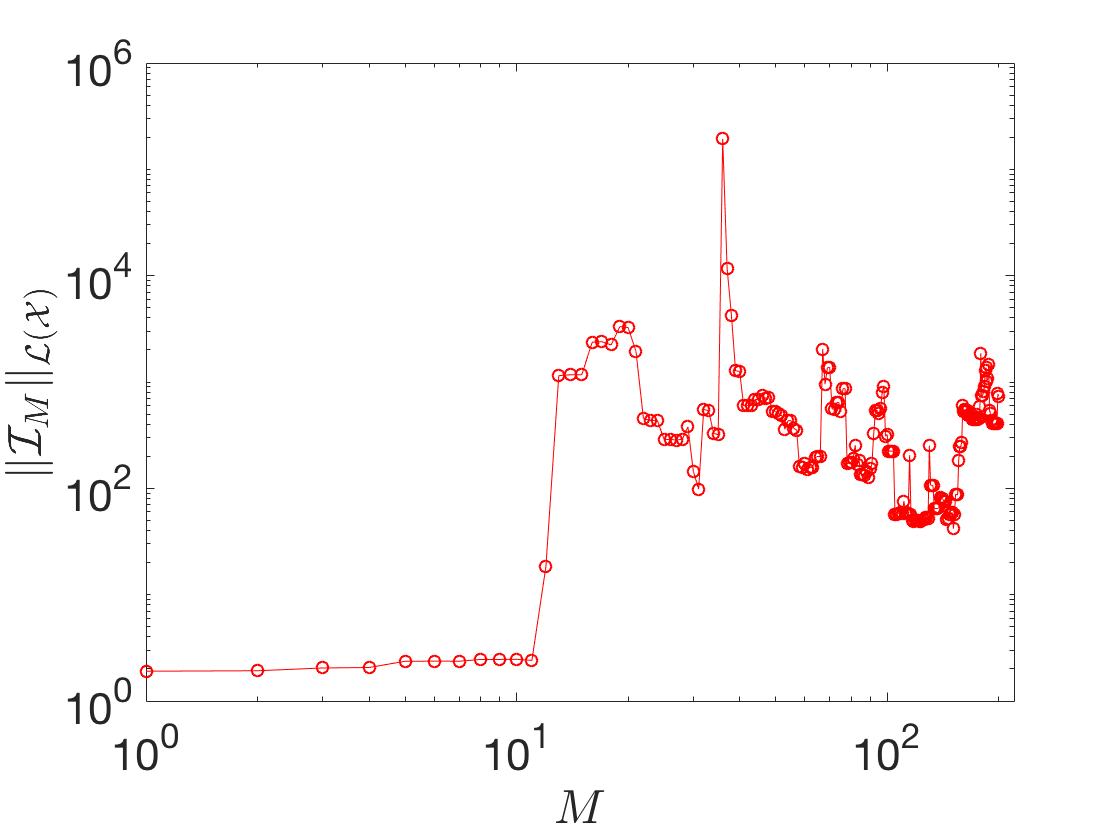}}
  
\caption{  Application to a  two-dimensional acoustic problem:
behavior of $\| \mathcal{I}_M \|_{\mathcal{L}(\mathcal{X})}$ ($N=6$, $r_{\rm w}=0.01$, $tol = 0.6$).}
 \label{fig:PBDW_calIM}
\end{figure} 

Figure \ref{fig:PBDW_M_conv} shows the behavior of the relative error $E_{\rm avg}^{\rm rel}$
\eqref{eq:E_avg_APBDW} with $M$ for perfect observations, $N=6$, and two values of $r_{\rm w}$ in \eqref{eq:exp_observations}. We observe that the user-defined update based on inverse multiquadrics leads to more accurate state estimates compared to the standard $H^1$ PBDW, particularly for $r_{\rm w}=0.01$. 
We further observe that  inverse multiquadrics outperform csRBF in all tests considered:
since the true state is smooth in $\Omega$,
the  inverse multiquadric kernel ---  which is  $C^{\infty}$ --- exhibits superior approximation properties compared to the $C^2$  csRBF kernel (see \cite[Chapter 11]{wendland2004scattered}).

\begin{figure}[h!]
\centering
\subfloat[ $r_{\rm w} = 0.01$ ] {\includegraphics[width=0.33\textwidth]
 {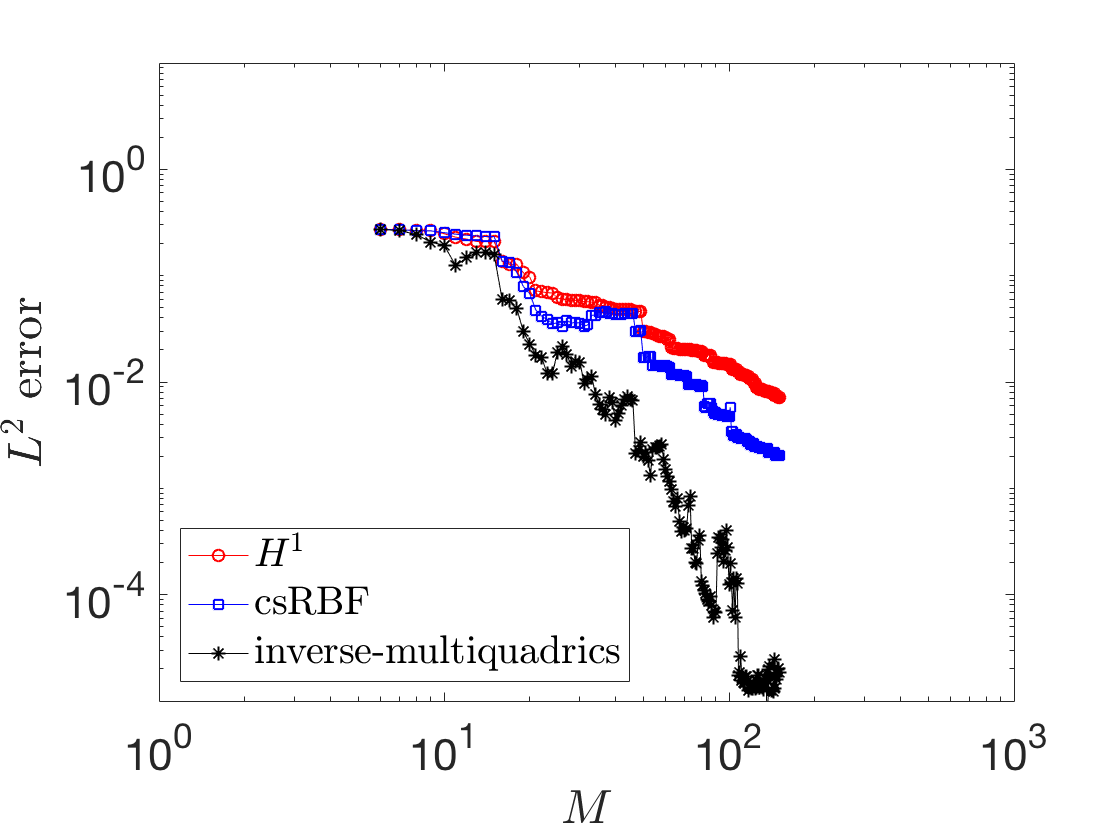}}
 ~
\subfloat[ $r_{\rm w} = 0.01$ ] {\includegraphics[width=0.33\textwidth]
 {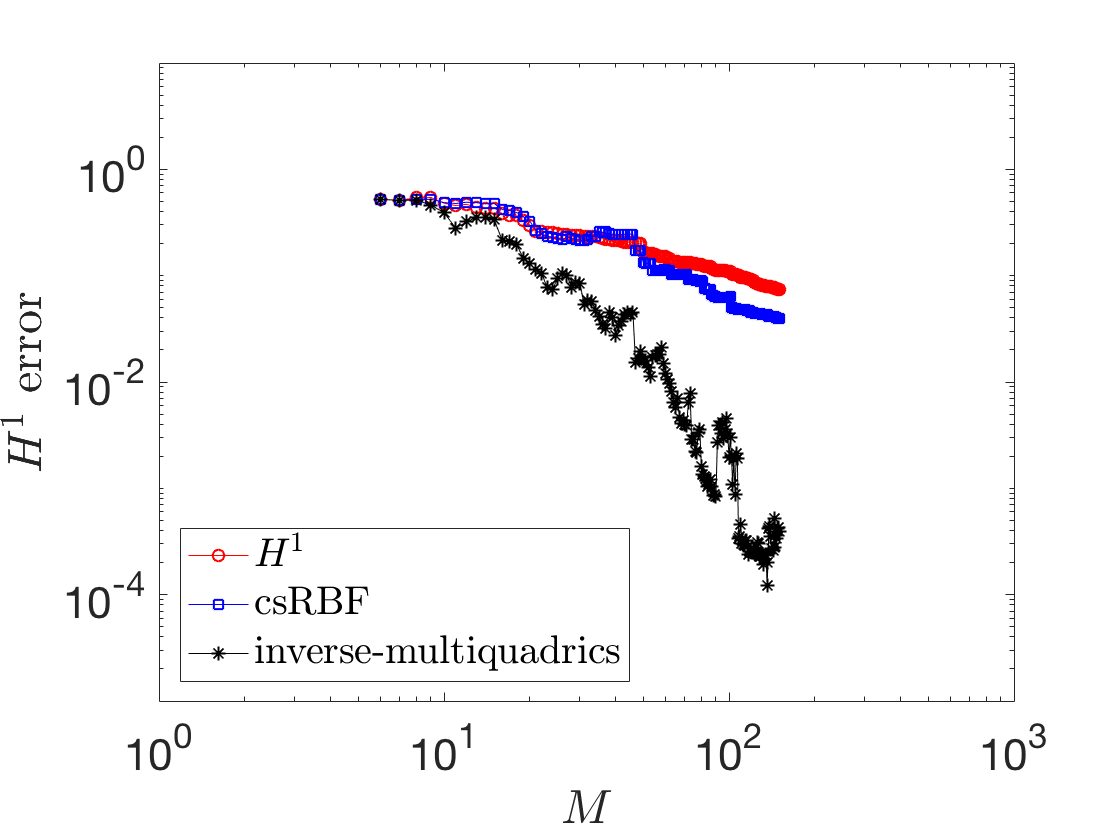}}

\subfloat[ $r_{\rm w} = 0.05$ ] {\includegraphics[width=0.33\textwidth]
 {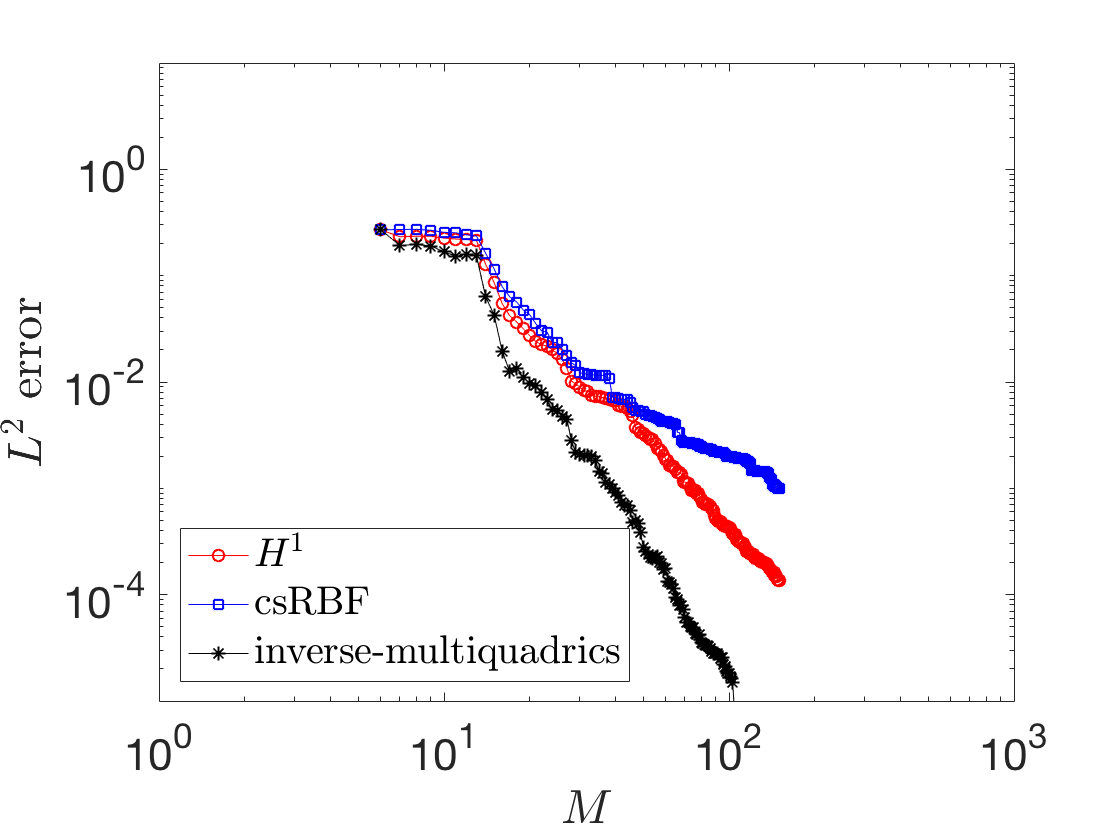}}
 ~
\subfloat[ $r_{\rm w} = 0.05$ ] {\includegraphics[width=0.33\textwidth]
 {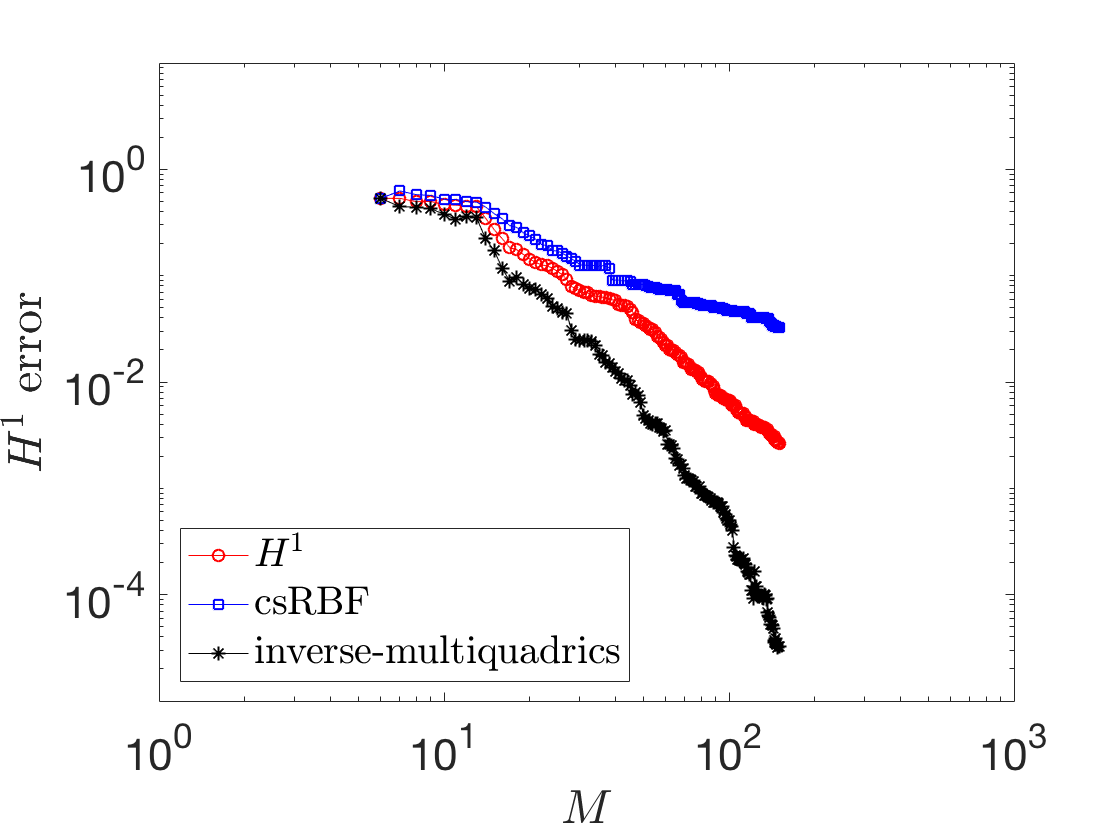}}

    \caption{  Application to a  two-dimensional acoustic problem:
   $M$ convergence for three different choices of the update space for perfect observations ($N=6$, ${\rm SNR}=0$, $tol = 0.6$).
  }
 \label{fig:PBDW_M_conv}
\end{figure} 

Figure \ref{fig:PBDW_M_conv_noisy} shows results for the behavior of the $L^2$ relative error $E_{\rm avg}^{\rm rel}$ with $M$ for imperfect observations for the user-defined (inverse multiquadrics) update and for the variational $H^1$ update.
The error $E_{\rm avg}^{\rm rel}$ is averaged over $25$ realizations of the homoscedastic random noise.
More in detail, we here consider three different noise levels ${\rm SNR}$; we set $N=4$, and we consider $r_{\rm w}=0.01$ in \eqref{eq:exp_observations}. Furthermore, in order to select the value of $\xi$, we resort to the holdout validation procedure outlined in section 
\ref{sec:choice_xi} based on additional $I=M/2$ measurements, associated with uniformly-generated points in $\Omega$.
The rate of convergence with $M$ seems to weakly depend on the choice of the update;
nevertheless, also for noisy measurements, the use of inverse multiquadrics strongly improves performance.

\begin{figure}[h!]
\centering
\subfloat[ ${\rm SNR}= 0.01$ ] {\includegraphics[width=0.32\textwidth]
 {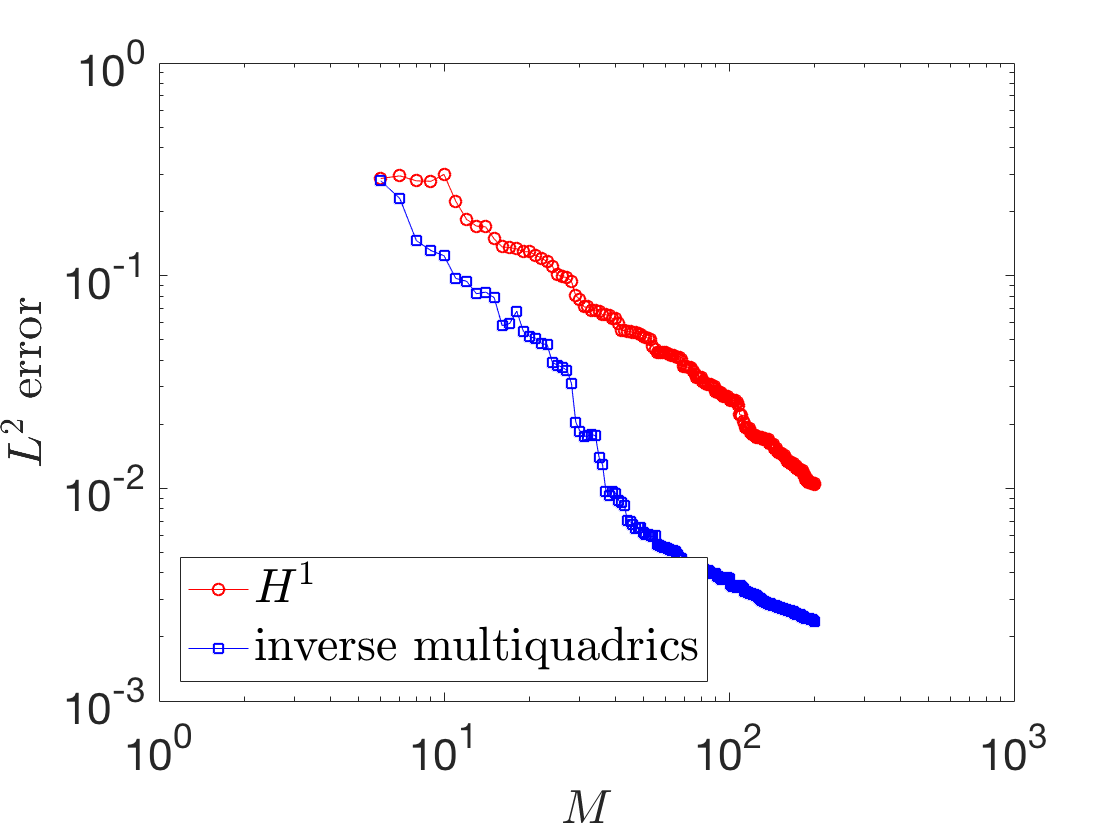}}
 ~
\subfloat[ ${\rm SNR}= 0.05$ ] {\includegraphics[width=0.32\textwidth]
 {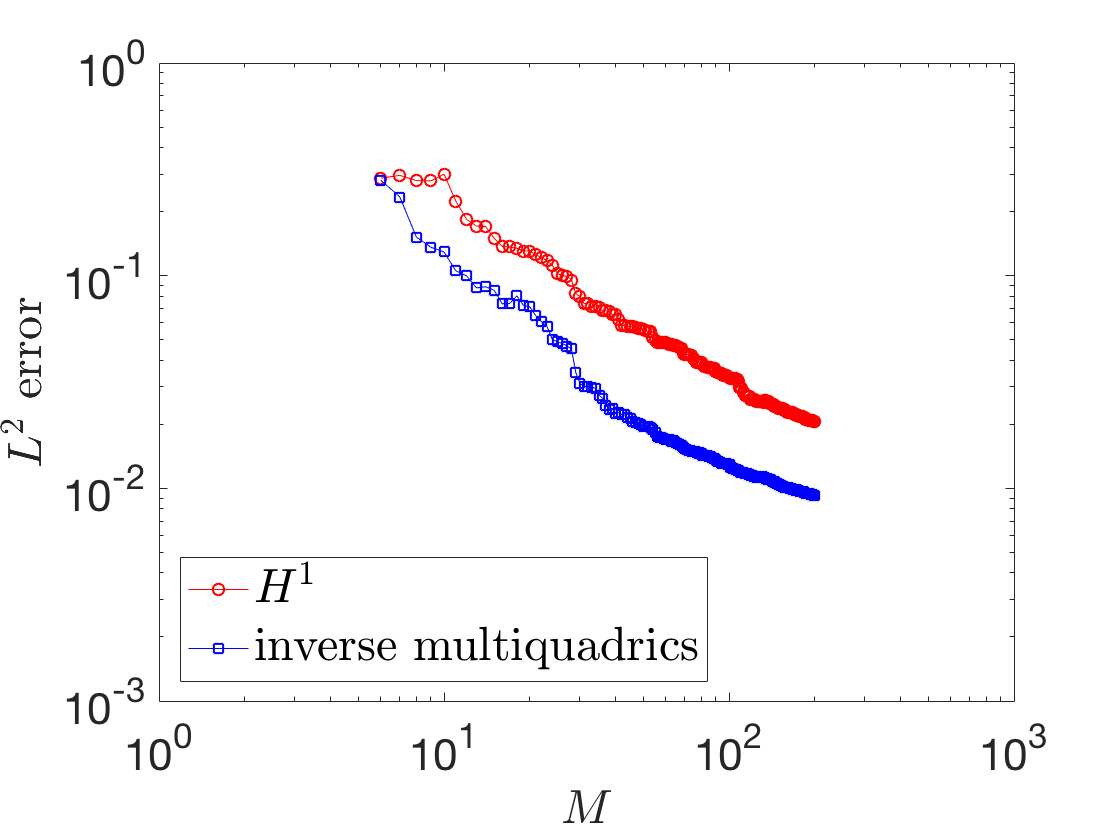}}
 ~
\subfloat[ ${\rm SNR}= 0.1$ ] {\includegraphics[width=0.32\textwidth]
 {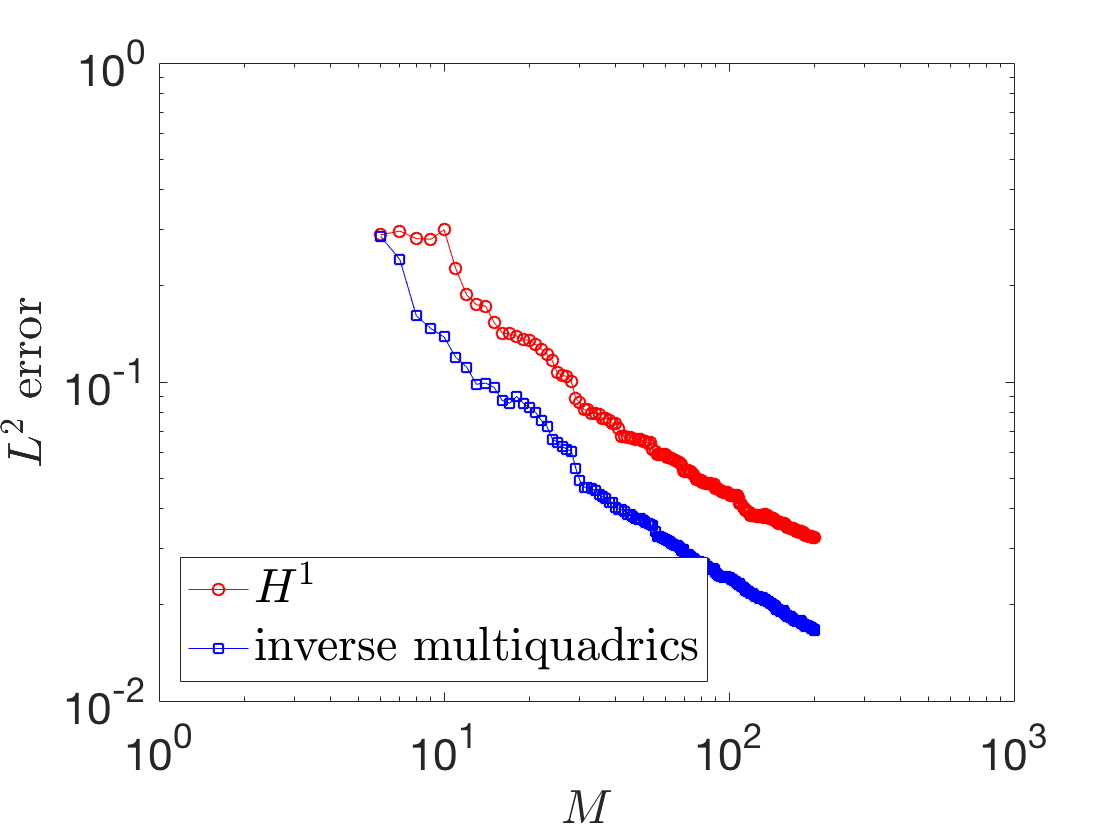}}

    \caption{  Application to a  two-dimensional acoustic problem:
   $M$ convergence for three different choices of the update space for noisy observations ($N=4$, $r_{\rm w} = 0.01$, $tol = 0.6$).
   Results are averaged over $25$  realizations of the random disturbances for each value of $M$.
  }
 \label{fig:PBDW_M_conv_noisy}
\end{figure}

In Figure \ref{fig:PBDW_CV}  we discuss the interpretation of the regularization parameter $\xi$.
Towards this end, we show the behavior of the mean squared error
$\widehat{\rm MSE}(I)$ introduced in section \ref{sec:choice_xi}
for two different choices of the true field 
($u^{\rm true} = u_{g}(\mu=5.8)$ and 
$u^{\rm true} = u_{g=0}(\mu=5.8)$), 
  two different noise levels, and
$N=5$, $M=100$, $I=50$.
Note that for $g=0$ $u^{\rm true}$ belongs to the bk manifold.
We observe that the optimal value of $\xi$ increases as the noise increases, and decreases as the best-fit error increases:
this is in good agreement with the discussion in section \ref{sec:analysis}, and also with the results in 
\cite[Figure 4]{taddei_APBDW}
for pointwise measurements.
We further observe that,
as  anticipated in  section \ref{sec:choice_xi},
the error estimator is in good qualitative agreement with the true error
$\| u^{\rm true} - u_{\xi}^{\star} \|_{L^2(\Omega)}^2$: therefore, it can be used to select the optimal value of $\xi$.

\begin{figure}[h!]
\centering
\subfloat[ ${\rm SNR}= 0.05$, $g=0$ ] {\includegraphics[width=0.32\textwidth]
 {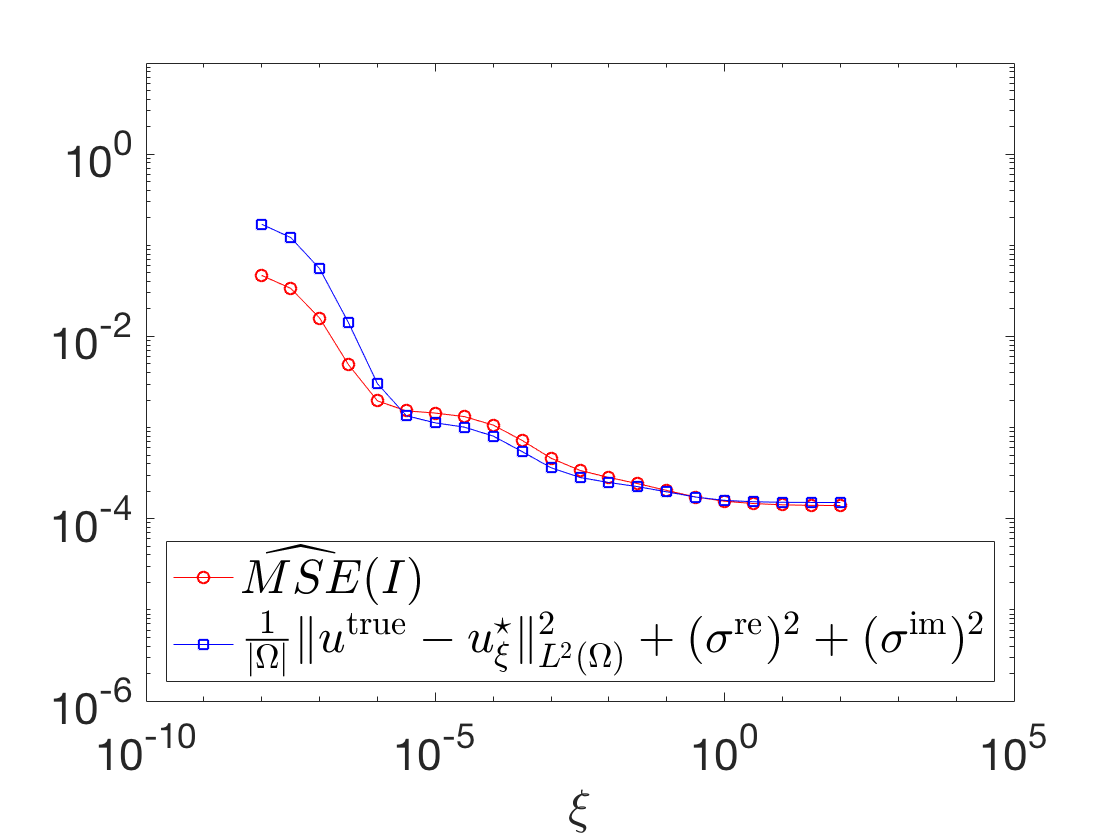}}
 ~
\subfloat[ ${\rm SNR}= 0.05$,  $g=\eqref{eq:bias_a}$ ] {\includegraphics[width=0.32\textwidth]
 {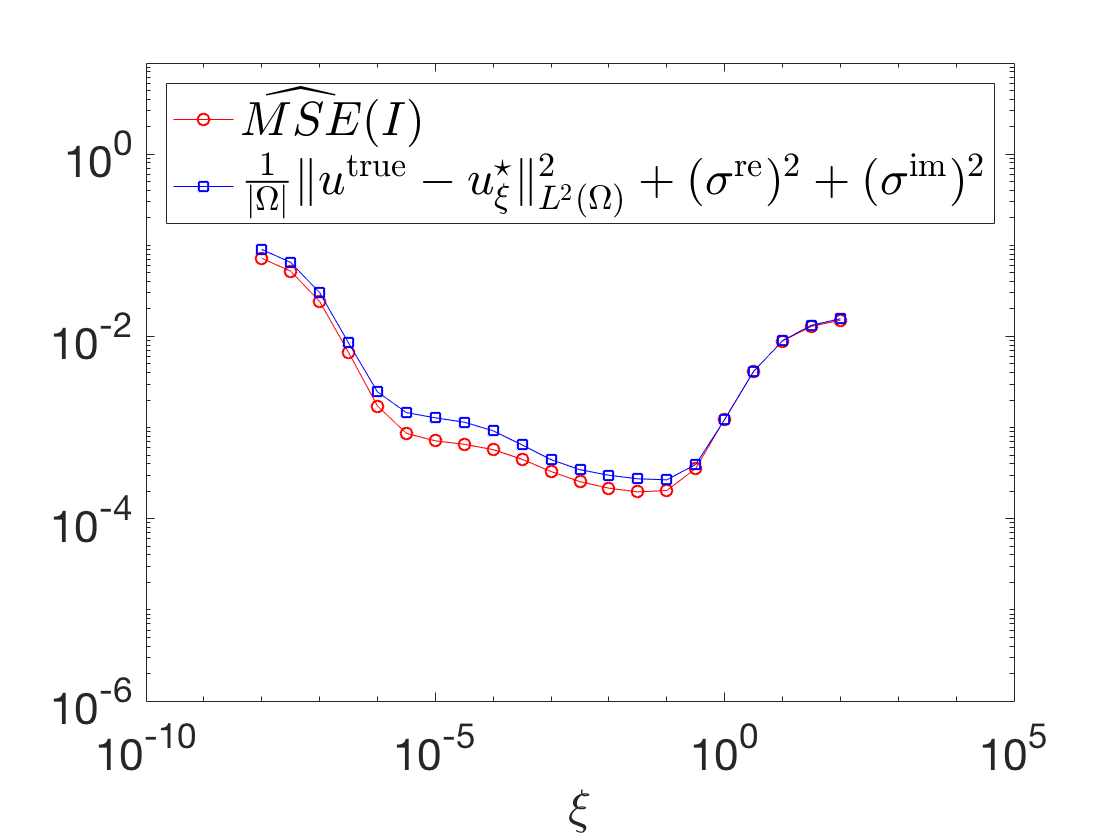}}

\subfloat[ ${\rm SNR}= 0.3$, $g=0$ ] {\includegraphics[width=0.32\textwidth]
 {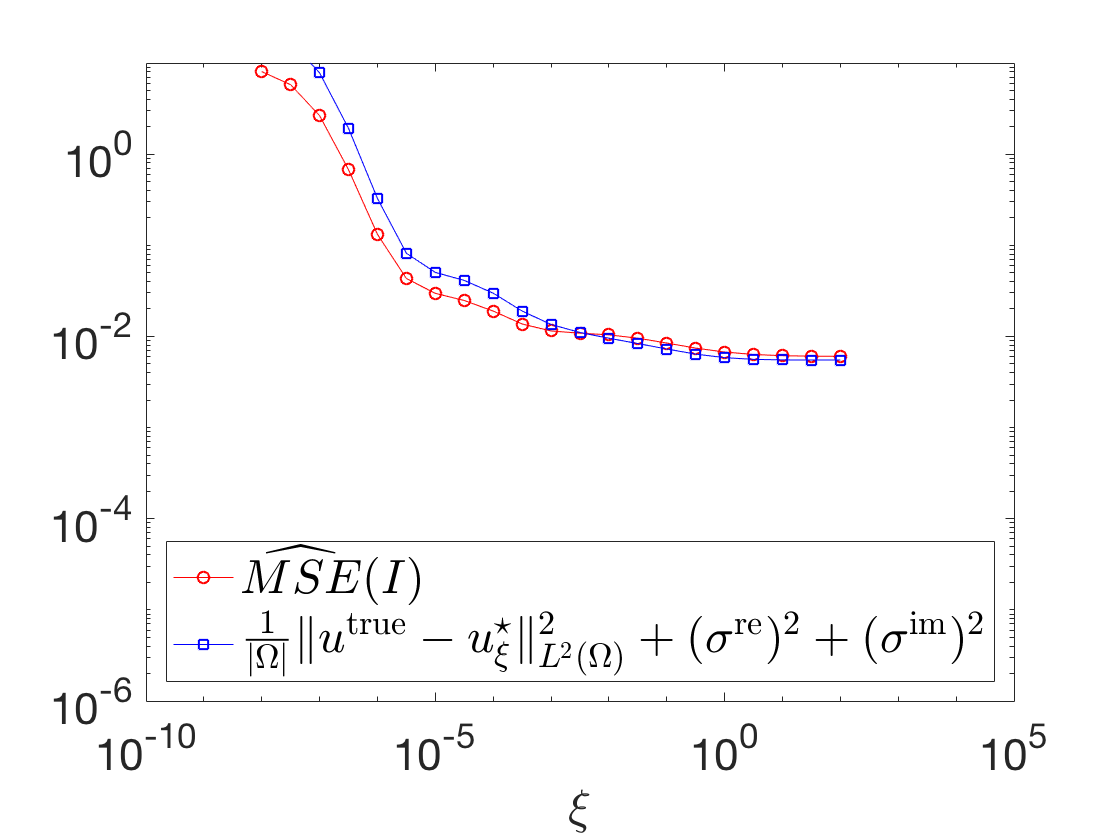}}
 ~
\subfloat[ ${\rm SNR}= 0.3$,  $g=\eqref{eq:bias_a}$ ] {\includegraphics[width=0.32\textwidth]
 {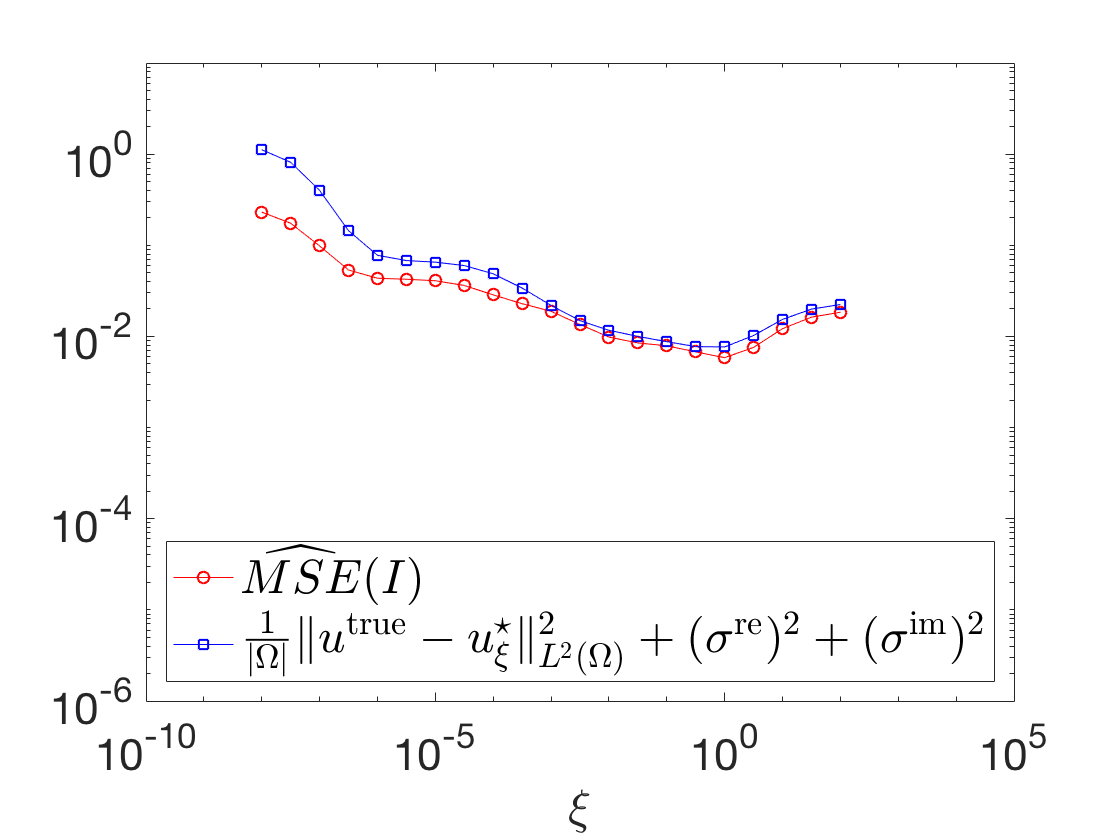}}

 \caption{Application to a  two-dimensional acoustic problem:
    interpretation of $\xi$.
    Behavior of $\widehat{\rm MSE}(I)$ and of 
    the error
$\frac{1}{|\Omega|}\| u^{\rm true} - u_{\xi}^{\star} \|_{L^2(\Omega)}^2 + (\sigma^{\rm re})^2 + (\sigma^{\rm im})^2$ for two different values of the bias $g$, and for two different values of ${\rm SNR}$.
($M=100$, $I=50$, $N=5$, $r_{\rm w}=0.01$, 
$tol=0.3$, empirical update based on inverse multiquadrics).
 }
 \label{fig:PBDW_CV}
\end{figure}

\subsection{A three-dimensional acoustic problem} 
 \subsubsection{Problem definition}
We  consider the three-dimensional  model problem:
\begin{equation}
\label{eq:acoustic_3D_exdom}
\left\{
\begin{array}{ll}
 - (1+ \epsilon i)  \Delta  u_g(\mu)
\,   - (2\pi \mu)^2  u_g (\mu) = g  &  \mbox{in} \; \Omega; \\[3mm]
\partial_n  u_g (\mu)
= 0  &  \mbox{on} \;  \partial   \Omega; \\
\end{array}
\right.
\end{equation}
where $\epsilon=10^{-2}$,   $\Omega  = (-1.5,1.5) \times (0,3) \times (0,3) \setminus  \Omega^{\rm cut}$,
$\Omega^{\rm cut} = (-0.5,0.5) \times (0.25,0.5) \times (0,1)$. 
Figure \ref{fig:acoustic_3D_extracted_dom} shows the geometry. 
In this example, we consider the bk manifold $\mathcal{M}^{\rm bk} = 
\{ u^{\rm bk}(\mu)
=u_{g^{\rm bk}}(\mu): \mu \in \mathcal{P}^{\rm bk} = [0.1,0.5] \}$, and we define the true field as the solution
to \eqref{eq:acoustic_3D_exdom} for some 
$\mu^{\rm true} \in \mathcal{P}^{\rm bk}$
and $g=g^{\rm true}$, where 
$$
g^{\rm bk}(x) = 10 \, e^{ - \| x - p^{\rm bk} \|_2^2};
\quad
g^{\rm true}(x) = 10 \, e^{ - \| x - p^{\rm true} \|_2^2};
$$
and $p^{\rm bk} = [0,2,1]$, $p^{\rm true} = [-0.1,2,1]$.
Parameterized uncertainty in the system models uncertainty in the input frequency $\mu$;
non-parametric or unanticipated uncertainty is here associated with the incorrect location of the acoustic source ($p^{\rm bk} \neq p^{\rm true}$).
Computations are based on a P2 FE discretization with roughly $\mathcal{N} = 16000$ degrees of freedom in $\Omega$. 
Figure \ref{fig:vis3D} shows the bk and true solutions  for two values of $\mu$.

\begin{figure}[h]
\centering
\subfloat[   ] {\includegraphics[width=0.33\textwidth]
 {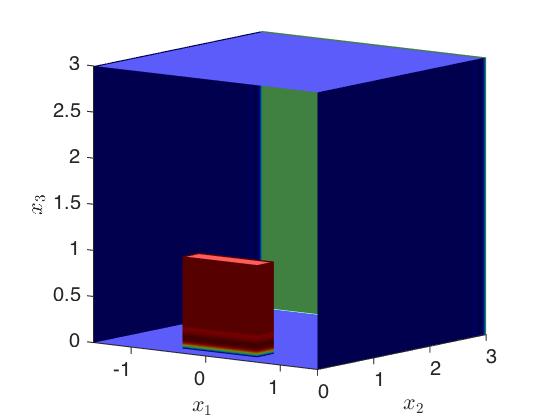}}
~~~~~
\subfloat[] {
\begin{tikzpicture}[x=1.0cm,y=1.0cm]
\linethickness{0.3 mm}
\linethickness{0.3 mm}
\draw  (0,0)--(0.25,0)--(0.25,1)--(0.5,1)--(0.5,0)-- (3,0)-- (3,3)-- (0,3)-- (0, 0);

\coordinate [label={center:  {\large {$\Omega$} }}] (E) at (1.3,2.3) ;
\coordinate [label={center:  {\large {$\Omega^{\rm cut}$} }}] (E) at (1.1,0.8) ;
\coordinate [label={center:  {\large {$x_2$}}}] (E) at (1.4, -0.3) ;
\coordinate [label={center:  {\large {$x_3$}}}] (E) at (3.3, 1.5) ;
\end{tikzpicture}
}
~~~~~
\subfloat[] {
\begin{tikzpicture}[x=1.0cm,y=1.0cm]
\linethickness{0.3 mm}
\linethickness{0.3 mm}
\draw  (-1.5,0)--(1.5,0)--(1.5,3)--(-1.5,3)--(-1.5,0);
\draw[densely dashed]     (-0.5,0)-- (-0.5,1)-- (0.5,1)-- (0.5,0);

\coordinate [label={center:  {\large {$\Omega^{\rm cut}$} }}] (E) at (0,0.4) ;
\coordinate [label={center:  {\large {$\Omega$}}}] (E) at (0, 2.3) ;

\coordinate [label={center:  {\large {$x_1$}}}] (E) at (0, -0.3) ;
\coordinate [label={center:  {\large {$x_3$}}}] (E) at (1.8, 1.5) ;
\end{tikzpicture}
}
\caption{ Application to a three-dimensional acoustic problem: computational domain. }
\label{fig:acoustic_3D_extracted_dom}
\end{figure}
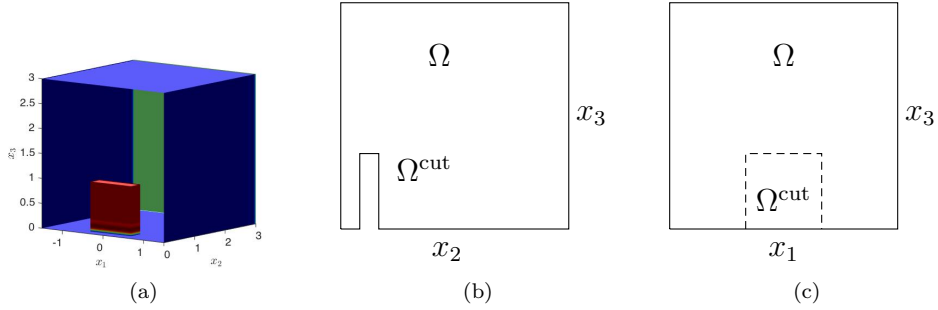

\begin{figure}[h!]
\centering
\subfloat[ ${\rm Re}(u^{\rm bk})$  $\mu=0.1$ ] {\includegraphics[width=0.33\textwidth]
 {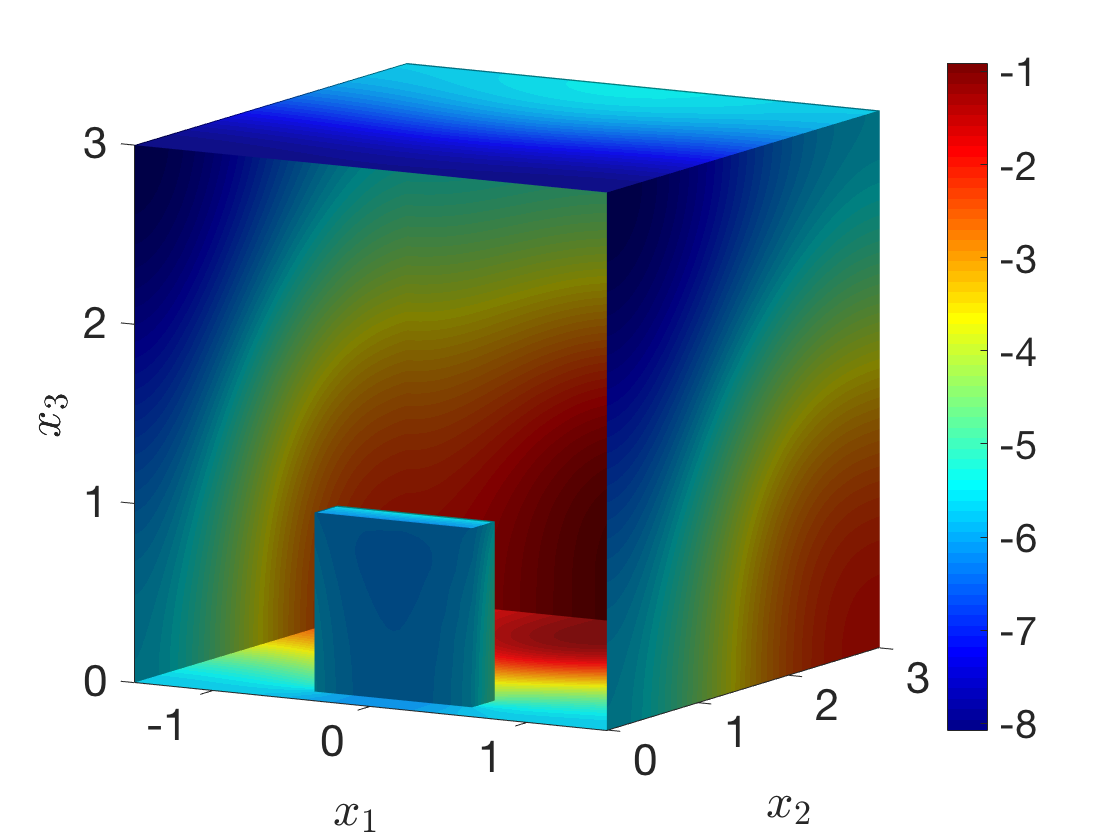}}
 ~
\subfloat[ ${\rm Re}(u^{\rm bk})$  $\mu=0.5$ ] {\includegraphics[width=0.33\textwidth]
 {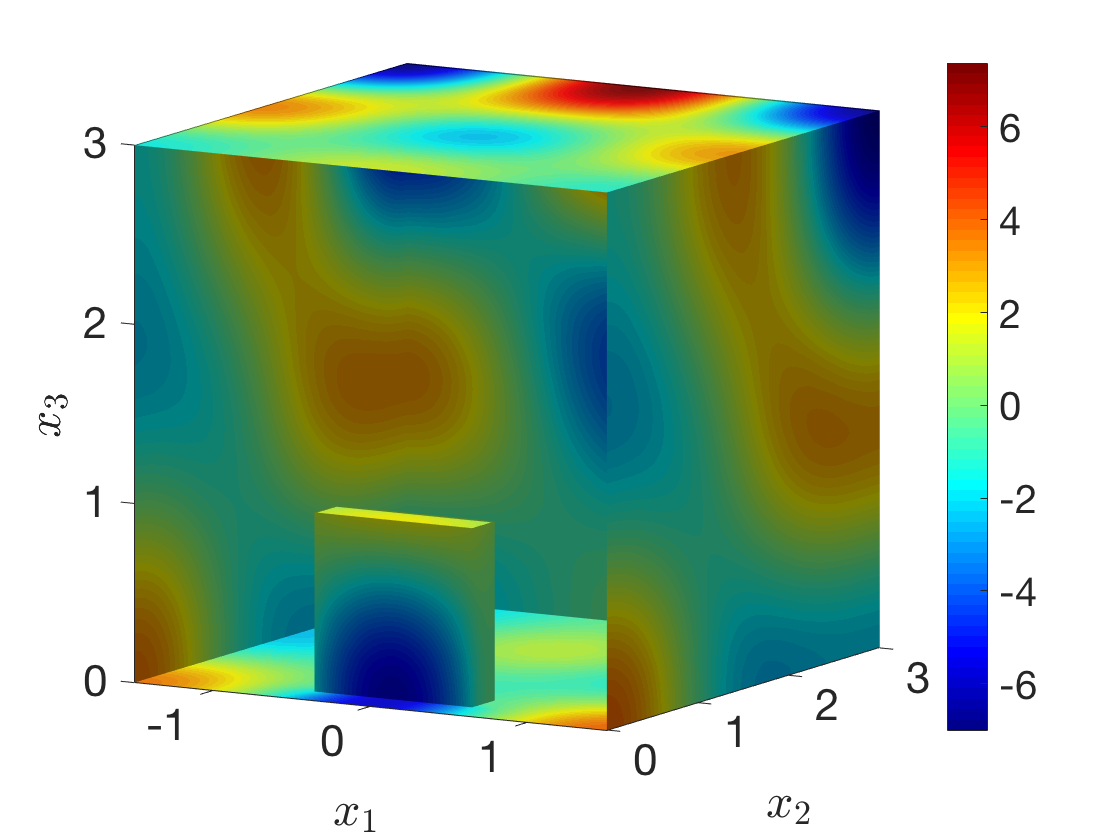}}
 
\subfloat[ ${\rm Re}(u^{\rm true})$  $\mu=0.1$ ] {\includegraphics[width=0.33\textwidth]
 {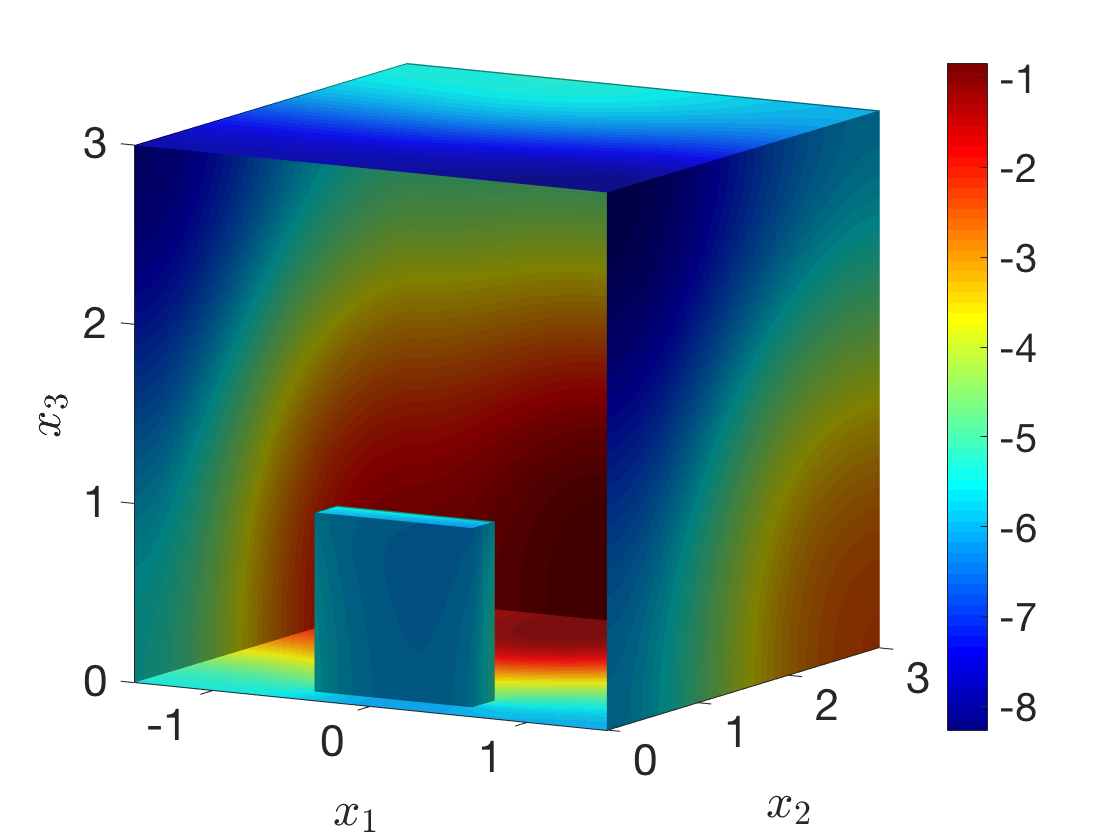}}
 ~
\subfloat[ ${\rm Re}(u^{\rm true})$  $\mu=0.5$ ] {\includegraphics[width=0.33\textwidth]
 {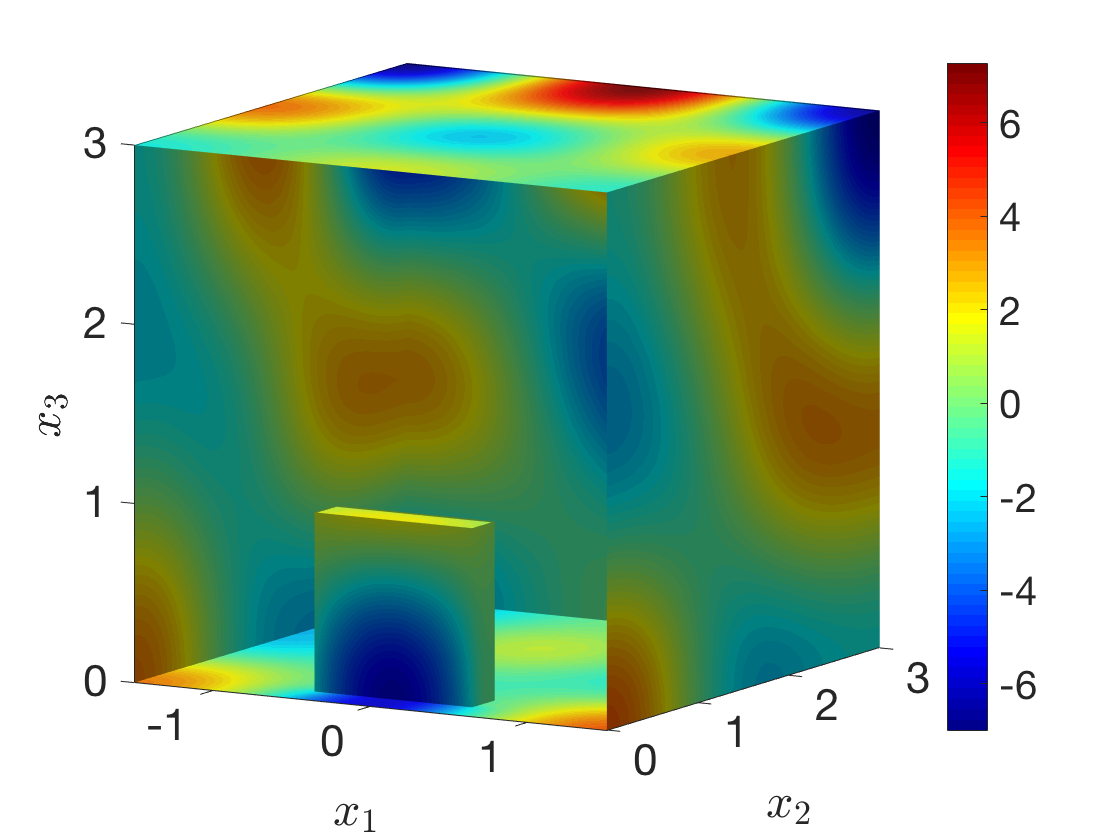}}

    \caption{Application to a  three-dimensional acoustic problem:  visualization of bk and true fields.
  }
 \label{fig:vis3D}
\end{figure}

As in the previous example, we model the synthetic observations by a Gaussian convolution with standard deviation $r_{\rm w}$, see \eqref{eq:exp_observations}.
For simplicity, in the tests below, we only consider perfect measurements.
Furthermore, we measure performance by computing the relative
$L^2$  and $H^1$ errors  $E_{\rm avg}^{\rm rel}$
\eqref{eq:E_avg_APBDW} for $n_{\rm test}=10$  different choices of the parameter $\mu$ in $\mathcal{P}^{\rm bk}$.

\subsubsection{PBDW spaces}

We consider the ambient space $\mathcal{X}=H^1(\Omega)$ endowed with the inner product $(\cdot, \cdot)$ \eqref{eq:inner_product}. 
Furthermore, the background space $\mathcal{Z}_N$ is built using the Weak-Greedy algorithm based on the residual.

As regards the update space, we consider the variational  update associated with the inner product \eqref{eq:inner_product}, and the user-defined updates 
$\mathcal{U}_M = {\rm span} \{  \phi (   \lambda \|  \cdot - x_m^{\rm obs} \|_2   )   \}_{m=1}^M$ for $ \phi(r) = \frac{1}{1+r^2}$ (inverse multiquadrics) and $\lambda=1$.
The observation centers $\{  x_m^{\rm obs} \}_{m=1}^M$ are chosen according to  Algorithm \ref{SGreedy_plus}.

\subsubsection{Numerical results}

In Figure \ref{fig:Sgreedy_3d}, we perform the same test of
Figure \ref{fig:PBDW_Sgreedy}.
We compare the behavior of the inf-sup constant $\beta_{N=30,M}$ with $M$ using the SGreedy procedure, and using randomly-generated centers. For the second choice, we average over $35$ different random choices of the $M$ centers. 
As in the previous example, the Greedy procedure   improves the stability of the PBDW formulation.

\begin{figure}[h!]
\centering
\subfloat[ $H^1$] {\includegraphics[width=0.33\textwidth]
 {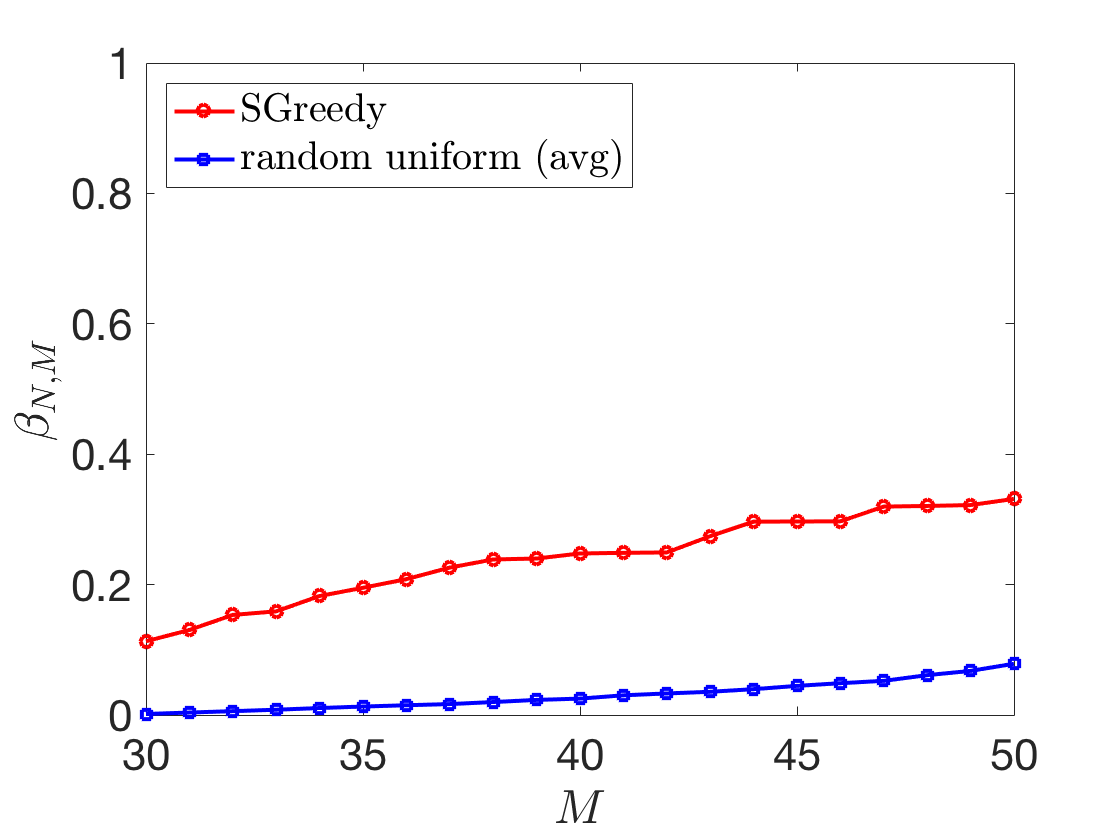}}
 ~
\subfloat[ inverse multiquadrics] {\includegraphics[width=0.33\textwidth]
 {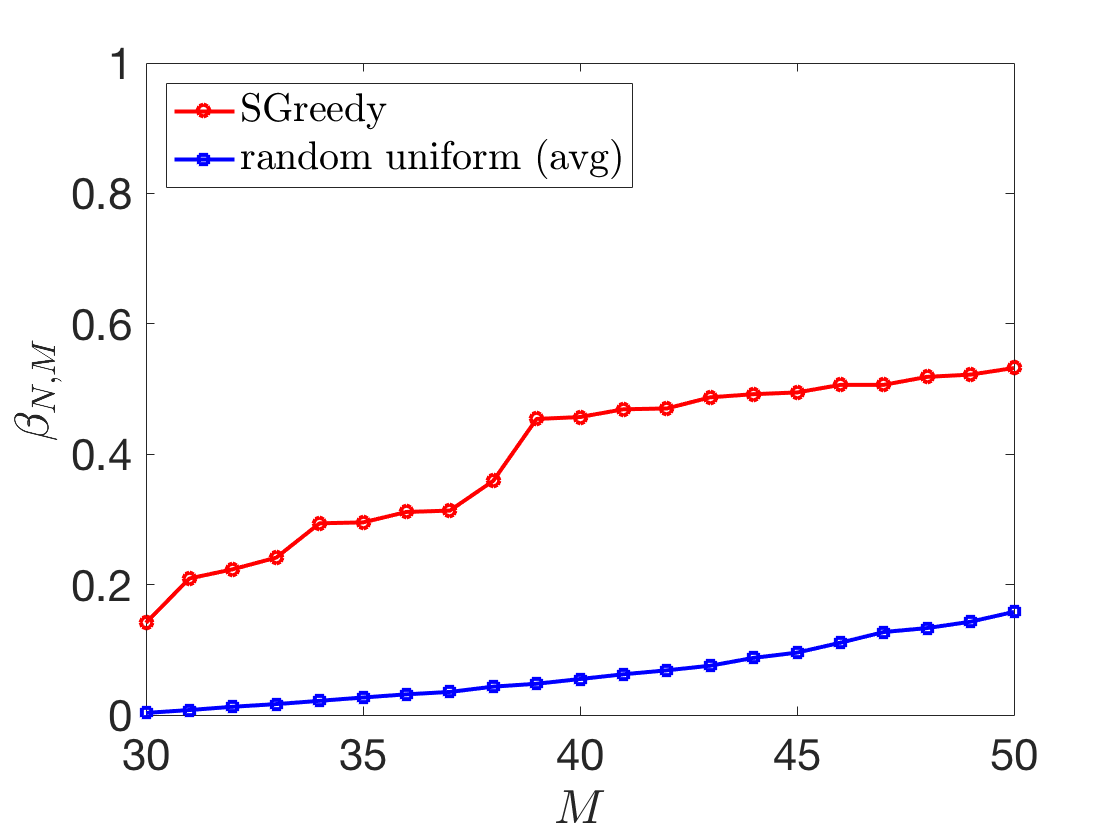}}
  
\caption{  Application to a  three-dimensional acoustic problem:
application of the SGreedy algorithm ($N=30$, $r_{\rm w}=0.02$, $tol=0.4$). 
Results for random centers are averaged over $35$ random trials.}
 \label{fig:Sgreedy_3d}
\end{figure} 

Figure \ref{fig:PBDW_M_conv} shows the behavior of the relative error $E_{\rm avg}^{\rm rel}$
\eqref{eq:E_avg_APBDW} with $M$ for perfect observations, $N=30$, and $r_{\rm w}=0.02$ in \eqref{eq:exp_observations}. As in the previous case, the user-defined update based on inverse multiquadrics leads to more accurate state estimates compared to the standard $H^1$ PBDW. 

\begin{figure}[h!]
\centering
\subfloat[ ] {\includegraphics[width=0.33\textwidth]
 {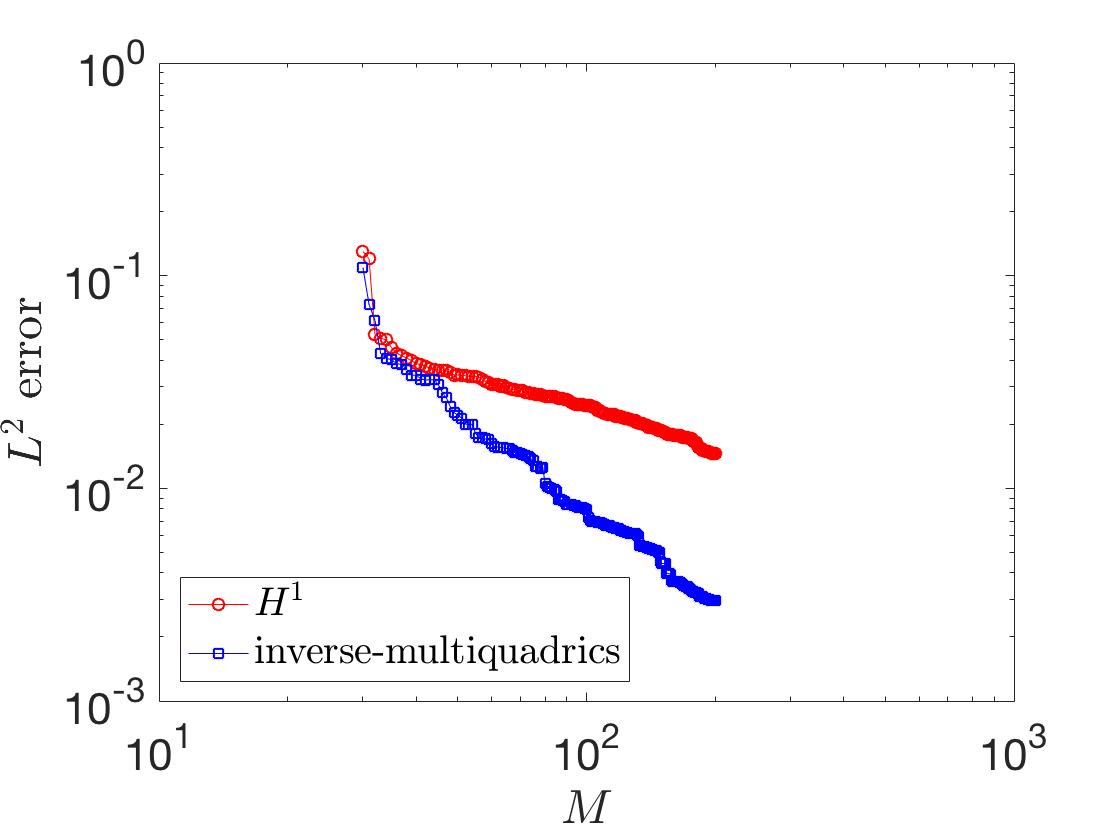}}
 ~
\subfloat[] {\includegraphics[width=0.33\textwidth]
 {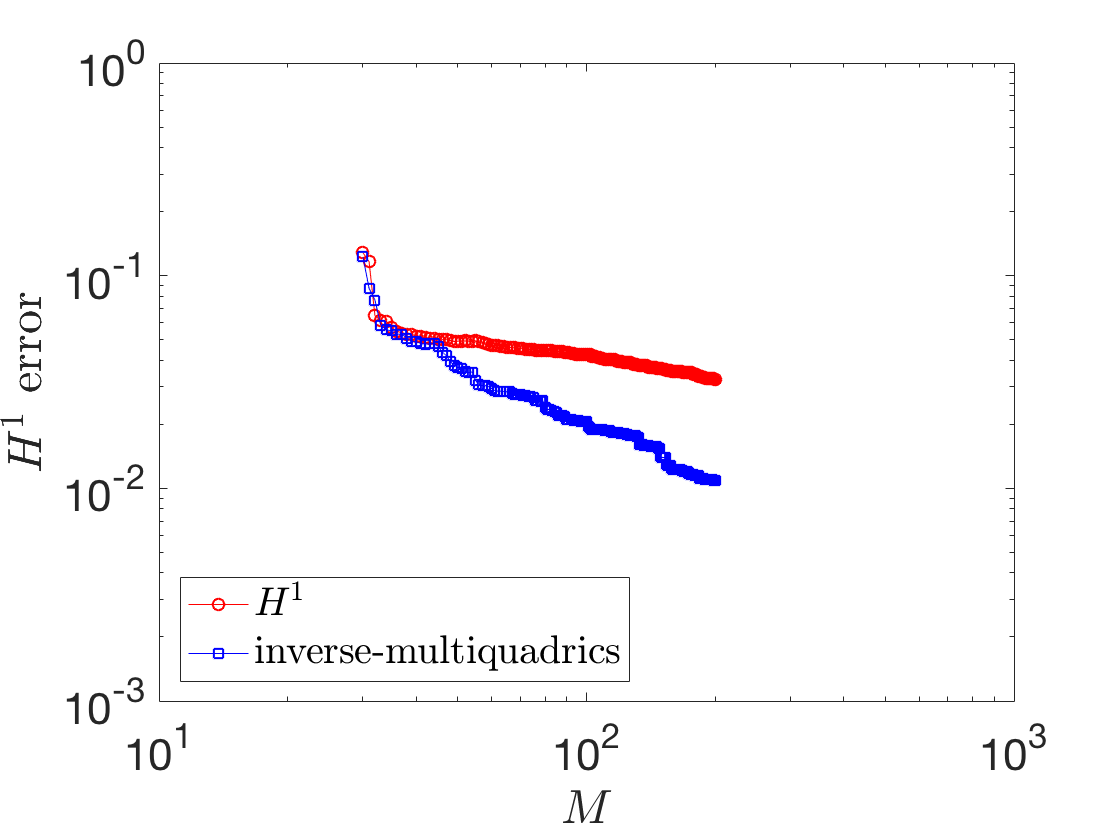}}
  
\caption{  Application to a  three-dimensional acoustic problem:
$M$ convergence fo two different choices of the update space for perfect observations ($N=30$, $r_{\rm w}=0.02$, $tol=0.4$).
}
 \label{fig:Mconv_3d}
\end{figure}

 \section{Application to Fluid Mechanics}
\label{sec:fluids} 
  \subsection{Problem statement}

We consider the following model problem\footnote{
The configuration is similar to the one considered in \cite[section 5]{d2012variational};
however, while here we consider a 2D domain, in \cite{d2012variational} the authors consider a more realistic axi-symmetric model.
}:
\begin{equation}
\label{eq:NS_model_problem}
\left\{
\begin{array}{ll}
\displaystyle{ 
-\,
\frac{1}{\rm Re}
\,
 \nabla \cdot \left(
\nabla u_g + 
\nabla u_g^T
\right)
\,
+
 (u_g \cdot \nabla) u_g + \nabla p_g = 0 } 
 & {\rm in} \, \Omega \\[4mm]
\displaystyle{\nabla \cdot u_g = 0} & {\rm in} \, \Omega \\[3mm]
p_g \mathbf{n} - \frac{1}{\rm Re} (\nabla u_g + \nabla u_g ^T) \mathbf{n}  = 0
& {\rm on} \, \Gamma_{\rm out,1} \cup  \Gamma_{\rm out,2}   \\[3mm]
u_g = g \mathbf{e}_1  & {\rm on} \, \Gamma_{\rm in}   \\[3mm]
u_g = 0    & {\rm on} \, \partial 
\Gamma_{\rm hom}=
\Omega \setminus \left(   \Gamma_{\rm in}  \cup  \Gamma_{\rm out} \right)  \\[3mm]
\end{array}
\right.
\end{equation}
where the domain $\Omega$ is depicted in Figure \ref{fig:domain}(a).
 We then define the bk and true manifolds  as follows:
  $$
  \mathcal{M}^{\rm bk}
  =
  \{
  u_g({\rm Re}):
  {\rm Re} \in [50, 350],
  \; \;
  g(x_2) = 4 (1-x_2) \, x_2
  \};
  $$
 and 
   $$
  \mathcal{M}^{\rm true}
  =
  \{
  u_g({\rm Re}):
  {\rm Re} \in [50, 350],
  \; \;
  g(x_2) = 4 (1-x_2) x_2
  \left(   
  1 + 0.1 \sin(2\pi x_2)
  \right) 
  \}.
  $$
Here,   uncertainty in $\mu = {\rm Re}$ constitutes the anticipated (parametric) uncertainty in the system, while 
  uncertainty in the inflow condition is the unanticipated (non-parametric) uncertainty.
  
  As in the previous examples, we model synthetic observations by  a Gaussian convolution with standard deviation $r_{\rm w}=0.01$. Furthermore, we resort to a conforming Taylor-Hood P3-P2 Finite Element discretization with 
  $\mathcal{N}_{\rm u}=24038$ degrees of freedom.

 \begin{figure}[h!]
 \centering
\subfloat[ ] {\includegraphics[width=0.33\textwidth]
 {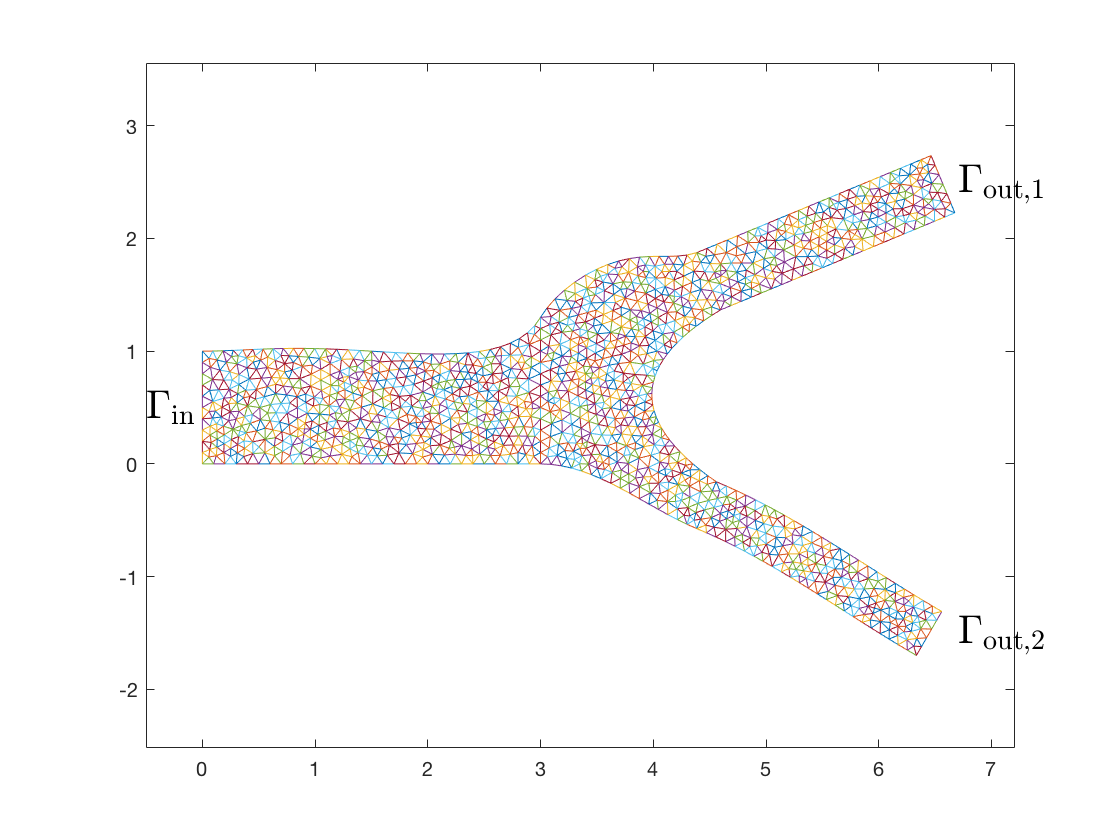}}
 ~
\subfloat[ ] {\includegraphics[width=0.33\textwidth]
 {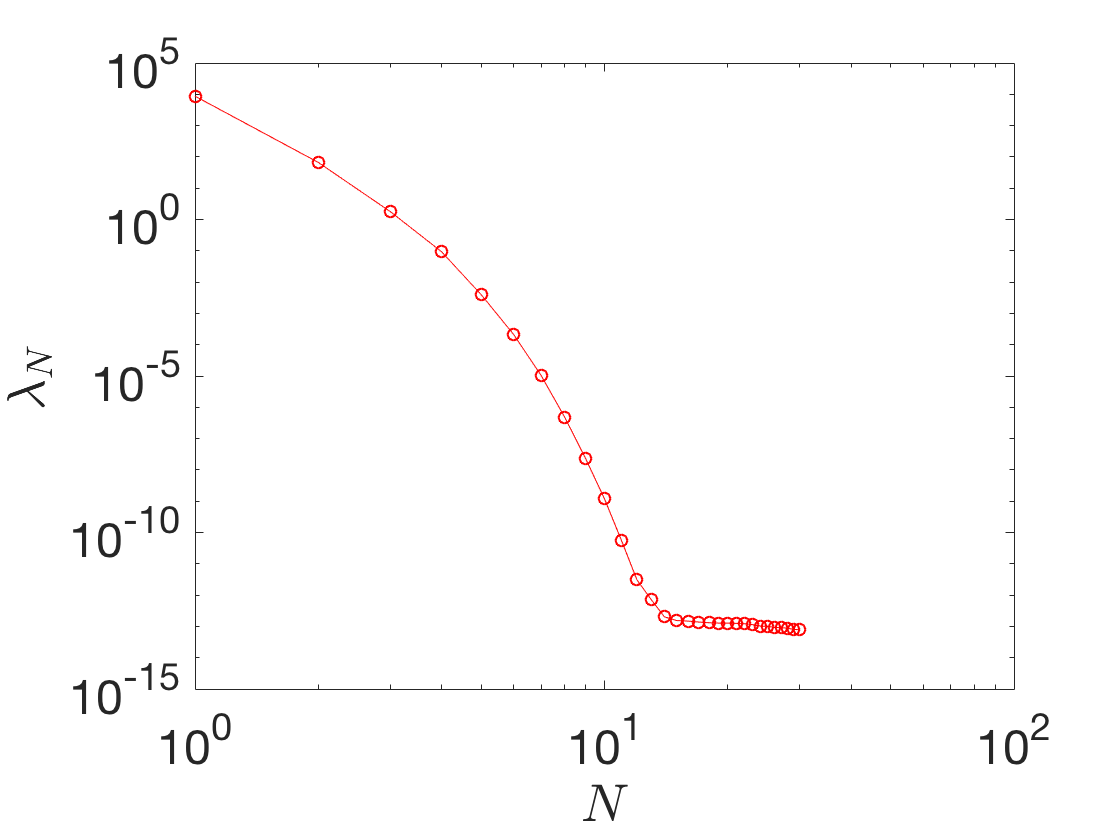}}  

\caption{Application to Fluid Dynamics.  
Figure (a): computational domain;
Figure (b): behavior of the POD eigenvalues ($n_{\rm train}=100$).
}
 \label{fig:domain}
\end{figure}

\subsection{PBDW spaces}

We consider the ambient space $\mathcal{X} = \{ v \in [H_{0,\Gamma_{\rm hom}}(\Omega)]^2: \, \nabla \cdot v =  0  \}$ endowed with the inner product:
$$
(u, v)
= \int_{\Omega} \, \nabla u : \nabla v \, + \, u \cdot v \, dx.
$$
The background space is built using Proper Orthogonal Decomposition based on the $(\cdot, \cdot)$ inner product
(see, e.g., \cite{volkwein2011model}).
Figure \ref{fig:domain}(b) shows the behavior of the POD eigenvalues.
On the other hand, we consider the update generators:
\begin{equation}
\label{eq:update_generator_fluids}
\phi_i(\cdot, x)
=
R_{\mathcal{X}} \, \ell_i(\cdot, x, R_{\rm w}),
\quad
\ell_i(u,x,R_{\rm w})
=
\ell(u_i, x, R_{\rm w}),
\qquad
i=1,2,
\end{equation}
for several values of $R_{\rm w} \geq r_{\rm w}$.
We observe that the update space is divergence-free; furthermore, it satisfies the no-slip boundary conditions on $\Gamma_{\rm hom}$.

\subsection{Numerical results}

In Figure \ref{fig:SGreedy_fluids}, we show the 
performance of the SGreedy procedure for $N=5$ and $R_{\rm w}=0.05$ ($tol=0.3$).
Figure \ref{fig:SGreedy_fluids}(a) shows the behavior of the inf-sup constant $\beta_{N,M}$ with $M$, 
for the user-defined update with $R_{\rm w}=0.05$:
we observe that the SGreedy procedure outperforms on average the random uniform selector.
Figure \ref{fig:SGreedy_fluids}(b)  shows the location of the points selected by the SGreedy (stabilization stage) Algorithm. We observe that  SGreedy selects   points in the boundary layer that develops near the junction:
this can be explained by recalling that the Reynolds number --- the parameter associated with the solution manifold ---
deeply affects the thickness of the boundary layer.

\begin{figure}[h!]
\centering
\subfloat[] {\includegraphics[width=0.33\textwidth]
 {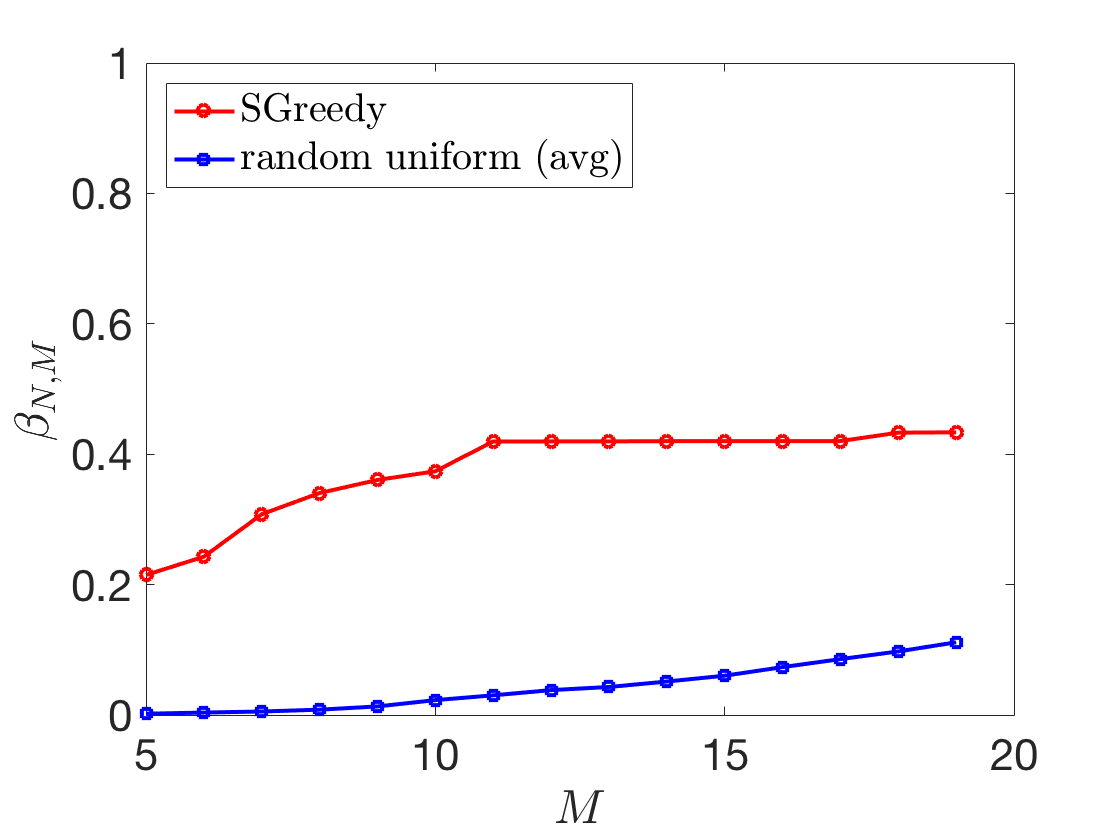}}
 ~
\subfloat[ ] {\includegraphics[width=0.33\textwidth]
 {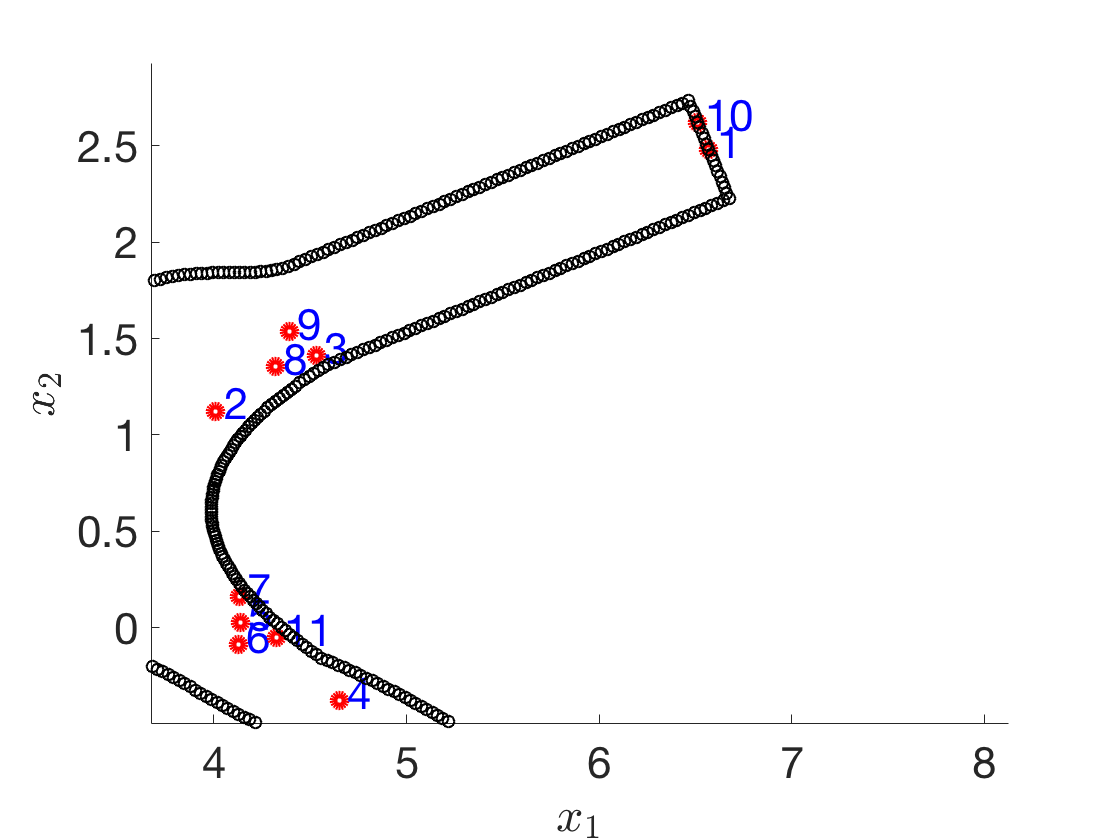}}
  
\caption{  application of the SGreedy + approx algorithm.
Figure (a): behavior of the inf-sup constant $\beta_{N,M}$ with $M$, 
for the user-defined update with $R_{\rm w}=0.05$;
Figure (b): centers $\{ x_m^{\rm obs} \}_m$ selected by the SGreedy algorithm
($N=5$, $r_{\rm w}=0.01$, $tol=0.3$). }
 \label{fig:SGreedy_fluids}
\end{figure} 

Figure \ref{fig:Mconv} shows the behavior of the relative error $E_{\rm avg}^{\rm rel}$ 
\eqref{eq:E_avg_APBDW} with $M$ for perfect measurements, $N=5$, and three values of $R_{\rm width}$ in
\eqref{eq:update_generator_fluids}. We observe that for $R_{\rm width}=0.01$ the user-defined update corresponds to the variational update.
As for the previous test cases, the use of an user-defined update strongly improves performance,  particularly for moderate-to-large values of $M$.
 
    \begin{figure}[h!]
\centering
\subfloat[] {\includegraphics[width=0.33\textwidth]
 {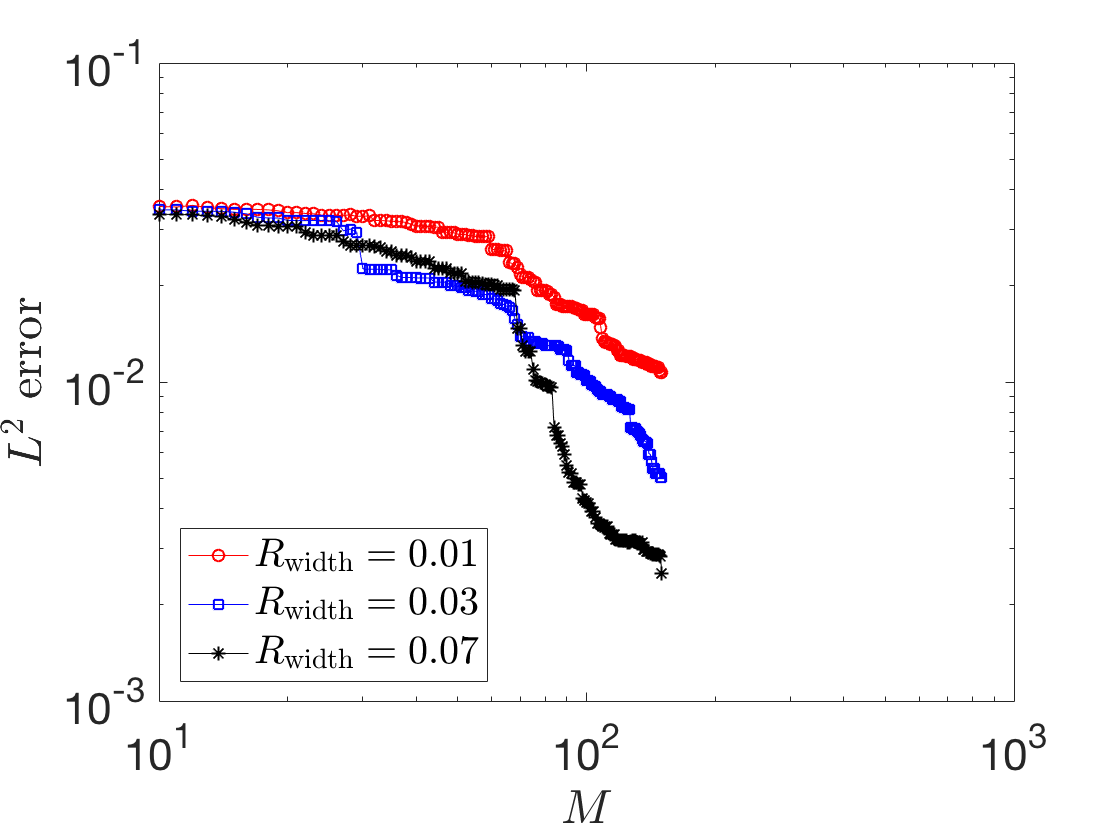}}
 ~
\subfloat[ ] {\includegraphics[width=0.33\textwidth]
 {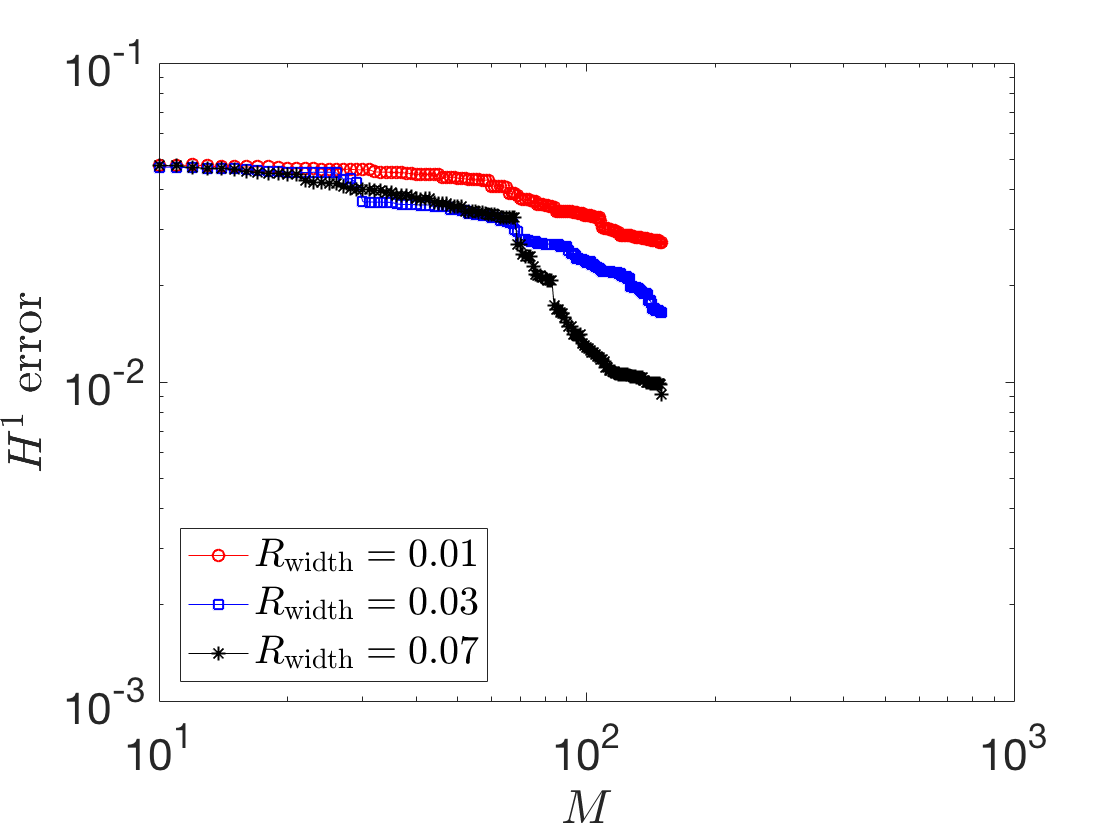}}
  
\caption{  $M$-convergence.
Behavior of the relative averaged error   $E_{\rm avg}^{\rm rel}$
over $n_{\rm test}=10$ different true fields in $\mathcal{M}^{\rm true}$.
($N=5$, $r_{\rm w}=0.01$, $tol=0.3$). }
 \label{fig:Mconv}
\end{figure}

 \section{Conclusions}
 \label{sec:conclusions}
 
In this paper, we presented a number of extensions  to the PBDW formulation for state estimation.
First, we proposed a Tikhonov regularization of the original PBDW statement for general linear functionals, which relies on holdout validation, to systematically deal with noisy measurements.
Second, we proposed user-defined update spaces, which guarantee rapid convergence with respect to the number of measurements $M$ and also might not require the solution to $M$ Riesz problems. 
Third, we presented an \emph{a priori} error analysis that provides insights into the role of the regularization hyper-parameter $\xi$ associated with the penalized formulation.

We identify a number of future research directions to improve the PBDW formulation  and extend its range of applications.
First, we wish to extend the formulation to different models for the experimental noise, and also to probabilistic background.
Towards this end, we wish to rely on the connection between PBDW and PSM established in \cite{taddei_APBDW}, and recapped in section \ref{sec:pb_statement}.
Second, we wish to extend the formulation to time-dependent problems.
This might be accomplished by exploiting a space-time variational formulation (\cite{urban2014improved})
to incorporate   the space-time structure of the evolution equations in the empirical
expansion.


\bibliographystyle{unsrt}
\bibliography{bib_folder/bib_MOR_RB,bib_folder/bib_MOR,bib_folder/bib_RBF,bib_folder/bib_data_assimilation,bib_folder/bib_statistics,bib_folder/bib_general_theory,bib_folder/bib_MOR_scrbe,bib_folder/bib_validation,bib_folder/bib_my_publications,bib_folder/bib_POD}


\end{document}